\documentclass[reqno, 12pt]{amsart}

\usepackage{amsmath}
\usepackage{amsfonts}
\usepackage{amssymb}
\usepackage{amsthm}
\usepackage{bbm}
\usepackage{graphicx, color}
\usepackage{tikz}
\usetikzlibrary{calc,decorations.markings}
\usetikzlibrary{arrows.meta}
\usetikzlibrary{positioning}
\usetikzlibrary{cd, intersections, calc, decorations.pathmorphing, arrows, decorations.pathreplacing}

\usepackage{slashbox}
\allowdisplaybreaks

\definecolor{mygreen}{rgb}{0,0.7,0.3}

\usepackage{subfiles}
\usepackage[colorlinks, linkcolor=blue, citecolor=red, urlcolor=cyan, pagebackref]{hyperref}

\numberwithin{equation}{section}

\usepackage{cleveref}
\crefname{thm}{Theorem}{Theorems}
\crefname{cor}{Corollary}{Corollaries}
\crefname{lem}{Lemma}{Lemmas}
\crefname{sublem}{Sublemma}{Sublemmas}
\crefname{prop}{Proposition}{Propositions}
\crefname{dfn}{Definition}{Definitions}
\crefname{defi}{Definition}{Definitions}
\crefname{ex}{Example}{Examples}
\crefname{claim}{Claim}{Claims}
\crefname{conj}{Conjecture}{Conjectures}
\crefname{conv}{Convention}{Conventions}
\crefname{rem}{Remark}{Remarks}
\crefname{rmk}{Remark}{Remarks}
\crefname{prob}{Problem}{Problems}
\crefname{quest}{Question}{Questions}
\crefname{figure}{Figure}{Figures}
\crefname{table}{Table}{Tables}
\crefname{section}{Section}{Sections}
\crefname{subsection}{Section}{Sections}
\crefname{appendix}{Appendix}{Appendices}

\newtheorem{thm}{Theorem}[section]
\newtheorem{prop}[thm]{Proposition}
\newtheorem{cor}[thm]{Corollary}
\newtheorem{lem}[thm]{Lemma}

\theoremstyle{definition}
\newtheorem{dfn}[thm]{Definition}
\newtheorem{defi}[thm]{Definition}
\newtheorem{ex}[thm]{Example}

\theoremstyle{remark}
\newtheorem{rmk}[thm]{Remark}

\newtheorem{quest}[thm]{Question}

\newcommand*{\chom}{\mathcal{H}\kern -.5pt om}

\newcommand{\bZ}{\mathbb{Z}}
\newcommand{\bQ}{\mathbb{Q}}
\newcommand{\bR}{\mathbb{R}}

\newcommand{\bT}{\mathbb{T}}

\newcommand{\bP}{\mathbb{P}}
\newcommand{\bG}{\mathbb{G}}
\newcommand{\bA}{\mathbb{A}}
\newcommand{\bE}{\mathbb{E}}

\newcommand{\bs}{{\mathsf{s}}}

\newcommand{\indi}{i}

\newcommand{\indk}{k}

\newcommand{\cA}{\mathcal{A}}

\newcommand{\cC}{\mathcal{C}}

\newcommand{\cE}{\mathcal{E}}
\newcommand{\cF}{\mathcal{F}}

\newcommand{\cK}{\mathcal{K}}

\newcommand{\cS}{\mathcal{S}}
\newcommand{\cT}{\mathcal{T}}
\newcommand{\cU}{\mathcal{U}}
\newcommand{\cV}{\mathcal{V}}

\newcommand{\X}{\mathcal{X}}
\newcommand{\cX}{\mathcal{X}}

\newcommand{\fS}{\mathfrak{S}}
\newcommand{\fsl}{\mathfrak{sl}}
\newcommand{\frakg}{\mathfrak{g}}

\newcommand{\bep}{\boldsymbol{\epsilon}}
\newcommand{\fol}{\cF}
\newcommand{\sfa}{\mathsf{a}}
\newcommand{\pot}{\mathsf{w}}
\newcommand{\bft}{\mathbf{t}}

\newcommand{\Hom}{\mathrm{Hom}}

\newcommand{\tri}{\triangle}
\newcommand{\sgn}{\mathrm{sgn}}
\newcommand{\trop}{\mathrm{trop}}
\newcommand{\stab}{\mathrm{stab}}

\newcommand{\MF}{\mathcal{MF}}

\newcommand{\eML}{\widehat{\mathcal{ML}}}
\newcommand{\dMF}{\widetilde{\mathcal{MF}}}

\newcommand{\Tri}{\mathrm{Tri}}
\newcommand{\bTri}{\bT \mathrm{ri}}
\newcommand{\tr}{\mathsf{T}}

\newcommand{\bExch}{\bE \mathrm{xch}}
\newcommand{\uf}{\mathrm{uf}}

\newcommand{\Teich}{Teichm\"uller}

\newcommand{\TT}{\mathrm{TT}}

\newcommand{\sD}{\mathsf{D}}
\newcommand{\Dfd}{\mathsf{D}_{\mathsf{fd}}}
\newcommand{\per}{\mathsf{per}}

\DeclareMathOperator{\im}{\mathrm{im}}
\DeclareMathOperator{\coker}{\mathrm{coker}}

\DeclareMathOperator{\relint}{\mathrm{relint}}
\DeclareMathOperator{\interior}{\mathrm{int}}
\DeclareMathOperator{\Spec}{\mathrm{Spec}}

\renewcommand{\mathbf}{\boldsymbol}

\makeatletter
\newcommand{\oset}[3][0ex]{%
  \mathrel{\mathop{#3}\limits^{
    \vbox to#1{\kern-2\ex@
    \hbox{$\scriptstyle#2$}\vss}}}}
\makeatother
\newcommand{\overbar}[1]{\oset{#1}{-\!\!\!-\!\!\!-}}

\makeatletter
\newcommand{\osetnear}[3][0ex]{%
  \mathrel{\mathop{#3}\limits^{
    \vbox to#1{\kern-.3\ex@
    \hbox{$\scriptstyle#2$}\vss}}}}
\makeatother
\newcommand{\overbarnear}[1]{\osetnear{#1}{-\!\!\!-\!\!\!-}}

\makeatletter
\newcommand{\osetfar}[3][0ex]{%
  \mathrel{\mathop{#3}\limits^{
    \vbox to#1{\kern-5\ex@
    \hbox{$\scriptstyle#2$}\vss}}}}
\makeatother

\tikzset{
    squigarrow/.style={-{Classical TikZ Rightarrow[length=4pt]}, decorate, decoration={snake, amplitude=1.8pt, pre length=2pt, post length=3pt}}
}

\title{Train track combinatorics and cluster algebras}

\author[Shunsuke Kano]{Shunsuke Kano}
\address{Shunsuke Kano, Research Alliance Center for Mathematical Sciences, Tohoku University,
6-3 Aoba, Aramaki, Aoba-ku, Sendai, Miyagi 980-8578, Japan.}
\email{s.kano@tohoku.ac.jp}

\date{\today}

\begin{document}
\maketitle

\begin{abstract}
The concepts of train track was introduced by W. P. Thurston to study the measured foliations/laminations and the pseudo-Anosov mapping classes on a surface.
In this paper, we translate some concepts of train tracks into the language of cluster algebras using the Goncharov--Shen's potential function \cite{GS15}.
Through this translation, we prove the sign stability \cite{IK21} of the general pseudo-Anosov mapping classes.
\end{abstract}

\tableofcontents
\section{Introduction}\label{sec:intro}
\emph{Train track} was introduced by W. P. Thurston \cite{Th} to study the \emph{measured foliations}\footnote{In \cite{Th}, he uses measured laminations instead for measured foliations. However, these concepts are closely related. For instance, it is well-known that the space of measured foliations and the space of measured laminations are isomorphic as a piecewise-linear manifold.} and \emph{pseudo-Anosov mapping classes}.
A train track is a graph embedded smoothly in a surface so that each vertex looks like a switch of a railway of trains. (See \cref{fig:switch_cond}.)
For a train track, we often assign a $\bR_{\geq 0}$-valued \emph{weight} (called \emph{measure} after this section) at each edge satisfying \emph{switch condition} \eqref{eq:switch_cond}.
We can think of a pair of a train track and a weight as a combinatorial model of a measured foliation.
(See \cref{subsec:traintrack} for the concrete definition.)
More precisely, it is proven in \cite{Th} that the space of weights of the ``recurrent'' and ``maximal'' (called \emph{complete}) train tracks gives an atlas of the space of the measured foliations.
On the other hand, for each pseudo-Anosov mapping class $\phi$, there is a train track, called the \emph{invariant track}, which behaves nicely about the action of $\phi$.
Many properties of pseudo-Anosov mapping classes are obtained from invariant tracks.
For instance, the topological entropy of a pseudo-Anosov mapping class is given by the logarithm of the spectral radius of the \emph{transition matrix} associated to the invariant track.

The \emph{cluster algebra}, invented by S. Fomin and A. Zelevinsky \cite{FZ-CA1}, is a combinatorial structure that appears in so many branches of mathematics and mathematical physics, such as discrete integrable systems, Painlev\'e equations, representation of associative algebras, Donaldson--Thomas invariants, quantum field theories.
The theory of (higher) \Teich\ spaces is one of them \cite{FG06,FG07}.
V. Fock and A. Goncharov introduce a pair $(\cA, \cX)$ of schemes with certain ``positivity", called a \emph{cluster ensemble}, coming with an action of a discrete symmetry group $\Gamma$, called a \emph{cluster modular group} \cite{FG09}.
The schemes are related by a $\Gamma$-equivariant morphism $p: \cA \to \cX$, called the \emph{ensemble map}.
By the positivity, we can consider the set of the \emph{semifield $\bP$-valued points} $(\cA(\bP), \cX(\bP))$ of a cluster ensemble.
A cluster ensemble is defined for the combinatorial datum, called a \emph{mutation class} of seeds.
There is a mutation class obtained from the \emph{ideal triangulations} of a punctured surface $\Sigma$, so we denote by $(\cA_\Sigma, \cX_\Sigma)$ the corresponding cluster ensemble.
It is known that the sets of $\bR_{>0}$-valued points of them correspond to the certain generalizations of the \Teich\ space $\cT(\Sigma)$:
\begin{align*}
    \cA_\Sigma(\bR_{>0}) \cong \widetilde{\cT}(\Sigma), \qquad
    \cX_\Sigma(\bR_{>0}) \cong \widehat{\cT}(\Sigma).
\end{align*}
Here, $\widetilde{\cT}(\Sigma)$ and $\widehat{\cT}(\Sigma)$ are called the \emph{decorated \Teich\ space} and the \emph{enhanced \Teich\ space}, respectively.
(See \cite[Section 2]{Pen} for more details. In this book, the latter is called the \emph{holed \Teich\ space}.)
The usual \Teich\ space $\cT(\Sigma)$ corresponds to the image of the ensemble map: $p(\cA_\Sigma(\bR_{>0})) \cong \cT(\Sigma)$.
It is also known that the sets of the tropical semifield $\bR^\trop = (\bR, \min, +)$-valued points of them correspond to certain generalizations of the space of measured foliations:
\begin{align*}
    \cA_\Sigma(\bR^\trop) \cong \dMF(\Sigma), \qquad
    \cX_\Sigma(\bR^\trop) \cong \widehat{\MF}(\Sigma).
\end{align*}
The elements of the former $\dMF(\Sigma)$ (resp. the latter $\widehat{\MF}(\Sigma)$) are allowed that there are the ``peripheral leaves'' (resp. the leaves which are stuck into the punctures).
(See \cref{sec:dMF} for the former $\dMF(\Sigma)$ and \cite{GW17} for the latter.)
The space $\MF(\Sigma)$ of ordinal measured foliations corresponds to the image of the ensemble map: $p(\cA_\Sigma(\bR^\trop)) \cong \MF(\Sigma)$.
Moreover, the mapping class group $MC(\Sigma)$ is regarded as a subgroup of the cluster modular group $\Gamma_\Sigma$ of finite index.

The purpose of this paper is twofold:
one is to translate the terminologies of the theory of train tracks into the theory of cluster algebras (\cref{part:TT_CA}), and the other is to compare the sequence of splittings of the invariant train track and the sign of a representation path of a pseudo-Anosov mapping class (\cref{part:SS}).

\subsection{Train tracks in terms of cluster algebras (\cref{part:TT_CA})}
The space $\MF(\Sigma)$ of the measured foliations is a piecewise-linear manifold.
One of the piecewise-linear structures (atlases) is given by complete train tracks as mentioned above.
It means that each complete train track gives a local chart.

On the other hand, A. Goncharov and L. Shen introduced a potential function on $\cA$ in \cite{GS15}.
We call it \emph{Goncharov--Shen potential function}.
We show that the maximal domains of linearity of the tropicalized Goncharov--Shen potential function\footnote{The tropicalization of the Goncharov--Shen potential (\emph{i.e.}, the induced function on $\cA(\bR^\trop)$) is a piecewise-linear map.} on $\cA_\Sigma(\bR^\trop)$ correspond to the local chars of $\MF(\Sigma)$ associated to the complete train tracks.

\begin{thm}[{\cref{cor:TT_D,thm:atlases}}]\label{introthm:atlas}
Let $\Sigma$ be a punctured surface.
Then, there is a bijection from the set of complete train tracks on $\Sigma$ to the set of maximal domains of linearity of the tropicalized Goncharov--Shen potential function on $\cA_\Sigma(\bR^\trop)$.
Moreover, the piecewise-linear structure given by train tracks and the one given by the cluster structure on $\MF(\Sigma) \setminus \{\emptyset\}$ are equivalent.
\end{thm}

\begin{rmk}
We can easily extend the piecewise-linear structure given by train tracks on $\MF(\Sigma)$ to $\dMF(\Sigma)$ and restate \cref{introthm:atlas} for it.
With the identification $\dMF(\Sigma) \cong \cA_\Sigma(\bR^\trop)$ in mind, it is a natural question whether there is the piecewise-linear structure given by train tracks on $\cX_\Sigma(\bR^\trop)$.
Train tracks for $\cX_\Sigma(\bR^\trop) (\cong \eML(\Sigma))$ are introduced in \cite[Appendix A]{IK20b} but these are suited to \emph{pants decompositions} as opposed to ideal triangulations.
\end{rmk}

The most basic operation in the theory of cluster algebras is the \emph{mutation}.
It is well-known the \emph{flip} for an ideal triangulation corresponds to the mutation for the corresponding dual quiver.
Also, there are famous elementary moves for the train tracks, called splitting, folding and shifting (see \cref{fig:split,fig:shift}).
We build a relationship between them:
\begin{thm}[\cref{thm:birel_TT}]\label{introthm:birel_TT}
Let $\Sigma$ be a punctured surface and let $\tri$ and $\tri'$ be ideal triangulations of $\Sigma$ which are related by the flip along an ideal arc.
Then, there is a binary relation between the set of complete train tracks which are suited to $\tri$ and the set of those that are suited to $\tri'$.
\end{thm}

\subsection{Train track splittings and tropical cluster transformations (\cref{part:SS})}

T. Ishibashi and the author introduce the notion of \emph{sign stability} of the mutation loops as a cluster algebraic analogy of pseudo-Anosovness.
In \cite{IK20a}, we compare the sign stability and the pseudo-Anosovness of the mapping classes.
As a consequence, we find that the \emph{uniform} sign stability is equivalent to the \emph{generic} pseudo-Anosovness.
The former is the strongest sign stability in some sense, and the latter is not so special class of pseudo-Anosov mapping classes.

The sign stability is some asymptotic property of the sequence of \emph{signs} of a mutation loop.
We require that a mutation loop is sign-stable if the sign stabilizes to a sequence of ``strict'' signs, but the sequence of signs of a general pseudo-Anosov mapping class might contain zeroes.

As mentioned above, a pseudo-Anosov mapping class has a nice train track, called \emph{invariant track}.
It is known that the action of the pseudo-Anosov mapping class is represented as a specific sequence of splittings or shiftings of the invariant track.
Using the detail of the binary relation (\cref{introthm:birel_TT}) of the train tracks arising from a flip, we can translate this sequence of splittings or shiftings of the invariant track into the sequence of ``strict'' signs of the corresponding mutation loop.

\begin{thm}[\cref{thm:SS_pA}]
Let $\phi$ be a (general) pseudo-Anosov mapping class of a punctured surface $\Sigma$.
Then, $\phi^r$ is sign-stable for some $r \geq 1$.
\end{thm}

By this theorem, we can calculate some entropies related to the pseudo-Anosov mapping classes (\cref{cor:entropy}).

\subsection{Future directions and problems}
Here, we give some directions for the continuation or application of this work.

\subsubsection{Completion of the $g$-vector fan}
On the tropicalization $\cX(\bR^\trop)$ of the variety $\cX$, there is some fans, for instance the $g$-vector fan \cite{GHKK, Rea14} (the Fock--Goncharov fan \cite{FG09}), the mutation fan \cite{Rea14} and the scattering fan \cite{Rea20}.
The $g$-vector fan is complete if and only if the cluster structure (\emph{i.e.,} the mutation class) is of finite.
T. Yurikusa studies the class of cluster algebras whose support of the $g$-vector fan is dense \cite{Yur20,Yur21}.
In particular, the cluster algebra of almost every marked surface is a member of this class \cite[Theorem 1.2]{Yur20}.
On the other hand, the scattering fan is complete \cite[Theorem 3.1]{Rea20} and the $g$-vector fan is a subfan of it.
The structure of the outside of the $g$-vector fan is very complicated, and it is called \emph{badlands} \cite{Nak21}.

The train tracks suited to an ideal triangulation give a complete fan on $\MF(\Sigma) \cong p(\cA_\Sigma(\bR^\trop))$ (\cref{cor:TT_fan}).
It is important that the subset $p(\cA_\Sigma(\bR^\trop)) \subset \cX_\Sigma(\bR^\trop)$ is located outside of the $g$-vector fan, so the following question arise:

\begin{quest}\label{quest:fans}
Are there some relationships between the fan given by train tracks suited to an ideal triangulation and the scattering fan on the outside of the $g$-vector fan?
\end{quest}

By using the sign of a representation path $\gamma$ (\cref{subsec:cluster_mod}) of a mutation loop $\phi$, we can define the \emph{sign fan} as the cones are parametrized by the possible signs of $\gamma$.
It is clear that the mutation fan is a refinement of the sign fan of any representation path of any mutation loop.
Also, the $g$-vector fan is a subfan of the mutation fan \cite[Conjecture 8.1]{Rea14}\footnote{It is a conjecture in \cite[Conjecture 8.1]{Rea14} but now we can prove it by \emph{sign-coherence conjecture} \cite{GHKK}}.
Moreover, for some pseudo-Anosov mapping classes, the following holds:
$\cC^\stab_\gamma \cap p(\cA_\Sigma(\bR^\trop)) = p(\cV(\tau^+_\phi))$.
Here, $\cC^\stab_\gamma$ is a cone of some sign fan, $\cV(\tau^+_\phi)$ is a cone of the ``weights'' of the invariant track $\tau^+_\phi$ of the mapping class $\phi$ (\cref{subsec:bdd_SS}).
It suggests the existence of the relationship of \cref{quest:fans}.

Also, there is a concept of $g$-vector fans for 2-term silting complexes over a finite-dimensional algebra \cite{AIR14}.
T. Yurikusa, P.-G. Plamondon and T. Aoki also study the class of the algebras whose $g$-vector fan is dense \cite{PY23,AY20}.
We hope that some techniques in the paper are useful to study the outside of the $g$-vector fan of these algebras.

\subsubsection{Higher train tracks}
There is a cluster ensemble $(\cA_{\frakg, \Sigma}, \cX_{\frakg, \Sigma})$ for a pair $(\frakg, \Sigma)$ of a semisimple Lie algebra $\frakg$ and a marked surface $\Sigma$.
The case $\frakg = \fsl_2$ is our case: $(\cA_{\fsl_2, \Sigma}, \cX_{\fsl_2, \Sigma}) = (\cA_{\Sigma}, \cX_{\Sigma})$.
In particular, the Goncharov--Shen potential function is defined for general semisimple Lie algebras $\frakg$ \cite{GS15}.

On the other hand, the geometric model of higher measured laminations is suggested in the case $\frakg = \fsl_3$ in \cite{DS20} and \cite{IK22} but these are only rational points of $\cA_{\fsl_3, \Sigma}(\bR^\trop)$ and $\cX_{\fsl_3, \Sigma}(\bR^\trop)$, respectively.

The train tracks give the geometric model of not only ``rational'' but also ``real'' measured laminations (foliations).
We hope that there is a higher analog of the train tracks which gives the geometric model of the ``real'' higher measured laminations.
Most of the discussions in \cref{part:TT_CA} are purely cluster algebraic, so if we can build the higher train track as some type of a graph, then it is expected that the claim similar to \cref{introthm:atlas} holds for them.
Also, the $g$-vector fan of the cluster structure of $(\frakg \neq \fsl_2, \Sigma)$ is not dense in general.
We hope that we can analyze the outside of the $g$-vector fan of $(\frakg \neq \fsl_2, \Sigma)$ by using the higher train tracks when the above aim is reached.

\subsection*{Organization of the paper}
In \cref{sec:dMF}, we review the theory of the space of the decorated measured foliations according to \cite{PP93}.
Here, we give the map corresponding to the tropicalized Goncharov--Shen potential function, geometrically.
In \cref{sec:TT}, we review the concept of train tracks and introduce the train tracks which suited to an ideal triangulation.
The main results of \cref{part:TT_CA} are essentially given here.
In \cref{sec:GSpot_Vvar}, we translate the results in \cref{sec:TT} via the tropicalized Goncharov--Shen potential function.
The concepts of cluster algebras is reviewed in \cref{sec:appCA}.

In \cref{sec:sign-stab}, we review the notion of sign stability of mutation loops and give the geometric meaning of the signs.
In \cref{sec:TTsplit}, we prove the sign stability of a general pseudo-Anosov mapping class by using the specific sequence of splittings and shiftings of the invariant track.
As a consequence, we compute some entropies of the pseudo-Anosov mapping classes.

\subsection*{Acknowledgements}
The author thanks Tsukasa Ishibashi for long-term discussion about this project.
He is also grateful to Yuji Terashima for telling him about the train track splittings and grateful to Osamu Iyama for sharing the problem about the $g$-vector fan of gentle algebras.
The author is partially supported by scientific research support of Research Alliance Center for Mathematical Sciences, Tohoku University.

\part{Train tracks in terms of cluster algebras}\label{part:TT_CA}
In this first half part, we translate some notions of the theory of train tracks into the language of the cluster algebras via the tropicalized Goncharov--Shen potential.
\section{Decorated measured foliations}\label{sec:dMF}
In this section, we recall the concepts around measured foliations.
A punctured surface $\Sigma = (\Sigma, P)$ is a connected oriented punctured surface $\Sigma$ with genus $g$, $h$ punctures $P = \{ p_1, p_2, \dots, p_h \}$ and no boundary components, such that
\begin{itemize}
\item[(S1)] $2g - 2 + h >0$,
\item[(S2)] when $g=0$, $h>3$.
\end{itemize}
The condition (S1) ensures the existence of an ideal triangulation and (S2) exclude the surface such that the quiver corresponding to any ideal triangulation of it has no arrows.

\subsection{The space of decorated measured foliations}
A \emph{foliatoin} on a punctured surface $\Sigma = (\Sigma, P)$ is a foliation $F$ on $\Sigma$ with isolated prong singularities such that the set of 1-pronged singularities of $F$ is contained in $P$.
We call the leaf of which a singular point is an endpoint \emph{singular leaf}.
\begin{figure}[h]
    \centering
    \begin{tikzpicture}
    \node  [fill, circle, red, inner sep = 1.5] (v1) at (0,0) {};
    \node [red] at (0.2,-0.23) {$x$};
    \draw (v1) -- (1.7,0);
    \draw (1.7,0.5) .. controls (0,0.5) and (-0.5,0.4) .. (-0.5,0) .. controls (-0.5,-0.4) and (0,-0.5) .. (1.7,-0.5);
    \draw (1.7,1) .. controls (0,1) and (-1,0.9) .. (-1,0) .. controls (-1,-0.9) and (0,-1) .. (1.7,-1);
    \draw (1.7,1.5) .. controls (0,1.5) and (-1.5,1.4) .. (-1.5,0) .. controls (-1.5,-1.4) and (0,-1.5) .. (1.7,-1.5);
    \end{tikzpicture}
    \qquad \quad
    \begin{tikzpicture}
    \node  [fill, circle, red, inner sep = 1.5] (v1) at (0,0) {};
    \node [red] at (0.2,-0.3) {$x$};
    \draw (-1.5,1.7) -- (v1) -- (1.7,0);
    \draw (-1.5,-1.7) -- (v1);
    \draw (-0.95,1.7) .. controls (0.05,0.6) and (0.45,0.45) .. (1.7,0.45);
    \draw (-0.95,-1.7) .. controls (0,-0.65) and (0.45,-0.45) .. (1.7,-0.45);
    \draw (-1.5,1.1) .. controls (-0.65,0.1) and (-0.65,-0.1) .. (-1.5,-1.1);
    \draw (-1.5,0.5) .. controls (-1.15,0.1) and (-1.15,-0.1) .. (-1.5,-0.5);
    \draw (-0.4,1.7) .. controls (0.15,1) and (0.6,0.9) .. (1.7,0.9);
    \draw (1.7,-0.9) .. controls (0.65,-0.85) and (0.15,-1.05) .. (-0.4,-1.7);
    \end{tikzpicture}
    \caption{Left: 1-pronged singularity, right: 3-pronged singularity at $x$.}
    \label{f:foliation_sing}
\end{figure}
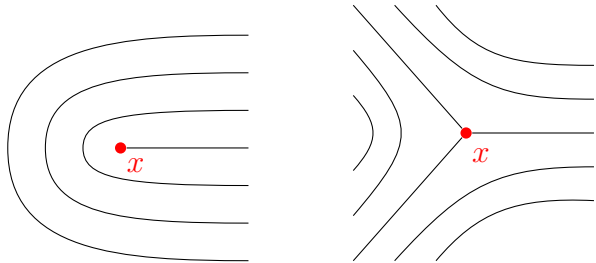

A \emph{transverse measure} for a foliation $F$ is a measure $\mu$ on each transverse arc which is equivalent to the Lebesgue measure on an interval of $\bR$ such that if two transverse arcs $\alpha$ and $\beta$ are isotopic through transverse arcs whose each endpoint remain in the same leaf then $\int_\alpha \mu = \int_\beta \mu$.

Such a pair $(F,\mu)$ is called a \emph{measured foliation} on $\Sigma$.
There is an equivalence relation on measured foliations generated by isotopy and \emph{Whitehead collapses} as illustrated in \cref{f:Whitehead} and the set of the equivalence classes of measured foliations (including the empty foliation $\emptyset$) is denoted by $\MF(\Sigma)$.
We will use the notation $[F, \mu]$ for the element of $\MF(\Sigma)$ which is represented by a measured foliation $(F, \mu)$.
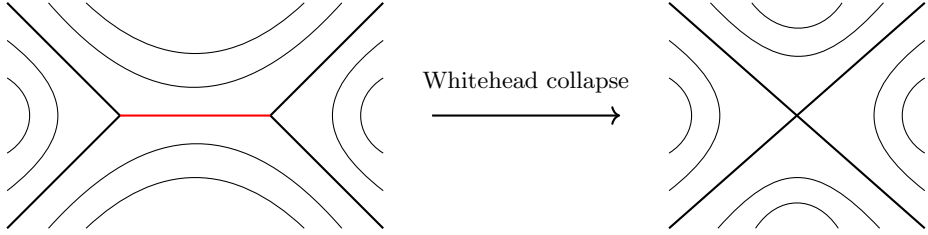
\begin{figure}[h]
    \centering\begin{tikzpicture}
    \draw [thick](-2.5,1.5) -- (-1,0) coordinate (v1);
    \draw [thick] (1,0) coordinate (v2) -- (2.5,1.5);
    \draw [thick](-2.5,-1.5) -- (v1);
    \draw [thick](2.5,-1.5) -- (v2);
    \draw (-1.95,1.5) .. controls (-0.65,0) and (0.6,0) .. (1.95,1.5);
    \draw (-1.45,1.5) .. controls (-0.5,0.6) and (0.5,0.6) .. (1.45,1.5);
    \draw (-2.5,1) .. controls (-1.6,0.3) and (-1.6,-0.3) .. (-2.5,-1);
    \draw (-2.5,0.5) .. controls (-2.1,0.25) and (-2.1,-0.25) .. (-2.5,-0.5);
    \begin{scope}[xscale=-1, yscale=-1]
    \draw (-1.95,1.5) .. controls (-0.65,0) and (0.6,0) .. (1.95,1.5);
    \draw (-1.45,1.5) .. controls (-0.5,0.6) and (0.5,0.6) .. (1.45,1.5);
    \draw (-2.5,1) .. controls (-1.6,0.3) and (-1.6,-0.3) .. (-2.5,-1);
    \draw (-2.5,0.5) .. controls (-2.1,0.25) and (-2.1,-0.25) .. (-2.5,-0.5);
    \end{scope}
    \draw [->, thick](3.15,0) -- (5.65,0);
    \node at (4.4,0.45) {\scriptsize Whitehead collapse};
    \draw [red, thick] (v1) edge (v2);
    \draw [thick](6.3,1.5) -- (8,0) coordinate (v1);
    \draw [thick] (8,0) coordinate (v2) -- (9.7,1.5);
    \draw [thick](6.3,-1.5) -- (v1);
    \draw [thick](9.7,-1.5) -- (v2);
    \draw (6.85,1.5) .. controls (7.75,0.55) and (8.3,0.55) .. (9.15,1.5);
    \draw (7.4,1.5) .. controls (7.75,1.05) and (8.3,1.05) .. (8.6,1.5);
    \draw (6.3,1) .. controls (7.2,0.3) and (7.2,-0.3) .. (6.3,-1);
    \draw (6.3,0.5) .. controls (6.7,0.25) and (6.7,-0.25) .. (6.3,-0.5);
    \begin{scope}[xscale=-1, yscale=-1, shift={(-7,0)}]
    \draw (-2.1,1.5) .. controls (-1.25,0.6) and (-0.75,0.6) .. (0.05,1.5);
    \draw (-1.55,1.5) .. controls (-1.25,1.05) and (-0.75,1.05) .. (-0.45,1.5);
    \draw (-2.7,1) .. controls (-1.8,0.3) and (-1.8,-0.3) .. (-2.7,-1);
    \draw (-2.7,0.5) .. controls (-2.3,0.25) and (-2.3,-0.25) .. (-2.7,-0.5);
    \end{scope}
    \end{tikzpicture}
    \caption{Whitehead collapsing}
    \label{f:Whitehead}
\end{figure}

A \emph{partial foliation} $(F, \mu)$ on $\Sigma$ is a measured foliation on a subsurface of $\Sigma$.
We call the subsurface \emph{support} of $(F, \mu)$ and write $|F|$.
We do not assume that the supports of partial measured foliations are connected.
In contrast to the partial measured foliations, we sometimes use the term \emph{total measured foliation} to refer to a partial measured foliation whose support is $\Sigma \setminus P$.
For a partial measured foliation $(F, \mu)$, one can obtain a total measured foliation $(F_0, \mu_0)$ of $\Sigma$ by collapsing each connected component of $\Sigma \setminus |F|$ to a spine of it and pull-back the transverse measure $\mu$ by the collapsing map.
It is obvious that the equivalence class of $F_0$ does not depend on the choice of the spine, so the collapsing operation gives a well-defined element in $\MF(\Sigma)$.
Thus, one can think that an element of $\MF(\Sigma)$ is represented by a partial measured foliation and we sometimes take a such representative.

Let us briefly mention the natural topology on $\MF(\Sigma)$. Let $\cS(\Sigma)$ denote the set of homotopy classes of non-peripheral simple closed curves on $\Sigma$.
Then it is known that each class $[c] \in \cS(\Sigma)$ has a representative $c_{F} \in [c]$ which attains the minimum measure for $\mu$ in that class for each measured foliation $(F, \mu)$. Moreover if two measured foliations $(F_1,\mu_1)$ and $(F_2,\mu_2)$ are equivalent, then we have $\int_{c_{F_1}}\mu_1 = \int_{c_{F_2}} \mu_2$. See \cite[Section 5.3]{FLP} for the proofs of these statements. Thus we get a well-defined map
\begin{align*}
    I_*^{\mathrm{MF}}: \MF(\Sigma) \to \bR_{\geq 0}^{\cS(\Sigma)}, \quad [F,\mu] \mapsto I^{\mathrm{MF}}_{[F,\mu]}
\end{align*}
with $I^{\mathrm{MF}}_{[F,\mu]}([c]):= \int_{c_F} \mu$ for $[c] \in \cS(\Sigma)$.
The empty foliation is identified with the zero vector in the image. The map $I_*^{\mathrm{MF}}$ is known to be injective (\cite[Theorem 6.13]{FLP}), and hence it pulls-back the weak topology on $\bR_{\geq 0}^{\cS(\Sigma)}$ to $\MF(\Sigma)$.

The space of \emph{decorated measured foliations} is a trivial bundle $\dMF(\Sigma) := \MF(\Sigma) \times \bR^P$.
One can think that an element $(\fol, (c_p)_{p \in P}) \in \dMF(\Sigma)$ is represented by a ``measured foliation with peripheral leaves'' $(\widetilde{F}, \widetilde{\mu})$ on $\Sigma$.
Namely, it restricts to the partial measured foliation $(F,\mu)$ with the support is the subsurface obtained from $\Sigma$ by removing a small open disk $D_p$ around each puncture $p \in P$ such that $[F, \mu] = \fol$, and it restricts to a foliation of $\overline{D_p} \setminus \{p\}$ whose leaves are nonsingular and surrounding $p$ with a ``transverse measure'' such that the arc connecting $p$ and a boundary point of $D_p$ has a measure $c_p$ for each $p \in P$.
Note that $c_p$ can be a negative value.
We call such a representative \emph{decorated measured foliation} and the leaves in $D_p$ \emph{peripheral leaves}.
In this point of view, the space of measured foliations $\MF(\Sigma)$ is considered as the zero-section of $\dMF(\Sigma)$.
\begin{figure}[h]
    \centering
    \begin{tikzpicture}[scale=1.1]
    \draw [fill=blue!20] (1.05,0) .. controls (0.6,1.2) and (-0.9,1.1) .. (-0.9,0) .. controls (-0.9,-1.1) and (0.6,-1.2) .. (1.05,0);
    \node  [fill, circle, red, inner sep = 1.5] (v1) at (0,0) {};
    \node [red] at (0.2,0) {$p$};
    \draw  (v1) ellipse (0.3 and 0.3);
    \draw  (v1) ellipse (0.6 and 0.6);
    \draw (2,0) -- (1.05,0);
    \draw (2,0.4) .. controls (1.25,0.4) and (0.75,1.15) .. (0,1.15) .. controls (-0.75,1.15) and (-1.25,0.7) .. (-1.25,0) .. controls (-1.25,-0.7) and (-0.75,-1.15) .. (0,-1.15) .. controls (0.75,-1.15) and (1.25,-0.4) .. (2,-0.4);
    \draw (2,0.85) .. controls (1.4,0.85) and (1,1.45) .. (0,1.45) .. controls (-0.95,1.45) and (-1.6,0.95) .. (-1.6,0) .. controls (-1.6,-0.95) and (-0.95,-1.45) .. (0,-1.45) .. controls (1,-1.45) and (1.4,-0.85) .. (2,-0.85);
    \end{tikzpicture}
    \caption{The local model of a decorated measured foliation near a puncture $p \in P$. The blue region shows $D_p$.}
    \label{fig:dec_fol}
\end{figure}
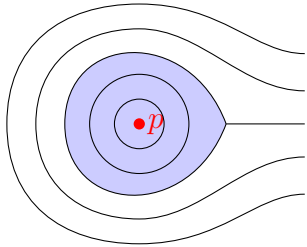

\subsection{Coordinate of \texorpdfstring{$\dMF(\Sigma)$}{the space of decorated measured foliations}}
An \emph{ideal arc} in $\Sigma$ is the isotopy class of a non-contractible curve $\alpha$ such that $\partial\alpha \subset P$.
An \emph{ideal triangulation} $\tri$ of $\Sigma$ is a collection of ideal arcs such that
\begin{itemize}
\item each pair of ideal arcs in $\tri$ can intersect only at their endpoints;
\item each region complementary to $\bigcup \tri$ is a triangle whose edges (resp. vertices) are ideal arcs (resp. punctures).
\end{itemize}
We will refer to the closure of a complementary region of $\bigcup \tri$ as an \emph{ideal triangle} of $\tri$.
The conditions (S1) and (S2) ensure that such an ideal triangulation $\tri$ exists, and in particular the number $-3 \chi(\Sigma) > 0$ gives the number of ideal arcs in $\tri$.
Let $\Tri(\Sigma)$ denote the graph whose vertices are ideal triangulations of $\Sigma$ and adjacent vertices are related by a flip.

For each ideal arc $\alpha$ in $\Sigma$, we define the function $\sfa_\alpha$ on $\MF(\Sigma)$ as 
\begin{align*}
    \sfa_\alpha(\fol) := \frac{1}{2} \int_\alpha \mu
\end{align*}
for a measured foliation $(F, \mu)$ which represents $\fol$.
Also, we extend $\sfa_\alpha$ to $\dMF(\Sigma)$ as 
\begin{align*}
    \sfa_\alpha(\fol, (c_p)_p) := \int_\alpha \widetilde{\mu} = \sfa_\alpha(\fol) + \frac{1}{2} c_{p_1} + \frac{1}{2} c_{p_2}
\end{align*}
where $(\widetilde{F}, \widetilde{\mu})$ is a decorated meausred foliation represents $(\fol, (c_p)_p)$ and $\partial \alpha = \{p_1, p_2\}$.

\begin{prop}[{\cite{PP93}}]\label{prop:A-coord}
For each ideal triangulation $\tri$ of $\Sigma$, the map
\[
\mathbf{\sfa}^\tri: \dMF(\Sigma) \to \bR^\tri,\ (\fol, \mathbf{c}) \mapsto (\sfa_\alpha(\fol, \mathbf{c}))_{\alpha \in \tri}
\] 
gives a global coordinate of $\dMF(\Sigma)$.
\end{prop}
\begin{proof}[Sketch of the proof]
The reconstruction of a decorated measured foliation from a given numbers $(\sfa_\alpha)_{\alpha} \in \bR^\tri$ is given as follows:
First, we take a number $v>0$ such that the numbers $\sfa'_\alpha := \sfa_\alpha + v$ are positive and satisfy triangle inequality for every ideal triangle of $\tri$ with edges $\alpha, \beta, \gamma$:
\[
|\sfa'_\alpha - \sfa'_\beta| \leq \sfa'_\gamma \leq \sfa'_\alpha + \sfa'_\beta.
\]
Then, for each ideal triangle $t$ of $\tri$, we can put the measured foliation such that each leaf surrounds a puncture (see \cref{fig:fol_tri}) and the measure of an ideal arc which is isotopic to an edge $\alpha$ of $t$ is given by $2\sfa'_\alpha$.
By gluing them, we obtain the decorated measusred foliation, that is, there is a component whose support is $D_p$ and whose all the leaves surround a puncture $p$.
By cutting off them and shrinking each complementary region to the puncture contained in the region, we obtain the measured foliation $(F, \mu)$ of $\Sigma$.
Let $c_p$ be the sum of the measure of an arc connecting $p$ and a point of the boundary of $D_p$ and $-v$.
Therefore, we have
\begin{align*}
    \int_\alpha \mu = 2 \sfa'_\alpha - (c_p + v) - (c_q + v) = 2 \sfa_\alpha - c_p - c_q
\end{align*}
for an ideal arc $\alpha \in \tri$ such that $\partial \alpha = \{p, q\}$.
Thus we obtain the element $([F, \mu], (c_p)_p) \in \dMF(\Sigma)$ satisfying $\sfa_\tri([F, \mu], (c_p)_p) = (\sfa_\alpha)_\alpha$.
\end{proof}

We call this coordinate as an \emph{$\cA$-coordinate} of $\dMF(\Sigma)$.
The coordinate transformation for a flip along $\kappa \in \tri$ is given by \cref{fig:trop_A_flip}.

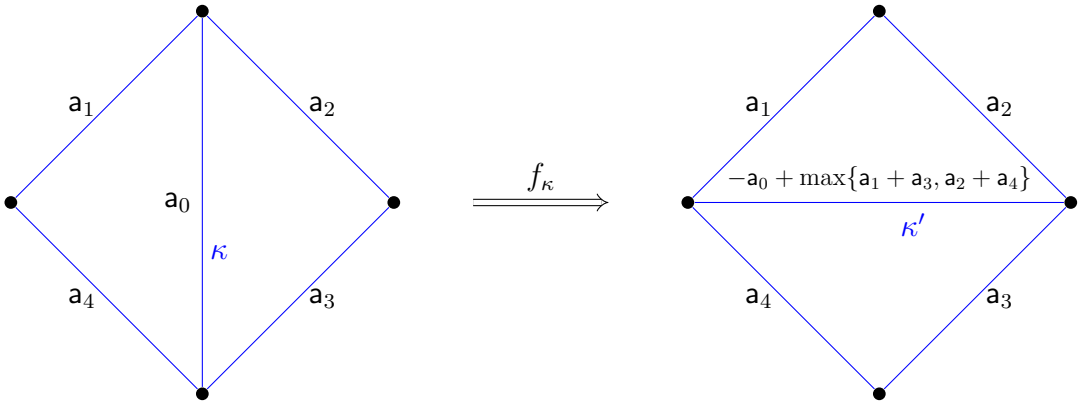
\begin{figure}[h]
\[
\begin{tikzpicture}[scale=0.9]
\path(0,0) node [fill, circle, inner sep=1.7pt] (x1){};
\path(135:4) node [fill, circle, inner sep=1.7pt] (x2){};
\path(0,4*1.4142) node [fill, circle, inner sep=1.7pt] (x3){};
\path(45:4) node [fill, circle, inner sep=1.7pt] (x4){};
\draw [blue] (x1) to node[midway,left, black]{$\sfa_4$} (x2) 
to node[midway,left, black]{$\sfa_1$} (x3) 
to node[midway,right, black]{$\sfa_2$} (x4) 
to node[midway,right, black]{$\sfa_3$} (x1) 
to node[midway,left, black]{$\sfa_0$} (x3);

\draw[-implies, double distance=2pt](4,2*1.4142) to node[midway,above]{$f_{\kappa}$} (6,2*1.4142);

\begin{scope}[xshift=10cm]
\path(0,0) node [fill, circle, inner sep=1.7pt] (x1){};
\path(135:4) node [fill, circle, inner sep=1.7pt] (x2){};
\path(0,4*1.4142) node [fill, circle, inner sep=1.7pt] (x3){};
\path(45:4) node [fill, circle, inner sep=1.7pt] (x4){};
\draw  [blue] (x1) to node[midway,left, black]{$\sfa_4$} (x2) 
to node[midway,left, black]{$\sfa_1$} (x3) 
to node[midway,right, black]{$\sfa_2$} (x4) 
to node[midway,right, black]{$\sfa_3$} (x1);
\draw [blue] (x2) to node[midway,above, black]{\scalebox{0.8}{$-\sfa_0+\max\{\sfa_1+\sfa_3,\sfa_2+\sfa_4\}$}} (x4);
\end{scope}
\node [blue] at (0.25,2.1) {$\kappa$};
\node [blue] at (10.5,2.5) {$\kappa'$};
\end{tikzpicture}
\]
\caption{The coordinate transformation for a flip. Here the transformation rule is still valid even when some of edges are identified.}
\label{fig:trop_A_flip}
\end{figure}

Therefore, the space $\dMF(\Sigma)$ of decorated measured foliations is a PL manifold with the PL atlas given by the $\cA$-coordinates $\mathbf{\sfa}^\tri$ for the ideal triangulations $\tri$.

By the reconstruction of a decorated (partial) measured foliation from an element of $\bR^\tri$ described in the proof of \cref{prop:A-coord}, an element $\widetilde{\fol} \in \dMF(\Sigma)$ has a representative $(\widetilde{F}_\tri, \widetilde{\mu}_\tri)$ such that it has at most one singularity in each ideal triangle of $\tri$.
Its support locally likes in the configuration \cref{fig:fol_tri} in each ideal triangle of $\tri$.
We call this representative a \emph{canonical model} of $\widetilde{\fol}$ with respect to $\tri$.
We note that the leaves peripheral to a vertex drown in \cref{fig:fol_tri} might be missing.

\begin{figure}[h]
    \centering
    \begin{tikzpicture}[scale=0.85]
    \draw [blue](0,3.5)--(-2.65,-1)--(2.65,-1)-- cycle;
    \node [fill, circle, inner sep=1.3pt] at (0,3.5) {};
    \node [fill, circle, inner sep=1.3pt] at (-2.65,-1) {};
    \node [fill, circle, inner sep=1.3pt] at (2.65,-1) {};
    \draw (-0.25,3.05) .. controls (-0.15,2.95) and (0.15,2.95) .. (0.25,3.05);
    \draw (-2.35,-0.55) .. controls (-2.2,-0.6) and (-2.05,-0.85) .. (-2.05,-1);
    \draw (2.05,-1) .. controls (2.05,-0.85) and (2.15,-0.6) .. (2.4,-0.55);
    \draw (-0.5,2.65) .. controls (-0.25,2.4) and (0.25,2.4) .. (0.5,2.65);
    \draw (-2.1,-0.1) .. controls (-1.8,-0.3) and (-1.55,-0.65) .. (-1.55,-1);
    \draw (1.55,-1) .. controls (1.55,-0.6) and (1.85,-0.25) .. (2.15,-0.15);
    \draw (-0.8,2.15) .. controls (-0.35,1.85) and (0.35,1.85) .. (0.8,2.15);
    \draw (-1.85,0.35) .. controls (-1.35,0.1) and (-1.05,-0.45) .. (-1.05,-1);
    \draw (1.05,-1) .. controls (1.05,-0.4) and (1.45,0.1) .. (1.85,0.3);
    \draw (-1.05,1.7) .. controls (-0.4,1.25) and (0.4,1.25) .. (1.05,1.7);
    \draw (-1.6,0.8) .. controls (-0.9,0.45) and (-0.55,-0.05) .. (-0.55,-1);
    \draw (0.55,-1) .. controls (0.55,-0.05) and (0.9,0.45) .. (1.6,0.75);
    \draw [thick](-1.3,1.25) -- (0,0.5) -- (1.35,1.2);
    \draw [thick](0,0.5) -- (0,-1);
    \end{tikzpicture}
    \caption{The local model of $\widetilde{F}_\tri$.}
    \label{fig:fol_tri}
\end{figure}

\subsection{The map \texorpdfstring{$\pot$}{w}}\label{subsec:pot_geom}
We define $\pot_{p}(\fol, (c_p)_p) := c_p$ for each $p \in P$ and $(\fol, (c_p)_p) \in \dMF(\Sigma)$.
They form the natural projection
\[
\pot: \dMF(\Sigma) \to \bR^P,\ (\fol, \mathbf{c}) \mapsto (\pot_p(\fol, \mathbf{c}))_p = \mathbf{c}.
\]
It is clear that $\pot^{-1}(0)$ gives the zero-section of the bundle $\dMF(\Sigma) \to \MF(\Sigma)$.
Thus $\pot^{-1}(0) \cong \MF(\Sigma)$.

\begin{defi}\label{def:T_p}
We denote by $T_{\tri, p}$ the set of ideal triangles of $\tri$ whose boundaries contain $p$.
\end{defi}
\begin{lem}
For an ideal triangulation $\tri$ and a puncture $p \in P$, the map $\pot_p$ is expressed as
\begin{align}\label{eq:pot_geom}
    \pot_{p} = \min \big\{ \sfa_{\alpha^1_{p,t}} + \sfa_{\alpha^2_{p,t}} - \sfa_{\alpha^0_{p,t}} \ \big|\  t \in T_{\tri, p} \big\}
\end{align}
by the $\cA$-coordinate $\mathbf{\sfa}^{\tri} = (\sfa_\alpha)_{\alpha \in \tri}$.
Here, the ideal arcs $\alpha^k_{p,t}$ ($k=0,1,2$) are defined as follows: 
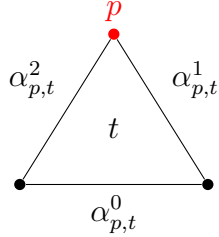
\begin{align*}
    \centering
    \begin{tikzpicture}[auto]
    \node[fill, circle, inner sep=1.5pt, red] (v1) at (0,1.5) {};
    \node[fill, circle, inner sep=1.5pt] (v2) at (-1.25,-0.5) {};
    \node[fill, circle, inner sep=1.5pt] (v3) at (1.25,-0.5) {};
    \draw  (v1) edge node[swap] {$\alpha_{p,t}^2$} (v2);
    \draw  (v2) edge node[swap] {$\alpha_{p,t}^0$} (v3);
    \draw  (v3) edge node[swap] {$\alpha_{p,t}^1$} (v1);
    \node at (0,0.25) {$t$};
    \node[red] at (0,1.81) {$p$};
    \end{tikzpicture}
\end{align*}
\end{lem}

\begin{proof}
For $\widetilde{\fol} \in \dMF(\Sigma)$, take a canonical model $(\widetilde{F}_\tri, \widetilde{\mu}_\tri)$ w.r.t. $\tri$.
Fix an ideal triangle $t \in T_{\tri, p}$.
If there is no singular leaf whose singular endpoint is contained in $t$ and which intersects to $\alpha^0_{p,t}$, then there are only leaves in $t$ which surround the corner $p$.
Otherwise, let us decompose $\alpha^0_{p,t}$ at the intersection point with the singular leaf into two arcs $\beta_1$ and $\beta_2$ such that $\beta_i$ is adjacent to $\alpha^i$ for $i=1,2$.
Also, let $\gamma_i$ be a subarc of $\alpha^i_{p,t}$ which is equivalent to $\beta_i$ for $i=1,2$.
We note that the subtractions $\alpha^1_{p,t} \setminus \gamma_1$ and $\alpha^2_{p,t} \setminus \gamma_2$ are equivalent to the arc $\delta_{p, t}$ which connects $p$ and the singular point in $t$ if exists.
Then, we have
\begin{align*}
    2\big( \sfa_{\alpha^1_{p,t}}(\widetilde{\fol}) + \sfa_{\alpha^2_{p,t}}(\widetilde{\fol}) - \sfa_{\alpha^0_{p,t}}(\widetilde{\fol}) \big)
    &=
    \int_{\alpha^1_{p,t}} \widetilde{\mu}_\tri + \int_{\alpha^2_{p,t}} \widetilde{\mu}_\tri - \int_{\alpha^0_{p,t}} \widetilde{\mu}_\tri\\
    &=
    \int_{\alpha^1_{p,t}} \widetilde{\mu}_\tri + \int_{\alpha^2_{p,t}} \widetilde{\mu}_\tri - \bigg( \int_{\beta_1} \widetilde{\mu}_\tri + \int_{\beta_2} \widetilde{\mu}_\tri \bigg)\\
    &=
    \int_{\alpha^1_{p,t}} \widetilde{\mu}_\tri - \int_{\gamma_1} \widetilde{\mu}_\tri  + \int_{\alpha^2_{p,t}} \widetilde{\mu}_\tri - \int_{\gamma_2} \widetilde{\mu}_\tri\\
    &= 2\int_{\delta_{p, t}} \widetilde{\mu}_\tri.
\end{align*}
Therefore, (RHS) of \eqref{eq:pot_geom} is the measure of the peripheral leaves surrounding $p$, it is nothing but $\pot_p$.
\end{proof}

\section{Train track atlases and cluster coordinates}\label{sec:TT}
A train track, introduced by W.Thurston in \cite{Th}, is a combinatorial model for the measured foliations on a surface\footnote{In \cite{Th}, he considers the measured `laminations', but there is a PL isomorphism between the space of measured foliations and the space of measured laminations \cite{CB}.}.
That is, it is well-known that the ``maximal'' train tracks give a PL atlas on the space $\MF(\Sigma) \setminus \{\emptyset\}$ of measured foliations without the empty foliation.
The theory of train track is mainly described by the pants decompositions.
For instance, the Dehn--Thurston coordinate, which is a global coordinate on the space of measured foliations on a surface associated with a pants decomposition of the surface.
In this section, we try to describe the train tracks by the ideal triangulations.



\subsection{Basic notions of train tracks}\label{subsec:traintrack}
In this subsection, we review the definition of train tracks and their properties. The main reference of this part is \cite{PH}.

\begin{defi}\label{def:traintrack}
A \textit{train track} (or simply \emph{track}) $\tau$ is a (non-necessarily connected) graph with trivalent vertices\footnote{Trivalent train tracks are called \emph{generic} train tracks in \cite{PH}. We can deform any train tracks to a generic one equivalently by \emph{combing}.} embedded in a surface $\Sigma \setminus P$
such that:
\begin{enumerate}
    \item The interior of each edge of $\tau$ is $C^1$.
    \item For each vertex $v$ of $\tau$, there is a well-defined one dimensional tangent space $T_v \tau \subset T_v \Sigma$.
    \item For each connected component $S$  of the complement of $\tau$, the double $D(S)$ of $S$ along the $C^1$ edges of $S$ has the negative Euler characteristic: $\chi(D(S)) < 0$.
\end{enumerate}
\end{defi}
In brief, a train track is a graph embedded into a surface smoothly also vertices look like railways of trains.
Each connected component of the complement of a train track has some cusps, we call it $k$-gon when it has $k$ cusps.
The last condition of the train tracks excludes the following types of connected components of the complement of the train tracks:
\begin{itemize}
    \item annulus,
    \item once punctured null-gon,
    \item unpunctured $k$-gon ($k \leq 2$).
\end{itemize}
We call the edges of train tracks \emph{branches} and the vertices  \emph{switches}.

A subgraph $\tau'$ of a trian track $\tau$ is called \emph{subtrack} of $\tau$.

Let $\tau$ be a train track and let $B(\tau)$ denote the collection of whose branches.
Then, let $V(\tau) \subset \bR^{B(\tau)}$ denote the closed convex cone consisting of the maps $\nu$ satisfying
\begin{align}\label{eq:switch_cond}
\begin{cases}
\nu(b) \geq 0 & \mbox{for every } b \in B(\tau),\\
\nu(b_0) = \nu(b_1) + \nu(b_2) & \mbox{for every switch $s$ in $\tau$ like in the configuration in \cref{fig:switch_cond}.}
\end{cases}    
\end{align}

\begin{figure}[h]
    \centering
    \begin{tikzpicture}[auto]
    \draw[very thick, red] (-1.3,0) .. controls (1,0) and (1,-0.08) .. (1.5,0.5);
    \draw[very thick, red] (-1.3,0) .. controls (1,0) and (1,0.08) .. (1.5,-0.5);
    \node at (0.5,0.25) {$s$};
    \node at (1.75,0.75) {$b_1$};
    \node at (1.75,-0.75) {$b_2$};
    \node at (-1,0.25) {$b_0$};
    \end{tikzpicture}
    \caption{Switch condition.}
    \label{fig:switch_cond}
\end{figure}
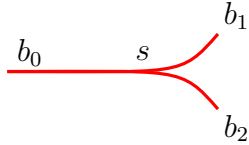

An element of $V(\tau)$ is called \emph{transverse measure} of $\tau$.
A train track $\tau$ is said to be \emph{recurrent} if there is a transverse measure $\nu \in V(\tau)$ such that $\nu(b) > 0$ for all $b \in B(\tau)$.
Moreover, the recurrent train track $\tau$ is \emph{complete} if it is not a proper subtrack of any recurrent train track.

\begin{thm}[{\cite[Corollary 1.4.2]{PH}}]\label{thm:subtrack}
\begin{enumerate}
    \item If $g>1$ or $h>1$, then any recurrent train track is a subtrack of a complete train track, each of whose complementary regions is either once punctured monogon or unpunctured trigon.
    \item If $g=h=1$, then any recurrent train track is a subtrack of a complete train track, whose unique complementary region is once punctured bigon.
\end{enumerate}
\end{thm}

For a cone $V$, we denote by $\dim V$ the dimension of the vector field defined by the $\bR$-span of $V$.
We have $\dim V(\tau) = 6g-6+2h$ for any complete train track $\tau$ (see \cite[Corollary 1.1.3 and Lemma 2.1.1]{PH}).
Therefore, we can think that $V(\tau)$ is a cone of the euclidean space $\bR^{6g-6+2h}$ by taking a basis of the $\bR$-span of $V(\tau)$ in $\bR^{B(\tau)}$.

For a train track $\tau \subset \Sigma$, there is a \emph{fibered neighborhood} $\tau \subset N_\tau \subset \Sigma$ equipped with a retraction $r_\tau: N_\tau \searrow \tau$.
The subspace $N_\tau$ has cusps on its boundary and they coincide with the switches of $\tau$.
We think $N_\tau$ is foliated by the fibers of $r_\tau$ and we call the leaves \emph{ties}.
See \cref{fig:fibered nbd}.
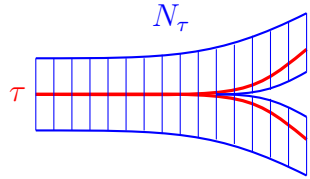
\begin{figure}[h]
    \centering
    \begin{tikzpicture}[scale=1.2]
    \draw[very thick, red] (-1.5,0) .. controls (0.8,0) and (0.8,-0.08) .. (1.5,0.5) node (v1) {};
    \draw[very thick, red] (-1.5,0) .. controls (0.8,0) and (0.8,0.08) .. (1.5,-0.5);
    \draw [blue, thick](-1.5,0.4) .. controls (0,0.4) and (0.5,0.4) .. (1.5,0.9);
    \draw [blue, thick](-1.5,-0.4) .. controls (0,-0.4) and (0.5,-0.4) .. (1.5,-0.9);
    \draw [blue, thick](0.5,0) .. controls (0.95,0) and (1.15,0.05) .. (1.5,0.25);
    \draw [blue, thick](0.5,0) .. controls (0.95,0) and (1.15,-0.05) .. (1.5,-0.25);
    \draw [blue](-1.5,0.4) -- (-1.5,-0.4);
    \draw [blue](-1.3,0.4) -- (-1.3,-0.4);
    \draw [blue](-1.1,0.4) -- (-1.1,-0.4);
    \draw [blue](-0.9,0.4) -- (-0.9,-0.4);
    \draw [blue](0.5,0.5) -- (0.5,-0.5);
    \draw [blue] (0.7,0.55) -- (0.7,0);
    \draw [blue] (0.7,0) -- (0.7,-0.55);
    \draw [blue](1.5,0.25) -- (1.5,0.9);
    \draw [blue](1.5,-0.25) -- (1.5,-0.9);
    \draw [blue](-0.7,0.4) -- (-0.7,-0.4);
    \draw [blue](-0.5,-0.4) -- (-0.5,0.4);
    \draw [blue](-0.3,0.4) -- (-0.3,-0.4);
    \draw [blue](-0.1,-0.44) -- (-0.1,0.44);
    \draw [blue](0.1,0.45) -- (0.1,-0.45);
    \draw [blue](0.3,0.48) -- (0.3,-0.48);
    \draw [blue](0.9,0.65) -- (0.9,0.01);
    \draw [blue](0.9,-0.03) -- (0.9,-0.65);
    \draw [blue](1.1,0.7) -- (1.1,0.05);
    \draw [blue](1.1,-0.05) -- (1.1,-0.7) -- cycle;
    \draw [blue](1.3,0.8) -- (1.3,0.15);
    \draw [blue](1.3,-0.15) -- (1.3,-0.8);
    \node [blue] at (0,0.85) {$N_\tau$};
    \node [red] at (-1.7,0) {$\tau$};
    \end{tikzpicture}
    \caption{A fibered neighborhood of a train track around a vertex}
    \label{fig:fibered nbd}
\end{figure}

We say that $\fol \in \MF(\Sigma)$ is carried by $\tau$, we write $\fol \prec \tau$, if we can deform a representative $(F, \mu)$ of $\cF$ as follows:
cutting and opening along a singular leaf $F$ so that the support of it is contained $N_\tau$ and each leaf of $F$ is transverse to the ties of $N_\tau$.
We denote by $\cV(\tau) \subset \MF(\Sigma)$ the subspace consisting of elements represented by partial measured foliations carried by $\tau$.

We define the map $\psi_\tau: V(\tau) \to \MF(\Sigma)$ as follows:
For each transverse measure $\nu \in V(\tau)$, let $\tau_\nu$ denote the subtrack of $\tau$ consisting of branches $b \in B(\tau)$ such that $\nu(b) >0$.
Let take a measured foliation $(F'_\nu, \mu'_\nu)$ on $N_{\tau_\nu}$ whose leaves are transverse to the ties and
\begin{align*}
    \int_{t_b} \mu'_\nu = \nu(b)
\end{align*}
for each $b \in B(\tau_\nu)$.
Here, $t_b$ is a tie of $N_{\tau_\nu}$ transverses to $b$.
Then, we define the measured foliation $(F_\nu, \mu_\nu)$ on $\Sigma$ by shrinking the complementary regions of $N_{\tau_\nu}$ and $\psi_\tau(\nu) := [F_\nu, \mu_\nu]$.

\begin{thm}[{\cite[Theorem 2.7.4]{PH}}]
The map $\psi_\tau: V(\tau) \to \MF(\Sigma)$ is an embedding with the image $\cV(\tau)$.
\end{thm}

Recall that the space $\dMF(\Sigma)$ of decorated measured foliations is a trivial bundle of $\MF(\Sigma)$ with fibers $\bR^P$.
We denotes by $\widetilde{\cV}(\tau) \subset \dMF(\Sigma)$ the subset consisting of the decorated measured foliations $(\cF, \mathbf{c})$ such that $\cF \in \cV(\tau)$ and define
\begin{align*}
\widetilde{\psi}_\tau := \psi_\tau \times \mathrm{id}_{\bR^P} : \widetilde{V}(\tau) := V(\tau) \times \bR^P \to \dMF(\Sigma).
\end{align*}

\begin{rmk}
We referred the reader to \cite{PH} above but the corresponding statements are about ``birecurrent train tracks''.
A train track is birecurrent if it is recurrent and ``transversely recurrent''.
One can verify that the statements are true for recurrent train tracks.
Moreover, every complete train track $\tau$ which is suited to an ideal triangulation, which is defined in the next subsection and used mainly in this paper, is transversely recurrent (since the map $\rho: B(\tau) \to \bR$ defined by $\rho(b) = 1$ satisfies the conditions of the definition of tangential measure on $\tau$).
\end{rmk}

\subsection{Train tracks which are suited to an ideal triangulation}

In this subsection, we define the type of train tracks which are suited to an ideal triangulation.
These train tracks are a good match to the theory of cluster algebras.

\begin{defi}
A train track $\tau$ is suited to an ideal triangulation $\tri$ if
\begin{itemize}
    \item $\tau$ is transverse to $\tri$ and
    \item in each triangle of $\tri$, $\tau$ likes in the one of the configurations of \cref{fig:suit_train_track}.
\end{itemize}
For such a train track, we say that a branch is \emph{short} if it is contained in a single triangle.
Namely, the rightmost subtrack of \cref{fig:suit_train_track} has three short branches.
Also, we call the other branches \emph{long} branches.
If the train track $\tau$ which is suited to an ideal triangulation $\tri$ likes in the configuration of the left, center or right of \cref{fig:suit_train_track} in a triangle $t$ of $\tri$, then we say that $\tau$ is of type I, II or III on $t$, respectively.

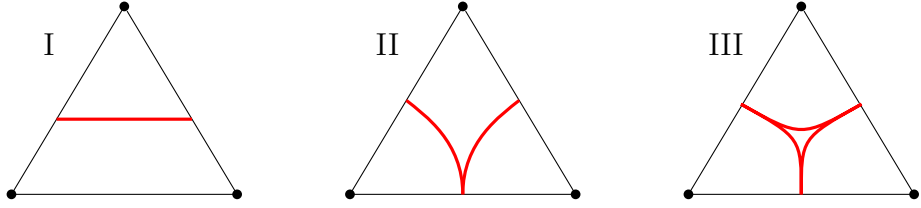
\begin{figure}[h]
    \centering
    \begin{tikzpicture}
    \draw (-0.5,4) node [fill, circle, inner sep=1.3] (v1) {} -- (-2,1.5) node [fill, circle, inner sep=1.3] {} -- (1,1.5) node [fill, circle, inner sep=1.3] {} -- (v1);
    \draw (4,4) node [fill, circle, inner sep=1.3] (v1) {} -- (2.5,1.5) node [fill, circle, inner sep=1.3] {} -- (5.5,1.5) node [fill, circle, inner sep=1.3] {} -- (v1);
    \draw (8.5,4) node [fill, circle, inner sep=1.3] (v1) {} -- (7,1.5) node [fill, circle, inner sep=1.3] {} -- (10,1.5) node [fill, circle, inner sep=1.3] {} -- (v1);
    \draw [red, very thick](-1.4,2.5) -- (0.4,2.5);
    \draw [red, very thick](4,1.5) .. controls (4,2.2) and (3.5,2.55) .. (3.25,2.75);
    \draw [red, very thick](4,1.5) .. controls (4,2.2) and (4.5,2.55) .. (4.75,2.75);
    \draw [red, very thick](7.7,2.7) .. controls (8.5,2.25) and (8.5,2.25) .. (8.5,1.5) .. controls (8.5,2.25) and (8.5,2.25) .. (9.3,2.7) .. controls (8.5,2.25) and (8.5,2.25) .. (7.7,2.7);
    \node at (-1.5,3.5) {I};
    \node at (3,3.5) {II};
    \node at (7.5,3.5) {III};
    \end{tikzpicture}
    \caption{The local models of a train track which suits to an ideal triangulation.}
    \label{fig:suit_train_track}
\end{figure}
\end{defi}

We denote by $\TT_\tri$ the collection of the triangles suited to an ideal triangle $\tri$ and let $\TT^{\max}_\tri \subset \TT_\tri$ denote the subset consisting of the complete train tracks.

\begin{ex}
Let $\Sigma$ be a sphere with four punctures.
We take an ideal triangulation $\tri$ as shown in \cref{fig:tt_fan_4shp}.
Then, $\# \TT_\tri = 8$ and $\# \TT_\tri^{\max} = 4$.

\begin{figure}[h]
    \centering
    \begin{tikzpicture}
    \node [fill, circle, inner sep=1.3] (v1) at (0,-3.5) {};
    \node [fill, circle, inner sep=1.3] (v3) at (0,-6.5) {};
    \node [fill, circle, inner sep=1.3] (v2) at (-1.5,-5) {};
    \node [fill, circle, inner sep=1.3] (v4) at (1.5,-5) {};
    \draw [blue](v1) -- (v2) -- (v3) -- (v4) -- (v1) -- (v3);
    \draw [blue](0,-6.5) .. controls (-1,-7) and (-2.45,-6.5) .. (-2.45,-5) .. controls (-2.45,-3.5) and (-1,-3) .. (0,-3.5);
    \node [blue] at (0.2,-5.4) {\scriptsize 1};
    \node [blue] at (-0.8,-4.05) {\scriptsize 2};
    \node [blue] at (-1,-5.75) {\scriptsize 3};
    \node [blue] at (0.8,-5.95) {\scriptsize 4};
    \node [blue] at (0.85,-4.1) {\scriptsize 5};
    \node [blue] at (-2.65,-5) {\scriptsize 6};
    \draw [red, very thick](0,-5) .. controls (-0.5,-4.9) and (-0.6,-4.6) .. (-0.85,-4.4) .. controls (-1.35,-3.9) and (-2.45,-4.95) .. (-1.45,-5.95) .. controls (0.15,-7.55) and (1.4,-6.1) .. (0.9,-5.6) .. controls (0.65,-5.35) and (0.5,-5.1) .. (0,-5);
    \draw [red, very thick](0,-5) .. controls (-0.5,-4.9) and (-0.95,-6.45) .. (-1.45,-5.95);
    \draw [red, very thick](0,-5);
    \draw [red, very thick](0,-5) .. controls (0.5,-5.1) and (0.45,-4.7) .. (0.95,-4.45) .. controls (1.45,-4.2) and (2,-4.5) .. (2,-5) .. controls (2,-5.5) and (1.4,-6.1) .. (0.9,-5.6);

    \node [fill, circle, inner sep=1.3] (v1) at (6.5,2.8) {};
    \node [fill, circle, inner sep=1.3] (v3) at (6.5,-0.2) {};
    \node [fill, circle, inner sep=1.3] (v2) at (5,1.3) {};
    \node [fill, circle, inner sep=1.3] (v4) at (8,1.3) {};
    \draw [blue](v1) -- (v2) -- (v3) -- (v4) -- (v1) -- (v3);
    \draw [blue](6.5,-0.2) .. controls (5.5,-0.7) and (4,-0.2) .. (4,1.3) .. controls (4,2.95) and (5.85,3.3) .. (6.5,2.8);
    \node [blue] at (6.75,1.3) {\scriptsize 1};
    \node [blue] at (5.5,2.05) {\scriptsize 2};
    \node [blue] at (5.5,0.55) {\scriptsize 3};
    \node [blue] at (7.5,0.55) {\scriptsize 4};
    \node [blue] at (7.5,2.05) {\scriptsize 5};
    \node [blue] at (3.95,0.55) {\scriptsize 6};
    \begin{scope}[xshift=-30, yshift=-190, xscale=-1, shift={(-0.6,6.3)}]
    \draw [red, very thick](-0.5,1.5) .. controls (-1,1.6) and (-1.1,1.9) .. (-1.35,2.1) .. controls (-1.85,2.6) and (-2.95,1.55) .. (-1.95,0.55) .. controls (-0.35,-1.05) and (0.9,0.4) .. (0.4,0.9) .. controls (0.15,1.15) and (0,1.4) .. (-0.5,1.5);
    \draw [red, very thick](-0.5,1.5) .. controls (-1,1.6) and (-1.45,0.05) .. (-1.95,0.55);
    \draw [red, very thick](-0.5,1.5);
    \draw [red, very thick](-0.5,1.5) .. controls (0,1.4) and (-0.05,1.8) .. (0.45,2.05) .. controls (0.95,2.3) and (1.5,2) .. (1.5,1.5) .. controls (1.5,1) and (0.9,0.4) .. (0.4,0.9);
    \end{scope}

    \node [fill, circle, inner sep=1.3] (v1) at (0.1,2.8) {};
    \node [fill, circle, inner sep=1.3] (v3) at (0.1,-0.2) {};
    \node [fill, circle, inner sep=1.3] (v2) at (-1.4,1.3) {};
    \node [fill, circle, inner sep=1.3] (v4) at (1.6,1.3) {};
    \draw [blue](v1) -- (v2) -- (v3) -- (v4) -- (v1) -- (v3);
    \draw [blue](0.1,-0.2) .. controls (-0.9,-0.7) and (-2.25,-0.2) .. (-2.25,1.3) .. controls (-2.15,2.7) and (-0.9,3.3) .. (0.1,2.8);
    \node [blue] at (0.35,1.3) {\scriptsize 1};
    \node [blue] at (-0.9,2.05) {\scriptsize 2};
    \node [blue] at (-0.9,0.55) {\scriptsize 3};
    \node [blue] at (1.1,0.55) {\scriptsize 4};
    \node [blue] at (1.1,2.05) {\scriptsize 5};
    \node [blue] at (-2.45,1.15) {\scriptsize 6};
    \begin{scope}[xshift=400, yshift=-190, xscale=-1, shift={(0.5,6.3)}]
    \draw [red, very thick](7,1.5) .. controls (6.6,1.65) and (6.6,1.8) .. (6.2,2.2) .. controls (5.7,2.7) and (4.5,1.75) .. (5.5,0.5) .. controls (6.5,-0.75) and (8.3,-0.4) .. (7.95,0.5) .. controls (7.75,1.15) and (7.4,1.35) .. (7,1.5);
    \draw [red, very thick](6.2,2.2) .. controls (6.6,1.8) and (5.65,0.35) .. (6.25,-0.1);
    \draw [red, very thick](7.8,0.85) .. controls (7.5,2.3) and (9.05,2.55) .. (9.05,1.5) .. controls (9.05,0.85) and (8.5,0.5) .. (7.8,-0.1);
    \end{scope}

    \node [fill, circle, inner sep=1.3] (v1) at (6.5,-3.5) {};
    \node [fill, circle, inner sep=1.3] (v3) at (6.5,-6.5) {};
    \node [fill, circle, inner sep=1.3] (v2) at (5,-5) {};
    \node [fill, circle, inner sep=1.3] (v4) at (8,-5) {};
    \draw [blue](v1) -- (v2) -- (v3) -- (v4) -- (v1) -- (v3);
    \draw [blue](6.5,-6.5) .. controls (5.5,-7) and (4.05,-6.25) .. (4.05,-5) .. controls (4.05,-3.5) and (5.5,-3) .. (6.5,-3.5);
    \node [blue] at (6.7,-4.85) {\scriptsize 1};
    \node [blue] at (5.75,-4) {\scriptsize 2};
    \node [blue] at (5.4,-5.65) {\scriptsize 3};
    \node [blue] at (7.55,-5.7) {\scriptsize 4};
    \node [blue] at (7.5,-4.25) {\scriptsize 5};
    \node [blue] at (3.85,-5) {\scriptsize 6};
    \draw [red, very thick](6.5,-5) .. controls (6.1,-4.85) and (6.1,-4.7) .. (5.7,-4.3) .. controls (5.2,-3.8) and (4,-4.75) .. (5,-6) .. controls (6,-7.25) and (7.8,-6.9) .. (7.45,-6) .. controls (7.25,-5.35) and (6.9,-5.15) .. (6.5,-5);
    \draw [red, very thick](5.7,-4.3) .. controls (6.1,-4.7) and (5.15,-6.15) .. (5.75,-6.6);
    \draw [red, very thick](7.3,-5.65) .. controls (7,-4.2) and (8.55,-3.95) .. (8.55,-5) .. controls (8.55,-5.65) and (8,-6) .. (7.3,-6.6);
    
    \draw (0,-2) node [blue] {} -- (6.5,-2) node [blue] {};
    \draw (3,-5.5) node [blue] {} -- (3,1.5) node [blue] {};
    
    \begin{scope}[yshift=-78, xshift=-27, scale=.5]
    \node [fill, circle, inner sep=1] (v1) at (-0.5,3) {};
    \node [fill, circle, inner sep=1] (v3) at (-0.5,0) {};
    \node [fill, circle, inner sep=1] (v2) at (-2,1.5) {};
    \node [fill, circle, inner sep=1] (v4) at (1,1.5) {};
    \draw [blue](v1) -- (v2) -- (v3) -- (v4) -- (v1) -- (v3);
    \draw [blue](-0.5,0) .. controls (-1.5,-0.5) and (-3,0) .. (-3,1.5) .. controls (-3,3) and (-1.5,3.5) .. (-0.5,3);
    \draw [red, very thick, yshift=100](-0.5,-2) .. controls (-1,-2) and (-1.4,-1.3) .. (-1.9,-1.3) .. controls (-2.75,-1.3) and (-2.75,-2.7) .. (-1.9,-2.7) .. controls (-1.4,-2.7) and (-1,-2) .. (-0.5,-2) .. controls (0,-2) and (0.45,-2.7) .. (0.95,-2.7) .. controls (1.7,-2.7) and (1.7,-1.3) .. (0.95,-1.3) .. controls (0.45,-1.3) and (0,-2) .. (-0.5,-2);
    \end{scope}
    
    \begin{scope}[yshift=-78, xshift=233, scale=0.5, shift={(0,0)}]
    \node [fill, circle, inner sep=1] (v1) at (-0.5,3) {};
    \node [fill, circle, inner sep=1] (v3) at (-0.5,0) {};
    \node [fill, circle, inner sep=1] (v2) at (-2,1.5) {};
    \node [fill, circle, inner sep=1] (v4) at (1,1.5) {};
    \draw [blue](v1) -- (v2) -- (v3) -- (v4) -- (v1) -- (v3);
    \draw [blue](-0.5,0) .. controls (-1.5,-0.5) and (-3,0) .. (-3,1.5) .. controls (-3,3) and (-1.5,3.5) .. (-0.5,3);
    \draw [red, very thick, yshift=100](-1,-3.7) .. controls (-1.5,-3.4) and (-1.3,-1.5) .. (-1.9,-1.3) .. controls (-2.55,-1.2) and (-2.55,-2) .. (-2.1,-2.7) .. controls (-1.8,-3.2) and (-1.5,-3.4) .. (-1,-3.7) .. controls (0.1,-4.1) and (0.55,-3.3) .. (1.05,-2.7) .. controls (1.5,-2.1) and (1.6,-1.2) .. (0.95,-1.3) .. controls (0.35,-1.5) and (0.3,-4.3) .. (-1,-3.7);
    \end{scope}
    
    \begin{scope}[yshift=42, xshift=92, scale=0.5, shift={(0,-17.5)}]
    \node [fill, circle, inner sep=1] (v1) at (-0.5,3) {};
    \node [fill, circle, inner sep=1] (v3) at (-0.5,0) {};
    \node [fill, circle, inner sep=1] (v2) at (-2,1.5) {};
    \node [fill, circle, inner sep=1] (v4) at (1,1.5) {};
    \draw [blue](v1) -- (v2) -- (v3) -- (v4) -- (v1) -- (v3);
    \draw [blue](-0.5,0) .. controls (-1.5,-0.5) and (-3,0) .. (-3,1.5) .. controls (-3,3) and (-1.5,3.5) .. (-0.5,3);
    \draw [red, very thick, rotate=-45] (-1.5,-0.3586) ellipse (1.5 and 0.6);
    \end{scope}

    \begin{scope}[yshift=-194, xshift=92, scale=0.5, shift={(0,13)}]
    \node [fill, circle, inner sep=1] (v1) at (-0.5,7.5) {};
    \node [fill, circle, inner sep=1] (v3) at (-0.5,4.5) {};
    \node [fill, circle, inner sep=1] (v2) at (-2,6) {};
    \node [fill, circle, inner sep=1] (v4) at (1,6) {};
    \draw [blue](v1) -- (v2) -- (v3) -- (v4) -- (v1) -- (v3);
    \draw [blue](-0.5,4.5) .. controls (-1.5,4) and (-3,4.5) .. (-3,6) .. controls (-3,7.5) and (-1.5,8) .. (-0.5,7.5);
    \draw [red, very thick, rotate=45] (3.8235,3.5407) ellipse (1.5 and 0.6);
    \end{scope}
    \end{tikzpicture}  
    \caption{The cones $V(\tau)$ for $\tau \in \TT_\tri$.}
    \label{fig:tt_fan_4shp}
\end{figure}
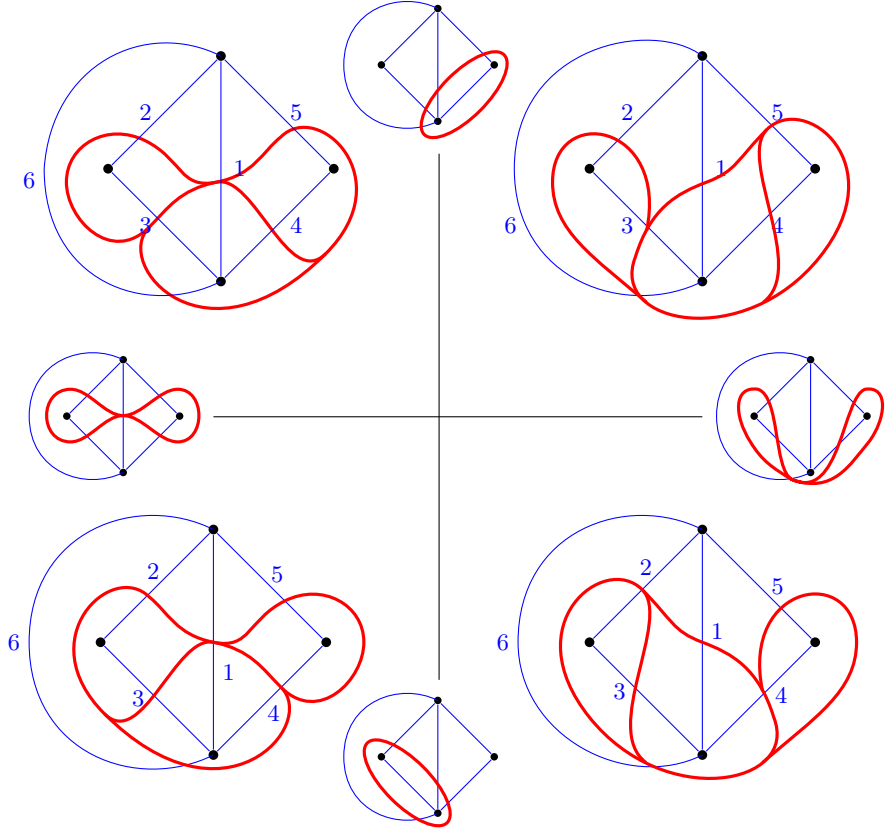
\end{ex}

Let $D_\tri(\pot)$ denote the set of maximal domains of linearity of the map $\pot : \dMF(\Sigma) \to \bR^P$ in the coordinate $\mathbf{\sfa}^\tri$.

\begin{thm}\label{thm:TT_fan}
Let $\tri$ be an ideal triangulation of $\Sigma$.
Then, the map
\begin{align*}
    \widetilde{\cV}: \TT_\tri^{\max} \to D_\tri(\pot),\ \tau \mapsto \widetilde{\cV}(\tau)
\end{align*}
is well-defined bijection.
\end{thm}

\begin{proof}
First, we see that the well-definedness of the map $\widetilde{\cV}$.
Let us take $\tau \in \TT^{\max}_\tri$.
Then, $\tau$ is of type II or III since $\tau$ is complete.
That is, if there is a triangle on which $\tau$ is of type I, then the complementary region which contains the side to which $\tau$ is not transverse has two punctures.
Moreover, for each puncture $p \in P$, there is a unique triangle $t_p$ of $\tri$ on which $\tau$ is of type II, and $\tau$ is of type III on the other triangles of $\tri$.

We note that the absence of the short branch across from $p$ in $t_p$ corresponds to the equality $\pot_p = 0$ in $\dMF(\Sigma)$.
Let $\pot_{t_p} := \sfa_{\alpha^1_p} + \sfa_{\alpha^2_p} - \sfa_{\alpha^0_p}$ where $\alpha^i_p$ are the edges of $t_p$ as \cref{fig:t_p}.
Then, we can write the cone $\cV(\tau)$ as
\begin{align*}
    \cV(\tau) = \{ \cF \mid \pot_p(\cF) = \pot_{t_p}(\cF) = 0,\ p \in P \} \subset \MF(\Sigma)
\end{align*}
since $\pot^{-1}(0) \cong \MF(\Sigma)$.
Therefore, we have
\begin{align*}
    \widetilde{\cV}(\tau) = \{ \widetilde{\cF} \mid \pot_p(\widetilde{\cF}) = \pot_{t_p}(\widetilde{\cF}),\ p \in P \} \subset \dMF(\Sigma).
\end{align*}
Namely, the map $\pot_p$ is the linear map $\pot_{t_p}$ on $\widetilde{\cV}(\tau)$.
If $\widetilde{\cV}(\tau)$ is not maximal among the domains of linearity of $\pot$, then it is a face of the other domain but this cannot be happen since dimension of $\widetilde{\cV}(\tau)$ is full by the completeness of $\tau$.
Therefore, $\widetilde{\cV}(\tau) \in D_\tri(\pot)$.
\begin{figure}[h]
    \centering
    \begin{tikzpicture}
    \draw (4,4) node [fill, circle, inner sep=1.5, red] (v1) {} -- (2.5,1.5) node [fill, circle, inner sep=1.5] {} -- (5.5,1.5) node [fill, circle, inner sep=1.5] {} -- (v1);
    \node[red] at (4.5,4) {$p$};
    \draw [red, very thick](4,1.5) .. controls (4,2.2) and (3.5,2.55) .. (3.25,2.75);
    \draw [red, very thick](4,1.5) .. controls (4,2.2) and (4.5,2.55) .. (4.75,2.75);
    \node at (2.5,2.75) {$\alpha^2_{p}$};
    \node at (5.35,2.75) {$\alpha^1_{p}$};
    \node at (4,1.1) {$\alpha^0_{p}$};
    \end{tikzpicture}
    \vspace{-5mm}
    \caption{$t_p$.}
    \label{fig:t_p}
\end{figure}

Conversely, we take $\widetilde{\cK} \in D_\tri(\pot)$.
From the definition of $D_\tri(\pot)$, for any $\widetilde{\cF} \in \widetilde{\cK}$ and $p \in P$, there is a triangle $t_p \in T_{\tri, P}$ such that $\pot_p(\widetilde{\cF}) = \pot_{t_p}(\widetilde{\cF})$.
Also, let $\cK := \widetilde{\cK} \cap \MF(\Sigma)$, then $\widetilde{\cK} = \cK \times \bR^P$ since the form $\pot_p = \pot_{t_p}$ is invariant under the action $\bR^P$ on $\MF(\Sigma)$ for any $p \in P$.
Next, we define the train track $\tau_{\bft}$ from the tuple of triangles $\bft = (t_p)_p$.
Let $\tau_\tri$ be a freeway of an ideal triangulation $\tri$ of $\Sigma$.
Namely, $\tau_\tri$ is a graph obtained from a dual fat graph of $\tri$ by replacing the neighborhood of each trivalent vertex to the small cusped triangle like in \cref{fig:freeway}.
A freeway has an once punctured null-gon as a complementary region of $\tau_\tri$ for each puncture $p \in P$, hence it is not a train track.
Although, the graph $\tau_{\bft}$, obtained by cutting off the short branch in each triangle $t_p$ which is across from the puncture $p$ for each $p \in P$, is a complete train track since the complementary regions of it are consisting of only once punctured monogons or unpunctured trigons.
In particular, the train track $\tau_{\bft}$ is suited to $\tri$ by the construction.
It is clear that $\pot_p = \pot_{t_p}$ on $\widetilde{\cV}(\tau_{\bft})$, thus $\widetilde{K} = \widetilde{\cV}(\tau_{\bft})$.
\begin{figure}[h]
    \centering
    \begin{tikzpicture}[scale=0.91]
    \node [fill, circle, inner sep=1.5pt] (v1) at (-7,1) {};
    \node [fill, circle, inner sep=1.5pt] (v5) at (-6,-1.5) {};
    \node [fill, circle, inner sep=1.5pt] (v4) at (-8.5,-0.5) {};
    \node [fill, circle, inner sep=1.5pt] (v6) at (-4.5,0.5) {};
    \node (v2) at (-7.5,2.5) {};
    \node (v3) at (-6,2.5) {};
    \node (v9) at (-10,-0.5) {};
    \node (v8) at (-6,-3) {};
    \node (v7) at (-3,0) {};
    \draw  (v1) edge (v2);
    \draw  (v1) edge (v3);
    \draw  (v1) edge (v4);
    \draw  (v4) edge (v5);
    \draw  (v5) edge (v1);
    \draw  (v1) edge (v6);
    \draw  (v6) edge (v5);
    \draw  (v6) edge (v7);
    \draw  (v5) edge (v8);
    \draw  (v4) edge (v9);
    \coordinate (v11) at (-7.1,-0.4) {} {} {};
    \coordinate (v13) at (-5.8,-0.1) {} {} {};
    \node (v14) at (-5.5,1.5) {};
    \node (v15) at (-4.5,-1) {};
    \node (v12) at (-8,-2) {};
    \node (v10) at (-8.5,1) {};
    \draw [very thick, red] (v10) edge (v11);
    \draw [very thick, red] (v11) edge (v12);
    \draw [very thick, red] (v11) edge (v13);
    \draw [very thick, red] (v13) edge (v14);
    \draw [very thick, red] (v13) edge (v15);
    
    \draw[ultra thick,-{Classical TikZ Rightarrow[length=4pt]},decorate,decoration={snake,amplitude=1.5pt,pre length=2pt,post length=3pt}] (-2.5, 0) to (-1, 0);
    
    \node [fill, circle, inner sep=1.5pt] (v1) at (2.5,1) {};
    \node [fill, circle, inner sep=1.5pt] (v5) at (3.5,-1.5) {};
    \node [fill, circle, inner sep=1.5pt] (v4) at (1,-0.5) {};
    \node [fill, circle, inner sep=1.5pt] (v6) at (5,0.5) {};
    \node (v2) at (2,2.5) {};
    \node (v3) at (3.5,2.5) {};
    \node (v9) at (-0.5,-0.5) {};
    \node (v8) at (3.5,-3) {};
    \node (v7) at (6.5,0) {};
    \draw  (v1) edge (v2);
    \draw  (v1) edge (v3);
    \draw  (v1) edge (v4);
    \draw  (v4) edge (v5);
    \draw  (v5) edge (v1);
    \draw  (v1) edge (v6);
    \draw  (v6) edge (v5);
    \draw  (v6) edge (v7);
    \draw  (v5) edge (v8);
    \draw  (v4) edge (v9);
    \node (v14) at (3.8,1.5) {};
    \node (v15) at (5,-1) {};
    \node (v12) at (1.7,-2) {};
    \node (v10) at (1,1) {};
    \coordinate (a1) at (1.8,0.2) {} {} {};
    \coordinate (a2) at (2.8,-0.3) {};
    \coordinate (a3) at (2.1,-0.9) {};
    \coordinate (b1) at (3.7,0.6) {} {} {};
    \coordinate (b2) at (4.2,-0.4) {};
    \coordinate (b3) at (3.2,-0.2) {};
    \draw [very thick, red] (v10) edge (a1);
    \draw [very thick, red] (a3) edge (v12);
    \draw [very thick, red] (a2) edge (b3);
    \draw [very thick, red] (b1) edge (v14);
    \draw [very thick, red] (b2) edge (v15);
    \coordinate (a) at (2.26, -0.33);
    \coordinate (b) at (3.7, 0);
    \draw [very thick, red] (a1) .. controls (a) and (a) .. (a2);
    \draw [very thick, red] (a2) .. controls (a) and (a) .. (a3);
    \draw [very thick, red] (a3) .. controls (a) and (a) .. (a1);
    \draw [very thick, red] (b1) .. controls (b) and (b) .. (b2);
    \draw [very thick, red] (b2) .. controls (b) and (b) .. (b3);
    \draw [very thick, red] (b3) .. controls (b) and (b) .. (b1);
    \end{tikzpicture}    
    \caption{The deformation from the dual fat graph to the freeway.}
    \label{fig:freeway}
\end{figure}
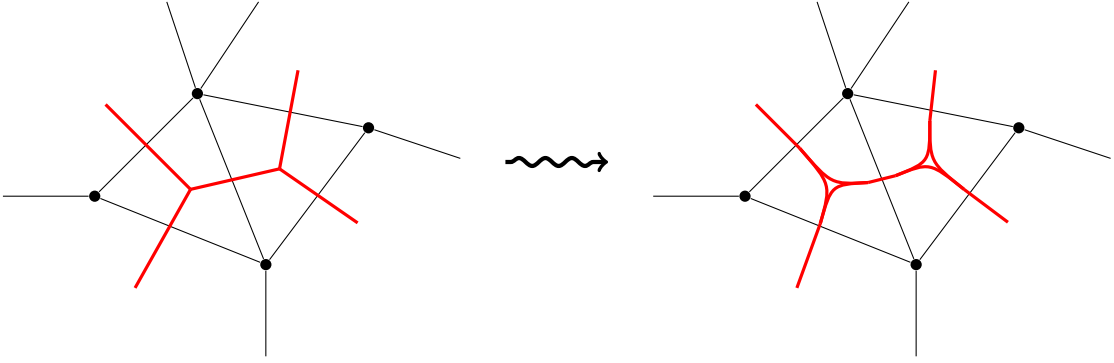
\end{proof}

\begin{ex}\label{ex:5-shp}
Here, we give an example of the construction of the train track $\tau_{\bft}$ from the tuple of triangles $\bft$ corresponding to an element in $D_\tri(\pot)$ in the proof of \cref{thm:TT_fan}.
Let $\Sigma$ be a sphere with five punctures (\emph{i.e.}, $g=0$, $h=5$) and give a triangulation $\tri$ like in the left of \cref{fig:5-sph}.
Also, we label the punctures and the triangles as on the left of \cref{fig:5-sph}.
Let us take a tuple $\bft = (t_{p_i} \mid i =1, \dots, 5) =  (t_5, t_2, t_3, t_4, t_6)$ of triangles of $\tri$.
One can verify the existence of the element of $D_\tri(\pot)$ corresponding to $\bft$.
The right of \cref{fig:5-sph} shows the train track $\tau_{\bft}$ obtained from the freeway $\tau_\tri$.
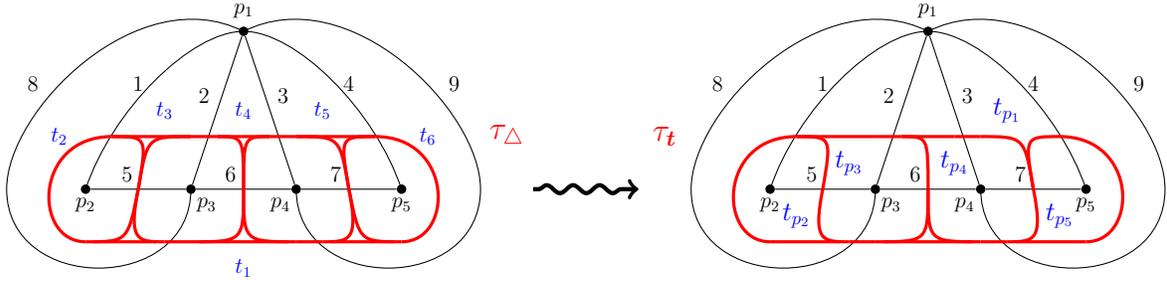
\begin{figure}[h]
    \centering
    \begin{tikzpicture}[auto, scale=0.7, every node/.style={scale=0.7}]
    \node[fill, circle, inner sep=1.8pt] (v2) at (-7.5,0) {};
    \node[fill, circle, inner sep=1.8pt] (v3) at (-5.5,0) {};
    \node[fill, circle, inner sep=1.8pt] (v1) at (-9.5,0) {};
    \node[fill, circle, inner sep=1.8pt] (v4) at (-3.5,0) {};
    \node[fill, circle, inner sep=1.8pt] (v5) at (-6.5,3) {};
    \draw  (v1) to node [xshift=-6pt] {5}  (v2);
    \draw  (v2) to node [xshift=-7pt] {6} (v3);
    \draw  (v3) to node [xshift=-7pt] {7} (v4);
    \draw  (v5) to node [swap] {2} (v2);
    \draw  (v5) to node {3} (v3);
    \node at (-8.5,2) {1};
    \node at (-4.5,2) {4};
    \node at (-10.5,2) {8};
    \node at (-2.5,2) {9};
    \draw (-6.5,3) .. controls (-7.5,3) and (-9,1.5) .. (-9.5,0);
    \draw (-6.5,3) .. controls (-5.5,3) and (-4,1.5) .. (-3.5,0);
    \draw (-7.5,0) .. controls (-7.5,-2) and (-11,-2) .. (-11,0) .. controls (-11,1.5) and (-8.5,4) .. (-6.5,3);
    \draw (-5.5,0) .. controls (-5.5,-2) and (-2,-2) .. (-2,0) .. controls (-2,1.5) and (-4.5,4) .. (-6.5,3);
    \coordinate (w1) at (-9,1) {} {};
    \coordinate (v6) at (-7.5,1) {} {} {};
    \coordinate (w2) at (-8.5,0) {} {};
    \coordinate (v7) at (-7,1) {} {} {};
    \coordinate (v8) at (-6,1) {} {} {};
    \coordinate (w3) at (-6.5,0) {} {};
    \coordinate (v9) at (-5.5,1) {} {} {};
    \coordinate (w4) at (-4,1) {} {};
    \coordinate (w5) at (-4.5,0) {} {};
    \coordinate (v11) at (-7.5,-1) {} {} {};
    \coordinate (v12) at (-5.5,-1) {} {};
    \coordinate (v10) at (-8,-1) {} {} {};
    \coordinate (w6) at (-9.5,-1) {} {};
    \coordinate (v13) at (-5,-1) {} {};
    \coordinate (w7) at (-3.5,-1) {} {};
    \draw [very thick, red] (v6) edge (v7);
    \draw [very thick, red] (v8) edge (v9);
    \draw [very thick, red] (v10) edge (v11);
    \draw [very thick, red] (v12) edge (v13);
    \draw [very thick, red](-9,1) .. controls (-10.5,1) and (-10.5,-1) .. (-9.5,-1);
    \draw [very thick, red](-4,1) .. controls (-2.5,1) and (-2.5,-1) .. (-3.5,-1);
    \coordinate (c) at (-8.33,1) {} {};
    \draw[very thick, red] (-9,1) .. controls (c) and (c) .. (-7.5,1) ;
    \draw[very thick, red] (-7.5,1) .. controls (c) and (c) .. (-8.5,0) ;
    \draw[very thick, red] (-8.5,0) .. controls (c) and (c) .. (-9,1) ;
    \coordinate (c) at (-6.5,1) {} {};
    \draw[very thick, red] (-7,1) .. controls (c) and (c) .. (-6,1) ;
    \draw[very thick, red] (-6,1) .. controls (c) and (c) .. (-6.5,0) ;
    \draw[very thick, red] (-6.5,0) .. controls (c) and (c) .. (-7,1) ;
    \coordinate (c) at (-4.67,1) {};
    \draw[very thick, red] (-5.5,1) .. controls (c) and (c) .. (-4,1) ;
    \draw[very thick, red] (-4,1) .. controls (c) and (c) .. (-4.5,0) ;
    \draw[very thick, red] (-4.5,0) .. controls (c) and (c) .. (-5.5,1) ;
    \draw [very thick, red](-8.5,0) .. controls (-8.7,-1) and (-8.7,-1) .. (-9.5,-1) node (v14) {};
    \draw [very thick, red](-8.5,0) .. controls (-8.7,-1) and (-8.5,-1) .. (-8,-1) node (v15) {};
    \draw [very thick, red](-6.5,0) .. controls (-6.5,-1) and (-6.5,-1) .. (-7.5,-1) node (v16) {};
    \draw [very thick, red](-6.5,0) .. controls (-6.5,-1) and (-6.5,-1) .. (-5.5,-1) node (v17) {};
    \draw [very thick, red](-4.5,0) .. controls (-4.3,-1) and (-4.5,-1) .. (-5,-1) node (v18) {};
    \draw [very thick, red](-4.5,0) .. controls (-4.3,-1) and (-4.3,-1) .. (-3.5,-1) node (v19) {};
    \draw [very thick, red] (v14) edge (v15);
    \draw [very thick, red] (v16) edge (v17);
    \draw [very thick, red] (v18) edge (v19);
    \node at (-6.5,3.4) {$p_1$};
    \node at (-9.5,-0.3) {$p_2$};
    \node at (-7.2,-0.3) {$p_3$};
    \node at (-5.8,-0.3) {$p_4$};
    \node at (-3.5,-0.3) {$p_5$};
    \node[blue] at (-6.5,-1.5) {$t_1$};
    \node[blue] at (-10,1) {$t_2$};
    \node[blue] at (-8,1.5) {$t_3$};
    \node[blue] at (-6.5,1.5) {$t_4$};
    \node[blue] at (-5,1.5) {$t_5$};
    \node[blue] at (-3,1) {$t_6$};
    
    \node[red] at (-1.5, 1) {\Large $\tau_\tri$};
    \draw[ultra thick,-{Classical TikZ Rightarrow[length=4pt]},decorate,decoration={snake,amplitude=1.5pt,pre length=2pt,post length=3pt}] (-1, 0) to (1, 0);
    \node[red] at (1.5, 1) {\Large $\tau_{\bft}$};
    
    \node[fill, circle, inner sep=1.8pt] (v2) at (5.5,0) {};
    \node[fill, circle, inner sep=1.8pt] (v3) at (7.5,0) {};
    \node[fill, circle, inner sep=1.8pt] (v1) at (3.5,0) {};
    \node[fill, circle, inner sep=1.8pt] (v4) at (9.5,0) {};
    \node[fill, circle, inner sep=1.8pt] (v5) at (6.5,3) {};
    \draw  (v1) to node [xshift=-6pt] {5}  (v2);
    \draw  (v2) to node [xshift=-7pt] {6} (v3);
    \draw  (v3) to node [xshift=-7pt] {7} (v4);
    \draw  (v5) to node [swap] {2} (v2);
    \draw  (v5) to node {3} (v3);
    \node at (4.5,2) {1};
    \node at (8.5,2) {4};
    \node at (2.5,2) {8};
    \node at (10.5,2) {9};
    \draw (6.5,3) .. controls (5.5,3) and (4,1.5) .. (3.5,0);
    \draw (6.5,3) .. controls (7.5,3) and (9,1.5) .. (9.5,0);
    \draw (5.5,0) .. controls (5.5,-2) and (2,-2) .. (2,0) .. controls (2,1.5) and (4.5,4) .. (6.5,3);
    \draw (7.5,0) .. controls (7.5,-2) and (11,-2) .. (11,0) .. controls (11,1.5) and (8.5,4) .. (6.5,3);
    \coordinate (w1) at (4,1) {} {};
    \coordinate (v6) at (5.5,1) {} {} {};
    \coordinate (w2) at (4.5,0) {} {};
    \coordinate (v7) at (6,1) {} {} {};
    \coordinate (v8) at (7,1) {} {} {};
    \coordinate (w3) at (6.5,0) {} {};
    \coordinate (v9) at (7.5,1) {} {} {};
    \coordinate (w4) at (9,1) {} {};
    \coordinate (w5) at (8.5,0) {} {};
    \coordinate (v11) at (5.5,-1) {} {} {};
    \coordinate (v12) at (7.5,-1) {} {};
    \coordinate (v10) at (5,-1) {} {} {};
    \coordinate (w6) at (3.5,-1) {} {};
    \coordinate (v13) at (8,-1) {} {};
    \coordinate (w7) at (9.5,-1) {} {};
    \draw [very thick, red] (v6) edge (v7);
    \draw [very thick, red] (v8) edge (v9);
    \draw [very thick, red] (v10) edge (v11);
    \draw [very thick, red] (v12) edge (v13);
    \draw [very thick, red](4,1) .. controls (2.5,1) and (2.5,-1) .. (3.5,-1);
    \draw [very thick, red](9,1) .. controls (10.5,1) and (10.5,-1) .. (9.5,-1);
    \coordinate (c) at (4.67,1) {} {};
    \draw[very thick, red] (4,1) .. controls (c) and (c) .. (5.5,1) ;
    \draw[very thick, red] (4.5,0) .. controls (c) and (c) .. (4,1) ;
    \coordinate (c) at (6.5,1) {} {};
    \draw[very thick, red] (6,1) .. controls (c) and (c) .. (7,1) ;
    \draw[very thick, red] (6.5,0) .. controls (c) and (c) .. (6,1) ;
    \coordinate (c) at (8.33,1) {};
    \draw[very thick, red] (9,1) .. controls (c) and (c) .. (8.5,0) ;
    \draw[very thick, red] (8.5,0) .. controls (c) and (c) .. (7.5,1) ;
    \draw [very thick, red](4.5,0) .. controls (4.3,-1) and (4.5,-1) .. (5,-1) node (v15) {};
    \draw [very thick, red](6.5,0) .. controls (6.5,-1) and (6.5,-1) .. (5.5,-1) node (v16) {};
    \draw [very thick, red](6.5,0) .. controls (6.5,-1) and (6.5,-1) .. (7.5,-1) node (v17) {};
    \draw [very thick, red](8.5,0) .. controls (8.7,-1) and (8.5,-1) .. (8,-1) node (v18) {};
    \draw [very thick, red] (3.5,-1) edge (v15);
    \draw [very thick, red] (v16) edge (v17);
    \draw [very thick, red] (v18) edge (9.5,-1);
    \node at (6.5,3.4) {$p_1$};
    \node at (3.5,-0.3) {$p_2$};
    \node at (5.8,-0.3) {$p_3$};
    \node at (7.2,-0.3) {$p_4$};
    \node at (9.5,-0.3) {$p_5$};
    \node[blue] at (4,-0.5) {\large $t_{p_2}$};
    \node[blue] at (5,0.5) {\large $t_{p_3}$};
    \node[blue] at (7,0.5) {\large $t_{p_4}$};
    \node[blue] at (8,1.5) {\large $t_{p_1}$};
    \node[blue] at (9,-0.5) {\large $t_{p_5}$};
    \end{tikzpicture}
    \caption{A complete train track on the sphere with five punctures obtained from the tuple $\mathbf{t}$ of triangles of an ideal triangle $\tri$.}
    \label{fig:5-sph}
\end{figure}
\end{ex}

From the definition of the piecewise linearity of the tropicalization of $\pot$, we have the following:
\begin{cor}\label{cor:TT_fan}
For each ideal triangulation of $\Sigma$, we have
\begin{align*}
    \bigcup_{\tau \in \TT^{\max}_\tri} \cV(\tau) = \dMF(\Sigma).
\end{align*}
Moreover, $\interior \cV(\tau_1) \cap \interior \cV(\tau_2) = \emptyset$ for $\tau_1 \neq \tau_2 \in \TT^{\max}_\tri$.
Namely, the set $\{\cV(\tau) \mid \tau \in \TT^{\max}_\tri \}$ gives a complete fan on $\dMF(\Sigma)$.
\end{cor}

\subsection{Elementary moves for train tracks}

In this subsection, we are going to consider the deformations of train tracks.
When considering a move of an ideal triangulation arising from a mapping class, we can decompose it into a composition of elementary moves, called flips.
For a kind of train track which ``suit'' to a mapping class, we can consider the move of the train track and decompose it to some ``elementary moves" of train tracks as flips of ideal triangulations.

First, we recall the most basic relation of train tracks, called \emph{carrying} (\emph{cf}. the carrying for the measured foliations \cref{subsec:traintrack}).
For two train tracks $\tau_1, \tau_2 \subset \Sigma$, $\tau_1$ is carried by $\tau_2$, we write $\tau_1 \prec \tau_2$, if $\tau_1$ is homotopic to a train track $\tau'_1$ such that $\tau'_1 \subset \interior N_{\tau_2}$ and it is transverse to the ties of $N_{\tau_2}$.
From the definition of the carrying for the measured foliations, it is clear that
\begin{align*}
    \cV(\tau_1) \subset \cV(\tau_2) \subset \MF(\Sigma)
\end{align*}
for any train tracks $\tau_1, \tau_2$ such that $\tau_1 \prec \tau_2$.
Thus, the composition 
\begin{align*}
    \psi_{\tau_2}^{-1} \circ \psi_{\tau_1}: V(\tau_1) \to V(\tau_2)
\end{align*}
is (the restriction of) a linear map.

Let us think that $V(\tau_i)$ is a cone of the vector space $\bR^{B(\tau_i)}$
for $i = 1,2$.
Then, the presentation matrix of some of the inverses of the linear extensions of the map $\psi_{\tau_2}^{-1} \circ \psi_{\tau_1}$ is called \emph{transition matrix} of the carrying $\tau_1 \prec \tau_2$.
We define it here.

Let $M = (m_{b_1, b_2})_{b_1 \in B(\tau_1), b_2 \in B(\tau_2)}$ denote the transition matrix of the carrying $\tau_1 \prec \tau_2$.
Then, 
\begin{align*}
    m_{b_1, b_2} := | r_{\tau_2}^{-1}(p_{b_2}) \cap b'_1 |
\end{align*}
where $p_{b_2}$ is an interior point of the branch ${b_2} \in B(\tau_2)$ and $b'_1$ is the branch of $\tau'_1$ corresponding to $b_1$ via the homotopy between $\tau_1$ and $\tau'_1$.
Obviously, this definition depends on the choice of the embedding $\tau'_1 \subset N_{\tau_2}$, the retraction $r_{\tau_2}: N_{\tau_2} \searrow \tau_2$ and the points $p_{b_2} \in b_2 \in B(\tau_2)$.
Owing to the switch conditions, these matrices define the same map on $\psi_{\tau_2}^{-1} \circ \psi_{\tau_1}(V(\tau_1))$ and this map coincides with the inverse of
$\psi_{\tau_2}^{-1} \circ \psi_{\tau_1}$.
We refer the reader to \cite[Proposition 1.7.5]{PH} for more details.


\begin{defi}[Elementary moves for train tracks]
Let $\tau$ be a train track.
An \emph{elementary move} is a local deformation of the train track around a branch $b_k \in B(\tau)$.
It is classified into 3 types:
\begin{enumerate}
    \item Splitting:
    If the train track $\tau$ looks like the left of \cref{fig:split} around the branch $b_k$, then there are 3 possibilities of splittings of it, called left, right and central splitting.
    They are drawn in the right of \cref{fig:split}.
    \item Folding:
    It is a general term for reverse deformations of any possibilities of splittings.
    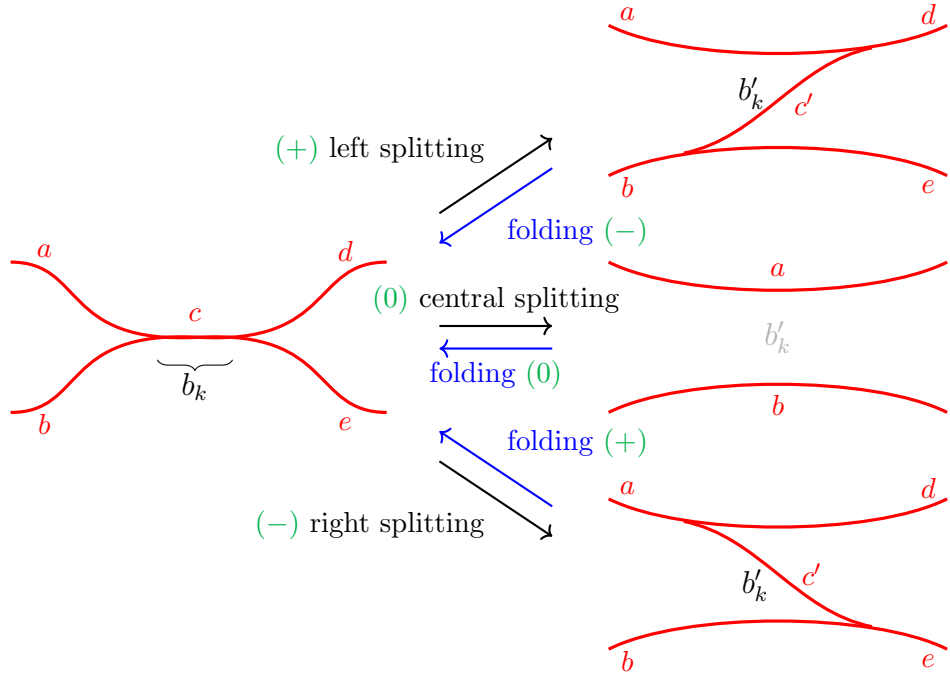
\begin{figure}[h]
        \centering
        \begin{tikzpicture}[auto]
        \draw [very thick, red](-4.95,1) .. controls (-3.95,1) and (-4.3,-0.1) .. (-2.45,0) .. controls (-0.65,-0.1) and (-0.95,1) .. (0.05,1);
        \draw [very thick, red] (3,4.15) .. controls (4,3.65) and (6.5,3.65) .. (7.5,4.15);
        \draw [very thick, red](3,2.15) .. controls (4,2.65) and (6.5,2.65) .. (7.5,2.15);
        \draw [very thick, red](4,2.45) .. controls (5,2.65) and (5.5,3.65) .. (6.5,3.85);
        \draw [very thick, red](3,1) .. controls (4,0.5) and (6.5,0.5) .. (7.5,1);
        \draw [thick, ->](0.75,1.65) -- node{\small \textcolor{mygreen}{$(+)$} left splitting } (2.25,2.65);
        \draw [thick, ->](0.75,0.15) --node{\small \textcolor{mygreen}{$(0)$} central splitting} (2.25,0.15);
        \draw [thick, ->](0.75,-1.65) --node[swap]{\small \textcolor{mygreen}{$(-)$} right splitting} (2.25,-2.65);
        \begin{scope}[yscale=-1]
        \draw [very thick, red](-4.95,1) .. controls (-3.95,1) and (-4.3,-0.1) .. (-2.45,0) .. controls (-0.65,-0.1) and (-0.95,1) .. (0.05,1);
        \draw [very thick, red] (3,4.15) .. controls (4,3.65) and (6.5,3.65) .. (7.5,4.15);
        \draw [very thick, red](3,2.15) .. controls (4,2.65) and (6.5,2.65) .. (7.5,2.15);
        \draw [very thick, red](4,2.45) .. controls (5,2.65) and (5.5,3.65) .. (6.5,3.85);
        \draw [very thick, red](3,1) .. controls (4,0.5) and (6.5,0.5) .. (7.5,1);
        \end{scope}
        \node [red] at (-4.5,1.15) {\small $a$};
        \node [red] at (-0.5,1.15) {\small $d$};
        \node [red] at (-2.5,0.3) {\small $c$};
        \node [red] at (-4.5,-1.15) {\small $b$};
        \node [red] at (-0.5,-1.15) {\small $e$};
        \node [red] at (3.25,4.3) {\small $a$};
        \node [red] at (7.25,4.3) {\small $d$};
        \node [red] at (5.6,3.1) {\small $c'$};
        \node [red] at (3.25,2) {\small $b$};
        \node [red] at (7.25,2) {\small $e$};
        \node [red] at (3.25,-2) {\small $a$};
        \node [red] at (7.25,-2) {\small $d$};
        \node [red] at (5.7,-3.15) {\small $c'$};
        \node [red] at (3.25,-4.3) {\small $b$};
        \node [red] at (7.25,-4.3) {\small $e$};
        \node [red] at (5.25,0.9) {\small $a$};
        \node [red] at (5.25,-0.9) {\small $b$};
        \draw [blue, thick, ->](2.25,2.25) --node {\small folding \textcolor{mygreen}{$(-)$}} (0.75,1.25);
        \draw [blue, thick, ->](2.25,-0.15) -- node {\small folding \textcolor{mygreen}{$(0)$}} (0.75,-0.15);
        \draw [blue, thick, ->](2.25,-2.25) --node[swap] {\small folding \textcolor{mygreen}{$(+)$}} (0.75,-1.25);
        \draw[decoration={brace, amplitude=4pt}, decorate] (-2,-0.3) -- (-3,-0.3);
        \node at (-2.5,-0.65) {$b_k$};
        \node at (4.9,3.25) {$b'_k$};
        \node at (4.95,-3.3) {$b'_k$};
        \node[gray!60] at (5.25,0) {$b'_k$};
        \end{tikzpicture}
        \caption{Splitting and folding.}
        \label{fig:split}
    \end{figure}
    \item Shifting:
    If the train track $\tau$ looks like the left or the right of \cref{fig:shift} around the branch $b_k$, then the deformation both of the left to the right and the right to left are called \emph{shifting}.
    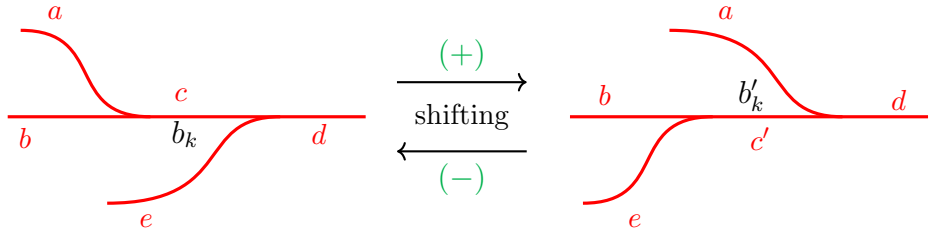
\begin{figure}[h]
        \centering
        \begin{tikzpicture}[scale=1.15, auto]
        \draw [very thick, red](-3,0) -- (1,0) node (v1) {} -- (v1) -- cycle;
        \draw [very thick, red](-1.35,0) .. controls (-2.35,0) and (-1.85,1) .. (-2.85,1);
        \draw [very thick, red](0.15,0) .. controls (-0.85,0) and (-0.35,-1) .. (-1.85,-1);
        \node [red] at (-2.45,1.2) {\small $a$};
        \node [red] at (-2.8,-0.25) {\small $b$};
        \node [red] at (-1.4,-1.2) {\small $e$};
        \node [red] at (-1,0.25) {\small $c$};
        \node [red] at (0.6,-0.2) {\small $d$};
        \begin{scope}[yscale=-1]
        \draw [very thick, red](3.5,0) -- (7.5,0) node (v1) {} -- (v1) -- cycle;
        \draw [very thick, red](5.15,0) .. controls (4.15,0) and (4.65,1) .. (3.65,1);
        \draw [very thick, red](6.65,0) .. controls (5.65,0) and (6.15,-1) .. (4.65,-1);
        \node [red] at (4.25,1.2) {\small $e$};
        \node [red] at (3.9,-0.25) {\small $b$};
        \node [red] at (5.3,-1.2) {\small $a$};
        \node [red] at (5.7,0.25) {\small $c'$};
        \node [red] at (7.3,-0.2) {\small $d$};
        \end{scope}
        \draw [thick, ->](1.5,0.4) -- node [mygreen] {$(+)$} (3,0.4) node [red] {};
        \draw [thick, ->](3,-0.4) -- node [mygreen] {$(-)$} (1.5,-0.4) node [red] {};
        \node at (2.25,0) {\small shifting};
        \node at (-0.95,-0.2) {$b_k$};
        \node at (5.6,0.25) {$b'_k$};
        \end{tikzpicture}
        \caption{Shifting.}
        \label{fig:shift}
    \end{figure}
\end{enumerate}
\end{defi}

\begin{lem}\label{lem:trans_mat_elem}
If the train track $\tau_1$ is obtained from $\tau_2$ by splitting or shifting at a branch $b_k \in B(\tau_2)$, then $\tau_1 \prec \tau_2$.
Moreover, there is a transition matrix of this carrying such that the corresponding linear map changes only the component which corresponds to the branch $b_k$, and the transformation is written by
\begin{align}\label{eq:trans_split}
c &= \begin{cases}
    a + c' + e & \mbox{if the move is a left splitting,}\\
    a + b & \mbox{if the move is a central splitting,}\\
    b + c' + d & \mbox{if the move is a right splitting,}
    \end{cases}\\
c &= a + b \qquad \mbox{if the move is a shifting}.\label{eq:trans_shift}
\end{align}
Here, the symbols $a,b,c,c',d,e$ mean some components of an element of $\bR^{B(\tau_1)}$ and $\bR^{B(\tau_2)}$, which are written in \cref{fig:split,fig:shift} by red letters.
\end{lem}
\begin{proof}
We can see the existence of the train track $\tau'_1$ which satisfies the conditions of the carrying $\tau_1 \prec \tau_2$ in \cref{fig:emb_carry}.
The explicit transformations \eqref{eq:trans_split} and \eqref{eq:trans_shift} are given by this embeddings and some suitable choice of $p_{b_1}$ for $b_1 \in B(\tau_1)$.
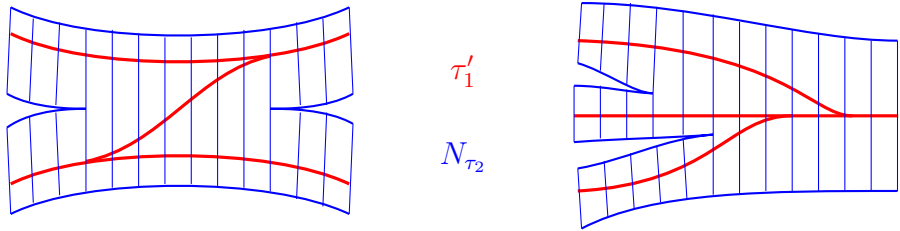
\begin{figure}[h]
    \centering
    \begin{tikzpicture}[baseline=85]
    \draw [very thick, red] (3,4.15) .. controls (4,3.65) and (6.45,3.65) .. (7.5,4.15);
    \draw [very thick, red](3,2.15) .. controls (4,2.65) and (6.45,2.65) .. (7.5,2.15);
    \draw [very thick, red](4,2.45) .. controls (5,2.65) and (5.5,3.65) .. (6.45,3.85);
    \draw [thick, blue](3,4.5) .. controls (4,4) and (6.45,4) .. (7.5,4.5);
    \draw [thick, blue](3,1.75) .. controls (4,2.25) and (6.45,2.25) .. (7.5,1.75);
    \draw [thick, blue](2.95,3.35) .. controls (3.3,3.15) and (3.85,3.15) .. (4,3.15) .. controls (3.85,3.15) and (3.3,3.15) .. (2.95,2.9);
    \draw [thick, blue](7.55,3.35) .. controls (7.2,3.15) and (6.65,3.15) .. (6.45,3.15) .. controls (6.65,3.15) and (7.2,3.15) .. (7.55,2.95);
    \draw [blue](4,4.2) -- (4,2.05);
    \draw [blue](4.35,4.2) -- (4.35,2.1);
    \draw [blue](4.7,4.15) -- (4.7,2.1);
    \draw [blue](5.05,4.1) -- (5.05,2.15);
    \draw [blue](5.4,4.1) -- (5.4,2.15);
    \draw [blue](6.45,4.2) -- (6.45,2.05);
    \draw [blue](5.75,4.15) -- (5.75,2.1);
    \draw [blue](6.1,4.2) -- (6.1,2.1);
    \draw [blue](6.8,4.3) -- (6.85,3.2);
    \draw [blue](6.85,3.1) -- (6.8,1.95);
    \draw [blue](7.15,4.35) -- (7.2,3.2);
    \draw [blue](7.2,3.1) -- (7.15,1.85);
    \draw [blue](7.5,4.5) -- (7.55,3.35);
    \draw [blue](7.55,2.95) -- (7.5,1.75);
    \draw [blue](3.65,4.3) -- (3.6,3.2);
    \draw [blue](3.6,3.1) -- (3.65,2);
    \draw [blue](3.3,4.35) -- (3.25,3.25);
    \draw [blue](3.25,3.05) -- (3.3,1.9);
    \draw [blue](3,4.5) -- (2.95,3.35);
    \draw [blue](2.95,2.9) -- (3,1.75);
    \end{tikzpicture}
    \qquad 
    $\begin{array}{c}
    \textcolor{red}{\tau'_1}\\\\
    \textcolor{blue}{N_{\tau_2}}
    \end{array}$
    \qquad
    \begin{tikzpicture}[yscale=-1, baseline=-2]
    \draw [very thick, red](-3.1,0) -- (1.05,0) node (v1) {} -- (v1) -- cycle;
    \draw [very thick, red](-0.2,0) .. controls (-1.05,0) and (-1.25,0.9) .. (-3.05,1);
    \draw [very thick, red](0.6,0) .. controls (0.05,0) and (-0.3,-0.9) .. (-3.05,-1);
    \draw [blue, thick](-3,1.5) .. controls (-2.05,1) and (-0.5,1) .. (1.2,1);
    \draw [blue, thick](-3,-1.5) .. controls (-1.5,-1.5) and (0,-1) .. (1.2,-1);
    \draw [blue, thick](-3.05,0.7) .. controls (-2.45,0.6) and (-1.9,0.3) .. (-1.25,0.25) .. controls (-1.9,0.3) and (-2.45,0.3) .. (-3.1,0.35);
    \draw [blue, thick](-3.1,-0.4) .. controls (-2.7,-0.4) and (-2.3,-0.3) .. (-2.05,-0.3) .. controls (-2.3,-0.3) and (-2.7,-0.6) .. (-3.05,-0.7);
    \draw [blue](-3,1.5) -- (-3.05,0.7);
    \draw [blue](-2.65,1.35) -- (-2.7,0.6);
    \draw [blue](-2.3,1.25) -- (-2.35,0.5);
    \draw [blue](-1.95,1.15) -- (-2,0.4);
    \draw [blue](-1.6,1.1) -- (-1.65,0.3);
    \draw [blue](-1.25,1.05) -- (-1.25,0.25);
    \draw [blue](-2.05,0.3) -- (-2.05,-0.3);
    \draw [blue](-2.4,0.3) -- (-2.4,-0.35);
    \draw [blue](-2.75,0.3) -- (-2.75,-0.4);
    \draw [blue](-3.1,0.35) -- (-3.1,-0.4);
    \draw [blue](-3.05,-0.7) -- (-3,-1.5);
    \draw [blue](-2.75,-0.6) -- (-2.7,-1.5);
    \draw [blue](-2.4,-0.4) -- (-2.35,-1.5);
    \draw [blue](-2.05,-0.3) -- (-2,-1.45);
    \draw [blue](-1.65,0.25) -- (-1.65,-1.4);
    \draw [blue](-1.25,0.25) -- (-1.25,-1.35);
    \draw [blue](-0.9,1.05) -- (-0.9,-1.25);
    \draw [blue](-0.55,1) -- (-0.55,-1.2);
    \draw [blue](-0.2,1) -- (-0.2,-1.15);
    \draw [blue](0.15,1) -- (0.15,-1.1);
    \draw [blue](0.5,1) -- (0.5,-1.05);
    \draw [blue](0.85,1) -- (0.85,-1);
    \draw [blue](1.2,1) -- (1.2,-1);
    \end{tikzpicture}
    \caption{The embeddings corresponding to the carrying $\tau_1 \prec \tau_2$.}
    \label{fig:emb_carry}
\end{figure}
\end{proof}


\begin{cor}\label{cor:cone_elem_move}
\begin{enumerate}
    \item Let the train tracks $\tau_R$, $\tau_L$ and $\tau_C$ are obtained from $\tau$ by the right, left and central splitting along a branch $b \in B(\tau)$, respectively.
    Then, we have $\cV(\tau) = \cV(\tau_R) \cup \cV(\tau_L)$ and $\cV(\tau_C) = \cV(\tau_R) \cap \cV(\tau_L)$.
    \item If the train track $\tau'$ is obtained from $\tau$ by shifting, then $\cV(\tau) = \cV(\tau')$ in $\MF(\Sigma)$.
\end{enumerate}
\end{cor}
\begin{proof}
(2): By \cref{lem:trans_mat_elem}, $\tau' \prec \tau$ and $\tau \prec \tau'$.
Thus, we have the inclusions of both directions between $\cV(\tau)$ and $\cV(\tau')$.

(1): Since $\tau_R, \tau_R \prec \tau$, we have $\cV(\tau_R) \cup \cV(\tau_L) \subset \cV(\tau)$.
If a measured foliation $(F, \mu)$ represents a point in $\cV(\tau)$, we can deform $F$ by cutting and opening along the singular leaves so that contained in $N_\tau$ whose leaves are transversal to the ties.
After cutting and opening it more around the branch $b$, the support is contained in $N_{\tau_R}$ or $N_{\tau_L}$ (or both of these).
It depends on the configuration of the singular leaves around $b$ that which one occurs.
Namely, there is no other possibilities, so we have $[F, \mu] \in \cV(\tau_R) \cup \cV(\tau_L)$.
We can see that $\cV(\tau_C) = \cV(\tau_R) \cap \cV(\tau_L)$ by the same argument.
\end{proof}

\subsection{Binary relations of train tracks suited to ideal triangulations arising from flips}
Let $\tau$ be a train track suited to $\tri$ an ideal triangulation of $\Sigma$.
Let us focus on a quadrilateral $Q$ consisting of two ideal triangles of $\tri$.
Also, let us consider the ideal triangle $\tri'$ which is obtained from $\tri$ by flipping along the diagonal arc $\alpha$ of $Q$.
Then, there are some possibilities of train tracks which is suited to $\tri'$ and agree with $\tau$ outside of $Q$.
In this subsection, we investigate these possibilities using the elementary moves of train tracks.

First, let us consider the train track $\tau \in \TT_\tri$ whose local model in the ideal triangles of $Q$ are both of type III.

\begin{lem}\label{lem:decomp_slide}
There is only one train track $\tau' \in \TT_{\tri'}$ which is obtained from $\tau$ by the sequence of elementary moves along the branches inside $Q$.
\end{lem}

\begin{proof}
Now, $\tau$ locally likes in the configuration as the left top of \cref{fig:decomp_slide}. 
Hence we can perform the elementary move along only the branch which is transversal to $\alpha$ since the other branches inside $Q$ do not appear in the local models of the elementary moves.
This move is left or right splitting, so we choose the left splitting, first.
For the train track shown in the right top of \cref{fig:decomp_slide}, we can perform only shifting at the branches $b_1$ and $b_2$ since the other possibility is nothing but the reversal operation of the first splitting.
After that, there is only one possibility of an elementary move, which is folding at $b'$.
Then, we get the train track $\tau'$ which is suited to the ideal triangulation $\tri' = f_\alpha(\tri)$.

Also, we can adapt the same argument for the case that we chose the right splitting.
In this case, also we get the same train track $\tau' \in \TT_{\tri'}$.
\end{proof} 

\begin{figure}[h]
\def\figurei{
\begin{tikzpicture}[auto, scale=0.9]
\begin{scope}[out=30,in=150,relative, shift={(-2,-0.5)}]
\draw[rotate=-90, very thick, red] (-1,0.5)
to (0,0.5) coordinate (v5) {} {}
to (-0.5,1.5) coordinate (v1) {} {}
to (-1,0.5) coordinate (v3) {} {};
\draw[rotate=90, very thick, red] (0,-3.5)
to (1,-3.5) coordinate (v9) {} {}
to (0.5,-2.5) coordinate (v2) {} {}
to (0,-3.5) coordinate (v7) {} {};
\end{scope}
\draw [very thick, red] (v1) to node {} (v2);
\node (v4) at (-2,1.5) {};
\node (v6) at (-2,-1.5) {};
\node (v8) at (2,-1.5) {};
\node (v10) at (2,1.5) {};
\draw [very thick, red] (v3) to node[swap] {} (v4);
\draw [very thick, red] (v5) to node {} (v6);
\draw [very thick, red] (v7) to node[swap] {} (v8);
\draw [very thick, red] (v9) to node {} (v10);
\node [fill, circle, inner sep=1.5pt] (v1) at (0,2) {};
\node [fill, circle, inner sep=1.5pt] (v3) at (0,-2) {};
\node [fill, circle, inner sep=1.5pt] (v2) at (-3,0) {};
\node [fill, circle, inner sep=1.5pt] (v4) at (3,0) {};
\draw  (v1) edge (v2);
\draw  (v2) edge (v3);
\draw  (v3) edge (v4);
\draw  (v4) edge (v1);
\draw  (v1) edge node[xshift=-14pt, yshift=18pt] {$\alpha$} (v3);
\node[red] at (0.25,-0.4) {$b$};
\node at (-0.55,-0.95) {$t_1$};
\node at (0.55,-0.95) {$t_2$};
\node at (0.5,2) {$p_1$};
\node at (-3,0.5) {$p_2$};
\node at (-0.5,-2) {$p_3$};
\node at (3,-0.5) {$p_4$};
\end{tikzpicture}}
\def\figureii{
\begin{tikzpicture}[scale=0.9]
\draw [very thick, red] (-1.9,1.3) 
.. controls (-1.4,0.75) and (-1.4,0.75) .. (0,0.75)
.. controls (1.5,0.75) and (1.5,0.75) .. (2,1.25);
\draw[rotate=180, very thick, red] (-2,1.2) 
.. controls (-1.5,0.7) and (-1.5,0.7) .. (0,0.7)
.. controls (1.5,0.7) and (1.5,0.65) .. (2,1.2);
\node [fill, circle, inner sep=1.5pt] (v1) at (0,2) {};
\node [fill, circle, inner sep=1.5pt] (v2) at (-3,0) {};
\node [fill, circle, inner sep=1.5pt] (v3) at (0,-2) {};
\node [fill, circle, inner sep=1.5pt] (v4) at (3,0) {};
\draw  (v1) edge (v2);
\draw  (v2) edge (v3);
\draw  (v3) edge (v4);
\draw  (v4) edge (v1);
\draw [very thick, red](-1.45,0.87) .. controls (-1,0.5) and (-1,-0.45) .. (-1.55,-0.8);
\draw [very thick, red](1.5,0.85) .. controls (1.05,0.5) and (1,-0.45) .. (1.45,-0.8);
\draw [very thick, red](0.55,0.75) .. controls (0,0.75) and (0,-0.7) .. (-0.55,-0.7);
\node[red] at (1,.9) {$b_1$};
\node[red] at (-1,-1.1) {$b_2$};
\end{tikzpicture}
}
\def\figureiii{
\begin{tikzpicture}[scale=0.9]
\begin{scope}[shift={(1.5,4.5)}]
\draw [very thick, red] (-4,-3.05) 
.. controls (-3.5,-3.6) and (-3,-3.55) .. (-1.5,-3.55)
.. controls (0,-3.6) and (0.5,-3.55) .. (1,-3.05);
\draw[rotate=180, very thick, red] (-1,6.05) 
.. controls (-0.5,5.5) and (0,5.55) .. (1.5,5.55)
.. controls (3,5.55) and (3.5,5.5) .. (4,6.05);
\draw[very thick, red] (0.3,-3.45) to (-3.3,-5.65);
\draw[very thick, red] (-2.8,-3.55) .. controls (-2,-3.55) and (-1.5,-4.5) .. (-2.2,-5);
\draw[rotate=180, very thick, red] (0.3,5.55) .. controls (1,5.55) and (1.5,4.5) .. (0.8,4.12);
\end{scope}
\node [fill, circle, inner sep=1.5pt] (v1) at (0,2) {};
\node [fill, circle, inner sep=1.5pt] (v2) at (-3,0) {};
\node [fill, circle, inner sep=1.5pt] (v3) at (0,-2) {};
\node [fill, circle, inner sep=1.5pt] (v4) at (3,0) {};
\draw  (v1) edge (v2);
\draw  (v2) edge (v3);
\draw  (v3) edge (v4);
\draw  (v4) edge (v1);
\node[red] at (0,-0.5) {$b'$};
\end{tikzpicture}}
\def\figureiv{
\begin{tikzpicture}[scale=0.9]
\begin{scope}[shift={(1.5,4.5)}]
\draw[very thick, red] (-2.5,-3.6) .. controls (-1.5,-3.8) and (-1.5,-4) .. (-1.5,-4.5);
\draw[rotate=180, very thick, red] (0.5,5.45) .. controls (1.5,5.2) and (1.5,5) .. (1.5,4.5);
\end{scope}
\draw [red, very thick](-2,1.5) .. controls (-1.4,0.55) and (1.4,0.55) .. (2,1.5);
\draw [red, very thick](-2,-1.5) .. controls (-1.4,-0.6) and (1.4,-0.6) .. (2,-1.5);
\draw[red, very thick] (1,0.9) .. controls (0,0.7) and (0,0.5) .. (0,0) .. controls (0,-0.5) and (0,-0.7) .. (-1,-0.95);
\node [fill, circle, inner sep=1.5pt] (v1) at (0,2) {};
\node [fill, circle, inner sep=1.5pt] (v2) at (-3,0) {};
\node [fill, circle, inner sep=1.5pt] (v3) at (0,-2) {};
\node [fill, circle, inner sep=1.5pt] (v4) at (3,0) {};
\draw  (v1) edge (v2);
\draw  (v2) edge (v3);
\draw  (v3) edge (v4);
\draw  (v4) edge (v1);
\draw  (v2) edge (v4);
\end{tikzpicture}}
    \centering
    \begin{tikzcd}[row sep=huge, column sep=huge]
    \figurei \ar[r, Rightarrow, red,  "\mbox{left splitting}", "\mbox{at $b$}"' font=\normalsize] \ar[d, Rightarrow, "f_\alpha" font=\normalsize
    ] & \figureii \ar[d, Rightarrow, red, "\mbox{shiftings}"', "\mbox{at $b_1$ and $b_2$}" font=\normalsize]\\
    \figureiv & \ar[l, Rightarrow, red, "\mbox{folding}"', "\mbox{at $b'$}" font=\normalsize] \figureiii
    \end{tikzcd}
    \caption{The decomposition of the deformation of train tracks arising from a flip into elementary moves.}
    \label{fig:decomp_slide}
\end{figure}
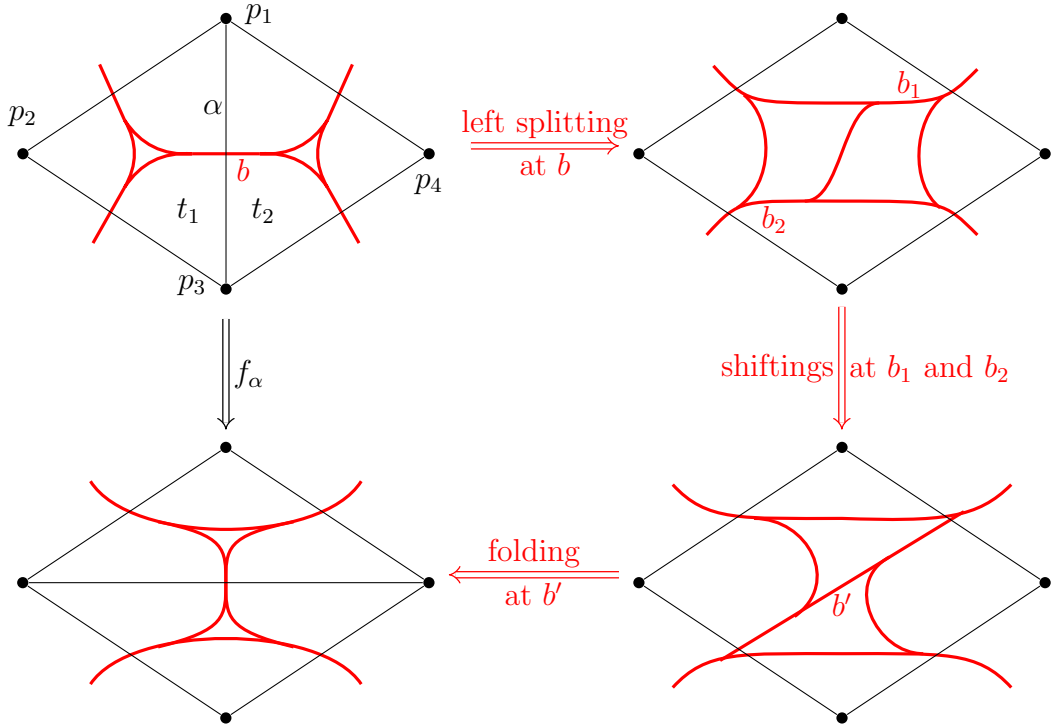

The cases that the train track $\tau$ is of type II in one or both of the ideal triangles of $Q$ are the degenerated cases of \cref{lem:decomp_slide}.
In order to summarize the other patterns, we define the symmetric binary relation between $\TT_{\tri}^{\max}$ and $\TT_{\tri'}^{\max}$ for $\tri, \tri' \in \Tri(\Sigma)$ such that $\tri' = f_\alpha(\tri)$.

First, we introduce the notion of labeled triangulations.
The pair $(\tri, \ell)$ of an ideal triangulation $\tri$ and a bijection $\ell: I := \{1, 2, \dots, -3 \chi(\Sigma) \} \to \tri$ is called \emph{labeled (ideal) triangulation}.
Let $\bTri(\Sigma)$ denote the graph whose vertices are labeled triangulations and adjacent vertices are related by a labeled flip or the action of a transposition of labelings in $\fS_I$.
For $k \in I$, $(\tri, \ell) \overbar{k} (\tri', \ell')$ denotes the edge of $\bTri(\Sigma)$, namely, $\tri' = f_{\ell(k)}(\tri)$ and $\ell(i) = \ell'(i)$ for $i \neq k$.
For $(\tri, \ell) \in \bTri(\Sigma)$, $\TT_{(\tri, \ell)} := \TT_\tri$ and $\TT^{\max}_{(\tri, \ell)} := \TT^{\max}_\tri$.

\begin{defi}
For $(\tri, \ell) \overbar{k} (\tri', \ell')$ in $\bTri(\Sigma)$, $\tau \in \TT^{\max}_{(\tri, \ell)}$ and $\tau' \in \TT^{\max}_{(\tri', \ell')}$, $\tau \lambda_k \tau'$ if they are related by a sequence of elementary moves along the branches inside the quadrangle of $\tri$ whose diagonal is $\ell(k)$.
It is clear that $\lambda_k$ is a symmetric binary relation between $\TT^{\max}_{(\tri, \ell)}$ and $\TT^{\max}_{(\tri', \ell')}$.
\end{defi}

\begin{thm}\label{thm:birel_TT}
Let $(\tri, \ell) \overbar{k} (\tri', \ell')$ in $\bTri(\Sigma)$ for $k \in I$ and let us take $\tau_0 \in \TT^{\max}_{(\tri, \ell)}$.
Then, the maximal length of the chains of binary relation $\lambda_k$ through $\tau_0$ is 1 or 2.
Namely, the possibilities of the chains are only $\tau_0 \lambda_k \tau'_0$, $\tau_0 \lambda_k \tau_0' \lambda_k \tau_1$ or $\tau'_0 \lambda_k \tau_0 \lambda_k \tau'_1$ for some $\tau_1 \in \TT^{\max}_{(\tri, \ell)}$ and $\tau'_0, \tau'_1 \in \TT^{\max}_{(\tri',\ell')}$.
Moreover, in each case, we have $\cV(\tau_0) = \cV(\tau'_0)$, $\cV(\tau_0) \cup \cV(\tau_1) = \cV(\tau'_0)$ and $\cV(\tau_0) = \cV(\tau'_0) \cup \cV(\tau'_1)$, respectively.
\end{thm}

\begin{proof}
Since $\tau_0$ is complete, $\tau_0$ is of type II or III in the ideal triangles of $\tri$ whose boundary contains $\ell(k)$.
Let $t_1$ and $t_2$ be ideal triangles such that $t_1 \cap t_2 = \ell(k)$.
Moreover, let $p_1, p_2, p_3, p_4$ be punctures of $\Sigma$ such that $p_1, p_2, p_3 \in t_1$, $p_1, p_3, p_4 \in t_2$.
See the left top of \cref{fig:decomp_slide}.


All the possible patterns are summarized in \cref{tab:deform_train_track}.
In this table, the row (resp. column) corresponding to the punctures means the absence of the short branch across to $p_i$ in $t_1$ (resp. $p_j$ in $t_2$).
The sign $\emptyset$ represents that the train track is of type III in $t_1$ or $t_2$.
The moves written in the cells are the possible sequence of elementary moves inside the quadrangle from $\tau_0$ to the train tracks which are suited to $\tau'$.
These moves are obtained from the sequence corresponding to the cell $(\emptyset, \emptyset)$ (\cref{lem:decomp_slide}) by omitting some moves which are not able to occur.
Since the train tracks corresponding to $(p_1, p_1)$ and $(p_3, p_3)$ are not complete, here are filled by the symbol $\times$.
We note that the moves written in the other cell appear only once in the sequence of moves corresponding to the cell.
By \cref{cor:cone_elem_move}, we can summarize as follows:
\begin{itemize}
    \item $(p_1, p_4)$, $(p_2, p_1)$, $(p_2, p_3)$, $(p_3, p_4)$ and $(\emptyset, \emptyset)$: the longest chain which contains $\tau_0$ is realized by $\tau_0 \lambda_k \tau'_0$ for some $\tau'_0 \in \TT^{\max}_{(\tri', \ell')}$.
    In this case, we have $\cV(\tau_0) = \cV(\tau'_0)$.
    \item $(p_1, p_3)$, $(p_1, \emptyset)$, $(p_3, p_1)$, $(p_3, \emptyset)$, $(\emptyset, p_1)$ and $(\emptyset, p_3)$: the longest chain which contains $\tau_0$ is realized by $\tau_0 \lambda_k \tau'_0 \lambda_k \tau_1$ for some $\tau_1 \in \TT^{\max}_{(\tri,\ell)}$ and $\tau'_0 \in \TT^{\max}_{(\tri', \ell')}$.
    In this case, we have $\cV(\tau_0) \cup \cV(\tau_1) = \cV(\tau'_0)$
    \item $(p_2, p_4)$, $(p_2, \emptyset)$ and $(\emptyset, p_4)$: the longest chain which contains $\tau_0$ is realized by $\tau'_0 \lambda_k \tau_0 \lambda_k \tau'_1$ for some $\tau'_0, \tau'_1 \in \TT^{\max}_{(\tri', \ell')}$.
    In this case, we have $\cV(\tau_0) = \cV(\tau'_0) \cup \cV(\tau'_1)$.
\end{itemize}

\end{proof}

\begin{table}[h]
    \centering
    \begin{tabular}{c|cccc}
    $t_1 \diagdown t_2$ & $p_1$ & $p_3$ & $p_4$ & $\emptyset$ \\
    \hline
    $p_1$ & $\times$  & fold & shift & shift \& fold \\
    $p_2$ & shift  & shift & split & split \& shift \\
    $p_3$ & fold  & $\times$ & shift & shift \& fold \\
    $\emptyset$ & shift \& fold  & shift \& fold  & split \& shift & split \& shift \& fold
    \end{tabular}
    \vspace{1.5mm}
    \caption{The possible sequence of elementary moves for the complete train tracks which are suited to an ideal triangulation $\tri$. }
    \label{tab:deform_train_track}
\end{table}

\subsection{Train track atlas of \texorpdfstring{$\MF(\Sigma)$}{the space of measured foliations}}

By \cref{cor:TT_fan}, it is clear that 
\begin{align}\label{eq:intD_subset_A}
    \bigcup_{\tau \in \TT^{\max}_\tri} \interior \widetilde{\cV}(\tau) \subsetneq \dMF(\Sigma).
\end{align}
However, by taking a union over the all ideal triangulation for the LHS of \eqref{eq:intD_subset_A}, we can recover the equality:
\begin{lem}\label{lem:open_cov_TT}
Let $\Sigma$ be a punctured surface, then we have
\begin{align*}
    \bigcup_{\tri \in \Tri(\Sigma)} \bigcup_{\tau \in \TT^{\max}_\tri} \interior \widetilde{\cV}(\tau) = \dMF(\Sigma) \setminus \{\mbox{\rm peripheral foliation}\}.
\end{align*}
Here, \emph{peripheral foliation} means the measured foliation whose every leaf is peripheral.
\end{lem}

We are not going to use this statement after \cref{sec:sign-stab} so one can skip the proof of it.

\begin{proof}
By the bijection of \cref{thm:TT_fan} and the proof of it, we are going to prove the following equivalent statement:
\begin{align*}
    \bigcup_{\tri \in \Tri(\Sigma)} \bigcup_{\tau \in \TT^{\max}_\tri} \interior \cV(\tau) = \MF(\Sigma) \setminus \{\emptyset\}.
\end{align*}
Let us take any point $\emptyset \neq \cF = [F, \mu] \in \MF(\Sigma)$.
Then, although $\cF \in \cV(\tau)$ for some $\tri \in \Tri(\Sigma)$ and $\tau \in \TT^{\max}_\tri$ by \cref{cor:TT_fan}, $\cF$ might be contained in the boundary of $\cV(\tau)$.
Let $\cK$ be the face of $\cV(\tau)$ which contains $\cF$ and whose codimension is maximal among such faces.
Since $V(\tau) \subset \bR^{B(\tau)}$ is polyhedral, such face $\cK$ is uniquely determined.
Let $K := \psi_\tau^{-1}(\cK) \subset V(\tau)$.
It is clear that the face $K$ is determined by the branches $B(\tau)_\cF = \{b_1, \dots, b_k\}$ of $\tau$ at which the measure is vanishing:
\begin{align*}
    K = \{ \nu \mid \nu(b) = 0 \mbox{ for } b \in B(\tau)_\cF\}.
\end{align*}

First, we assume that the set $B(\tau)_\cF$ contains long branches.
Let $\tri_\cF \subset \tri$ be the set of ideal arcs consisting of the dual of the long branches in $B(\tau)_\cF$.
Moreover, we consider the subgraph $G^\tri_\cF$ of the dual fat graph $G^\tri$ of $\tri$ corresponding to $\tri_\cF$.
We take $\alpha \in \tri_\cF$ such that the dual edge of it is an end of $G^\tri_\cF$.
Let $(\tri, \ell) \overbar{k} (\tri', \ell')$ and $\ell(k) = \alpha$.
Then, there is $\tau' \in \TT^{\max}_{(\tri',\ell')}$ such that $\tau \lambda_k \tau'$, $\cF \in \cV(\tau')$ by \cref{thm:birel_TT}.
Let $\cK'$ be the face of $\cV(\tau')$ which contains $\cF$ with maximal codimension and $K' := \psi_{\tau'}^{-1}(\cK')$.
Then, we have $\operatorname{codim} K' < \operatorname{codim} K$ since $\psi^{-1}_{\tau'}(\cF)(b') >0$ for the branch $b'$ of $\tau'$ which is transversal to $\ell'(k) \in \tri'$ since the dual edge of $\ell(k) = \alpha$ is an end of $G^\tri_\cF$.
(See the coordinate transformation \cref{fig:trop_A_flip} or \cref{lem:trans_mat_elem}.)
We repeat this procedure until the subgraph $G^{\tri''}_\cF$ vanishes where $\tri''$ is the final ideal triangulation.
Let $\tau'' \in \TT^{\max}_{\tri''}$ such that $\cF \in V(\tau'')$.
Then, all the branches of $B(\tau'')_\cF$ are short.

Next, we consider the case that all the branches of $B(\tau)_\cF$ are short.
Let us consider the subtrack $\tau_\cF$ of $\tau$ consists of the branches whose weight of $z$ are not zero.
If $B(\tau)_\cF \neq \emptyset$, there is a once-punctured $k >1$-gon as a complementary region of $\tau_\cF$.
Let $p$ be a puncture of such $k$-gon.
Then, flip the ideal arcs of $\tri$ adjacent to the puncture $p$ in any order and write $\tri''$ for the complete train track obtained from the corresponding binary relations such that $\cF \in V(\tau'')$.
Consider applying the elementary moves corresponding to the sequence of binary relations for the subtrack $\tau_\cF \subset \tau$, we can think that the $k$-gon of $\tau_\cF$ splits into $k-1$ trigons and one once-punctured monogon of $\tau''$. (See \cref{fig:punc5gon->4tri*punc1gon})
Hence $\tau''_\cF = \tau''$.
\end{proof}

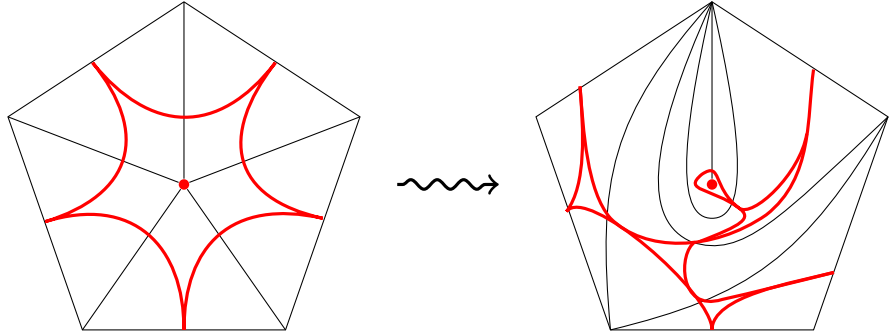
\begin{figure}[h]
    \centering
    \begin{tikzpicture}[scale=.9]
    \draw (0.35,0.7) coordinate (v1) {} -- (-2.25,-1) coordinate (v3) {} -- (-1.15,-4.15) coordinate (v4) {} -- (1.85,-4.15) coordinate (v5) {} -- (2.95,-1) coordinate (v6) {} -- (v1);
    \node [fill, circle, inner sep=1.4, red] (v2) at (0.35,-2) {};
    \draw  (v2) edge (v1);
    \draw  (v2) edge (v3);
    \draw  (v2) edge (v4);
    \draw  (v2) edge (v5);
    \draw  (v2) edge (v6);
    \draw [red, very thick](-1,-0.2) .. controls (-0.15,-1.25) and (-0.4,-2.15) .. (-1.7,-2.55) .. controls (-0.4,-2.15) and (0.35,-2.9) .. (0.35,-4.15) .. controls (0.35,-2.9) and (1.05,-2.2) .. (2.4,-2.5) .. controls (1.05,-2.2) and (0.85,-1.3) .. (1.7,-0.2) .. controls (0.85,-1.3) and (-0.15,-1.25) .. (-1,-0.2);
    
    \draw [squigarrow, very thick](3.5,-2) -- (5,-2);
    
    \draw (8.15,0.7) coordinate (v1) {} -- (5.55,-1) coordinate (v3) {} -- (6.65,-4.15) coordinate (v4) {} -- (9.65,-4.15) coordinate (v5) {} -- (10.75,-1) coordinate (v6) {} -- (v1);
    \node [fill, circle, inner sep=1.4, red] (v2) at (8.15,-2) {};
    \draw (v1) -- (v2);
    \draw (8.15,0.7) .. controls (7.65,-1.5) and (7.65,-2.5) .. (8.15,-2.5) .. controls (8.65,-2.5) and (8.65,-1.5) .. (8.15,0.7);
    \draw (8.15,0.7) .. controls (6.65,-1) and (6.45,-1.9) .. (6.65,-4.15);
    \draw (6.65,-4.15) .. controls (8.7,-3.65) and (9.5,-3.1) .. (10.75,-1);
    \draw (8.15,0.7) .. controls (7.15,-1.5) and (7.15,-2.55) .. (7.85,-2.85) .. controls (8.6,-3.1) and (9.15,-2.5) .. (10.75,-1);
    \draw [red, very thick](6.2,-0.55) .. controls (6.25,-1.2) and (6.3,-2) .. (6,-2.4) .. controls (6.3,-2) and (8.15,-3.6) .. (8.15,-4.15) .. controls (8.15,-3.65) and (9.35,-3.45) .. (9.95,-3.3);
    \draw [red, very thick](6.2,-0.55) .. controls (6.25,-1.2) and (6.3,-2.1) .. (6.65,-2.5) .. controls (7.15,-3) and (8.15,-3) .. (8.95,-2.45) .. controls (9.65,-1.9) and (9.55,-1) .. (9.65,-0.3);
    \draw [red, very thick](7.9,-2.85) .. controls (7.65,-3) and (7.65,-3.9) .. (8.3,-3.7) .. controls (8.75,-3.6) and (9.35,-3.45) .. (9.95,-3.3);
    \draw [red, very thick](9.55,-1.25) .. controls (9.45,-1.85) and (8.75,-2.5) .. (8.55,-2.35) .. controls (8.15,-2.05) and (8.3,-1.65) .. (8,-1.85) .. controls (7.7,-2.1) and (8.15,-2.05) .. (8.55,-2.35) .. controls (8.75,-2.5) and (8.6,-2.6) .. (7.95,-2.85);
    \end{tikzpicture}
    \caption{A deformation from once punctured pentagon to 4 unpunctured trigons and once punctured monogon.}
    \label{fig:punc5gon->4tri*punc1gon}
\end{figure}

For a complete train track $\tau$, let
\begin{align}\label{eq:tt_chart}
    \phi_\tau := (\psi_\tau|_{\interior V(\tau)})^{-1}: \interior \cV(\tau) \subset \MF(\Sigma) \to \interior V(\tau) \subset \bR^{B_\tau}.
\end{align}
Here, $B_\tau$ is the subset of the set of branches $B(\tau)$ of $\tau$ such that the weights at them form a basis of the $\bR$-span of $V(\tau)$.
Namely, 
\begin{align*}
    \mathrm{Span}_{\bR}(\nu(b) \mid b \in B(\tau),\ \nu \in V(\tau))=
    \mathrm{Span}_{\bR}(\nu(b) \mid b \in B_\tau,\ \nu \in V(\tau)).
\end{align*}

\begin{thm}\label{thm:tt_atlas_MF}
The set
\begin{align*}
    \{ (\interior \cV(\tau), \phi_\tau) \mid \tau \in \TT^{\max}_\tri,\ \tri \in \Tri(\Sigma) \}
\end{align*}
gives a PL atlas of $\MF(\Sigma) \setminus \{\emptyset\}$.
Therefore, the set
\begin{align*}
    \{ (\interior \widetilde{\cV}(\tau), \phi_\tau \times \mathrm{id}_{\bR^P}) \mid \tau \in \TT^{\max}_\tri,\ \tri \in \Tri(\Sigma) \}
\end{align*}
gives a PL atlas of $\widetilde{\MF}(\Sigma) \setminus \{\mbox{\rm peripheral foliations}\}$.
\end{thm}
One can prove this theorem directly by \cref{lem:open_cov_TT,lem:trans_mat_elem,thm:birel_TT} but we will prove in the next section by tropicalized cluster structure on $\dMF(\Sigma)$.

\begin{rmk}
Let $G_\tau$ be a graph obtained from a complete train track $\tau$ by collapsing every complementary region which is trigon to a trivalent vertex.
Then, the dual graph $\tri_\tau$ of $G_\tau$ is an ideal triangulation of $\Sigma$.
Namely, $\tau \in \TT^{\max}_{\tri_\tau}$.
Therefore, the union $\bigcup_{\tri \in \Tri(\Sigma)} \TT^{\max}_{\tri}$ is nothing but the set of the complete train tracks on $\Sigma$, so we can rewrite \cref{thm:tt_atlas_MF} as follows:
the set $\{ (\interior \cV(\tau), \phi_\tau) \mid \tau \mbox{ is a complete train track on } \Sigma \}$ gives a PL atlas of $\MF(\Sigma) \setminus \{\emptyset\}$.
This statement is already proven by W. P. Thurston \cite[Section 8.10]{Th}.
Our strategy of the proof of \cref{thm:tt_atlas_MF} is ''semi''-algebraic so some techniques might be generalized to the other mutation classes.
\end{rmk}

\section{Goncharov--Shen potential and $\cV$-variety}\label{sec:GSpot_Vvar}
In the last section, we see that there is a PL structure on $\dMF(\Sigma)$ given by the train tracks which is suited to ideal triangulations.
On the other hand, $\dMF(\Sigma)$ has a PL structure given by $\cA$-coordinates associated with the ideal triangulations.
It is nothing but the PL structure given by the tropicalization of the cluster structure associated with the ideal triangulations of $\Sigma$.
In this section, we translate train tracks which suited to ideal triangulations into the language of cluster algebras.



\subsection{Cluster ensemble associated with a punctured surface}
First, we recall the mutation class $\bs_\Sigma$ of a punctured surface $\Sigma$ and some properties of the cluster ensemble $p_\Sigma: \cA_\Sigma \to \cX_\Sigma$ associated with $\bs_\Sigma$.


\begin{thm}[{\cite{FST}}, see {\cite[Theorem 4.1]{IK20a}}]\label{thm:Tri_Exch}
If $\Sigma$ is a marked surface except for a once-punctured surface, there is a graph embedding
\begin{align*}
    \bTri(\Sigma) \hookrightarrow \bExch_\Sigma.
\end{align*}
Here, $\bExch_\Sigma$
denotes the labeled exchange graph
of the mutation class $\bs_\Sigma$.
\end{thm}
The labeled exchange graph is realized by the graph of labeled ``tagged'' triangulations. (See \cite{FST}.)
We sometimes identify $\tri \in \Tri(\Sigma)$ and some lift $(\tri, \ell) \in \bTri(\Sigma)$ when we focus only local chart associated with $(\tri, \ell)$.

Comparing with the coordinate change of the $\cA$-coordinate on the space of decorated measured foliations and the cluster coordinate of tropicalized cluster $\cA$-variety $\cA_\Sigma(\bR^\trop)$, we have the identification $\dMF(\Sigma) \xrightarrow{\sim} \cA_\Sigma(\bR^\trop)$ commutes the following:
\[
\begin{tikzcd}
\dMF(\Sigma) \ar[r, "\sim"] \ar[d, "\sfa_\alpha"'] & \cA_\Sigma(\bR^\trop) \ar[d, "a^{(\tri,\ell)}_i"]\\
\bR \ar[r, "(-1)"] & \bR.
\end{tikzcd}
\]
Here, $(\tri, \ell) \in \bTri(\Sigma)$ and $\ell(i) = \alpha \in \tri$.
Moreover, this identification induces the identification $\MF(\Sigma) \xrightarrow{\sim} \cU_\Sigma(\bR^\trop)$ which commutes the following:
\[
\begin{tikzcd}
\dMF(\Sigma) \ar[r, "\sim"] \ar[d] & \cA_\Sigma(\bR^\trop) \ar[d, "p_\Sigma"]\\
\MF(\Sigma) \ar[r, "\sim"] & \cU_\Sigma(\bR^\trop).
\end{tikzcd}
\]
Here, the left vertical map is the natural projection of the trivial bundle $\dMF(\Sigma) = \MF(\Sigma) \times \bR^P$.

\subsection{Goncharov--Shen potential and $\cV$-variety}
The Goncharov--Shen potential is defined in \cite{GS15} for the moduli space $\cA_{G, \Sigma}$ of decorated twisted $G$-local systems on $\Sigma$, where $G$ is a split semisimple simply-connected group $G$ over $\bQ$ but we will mention only for $G = SL_2$.

First, we recall some basic properties of cluster $\cA$-varieties.
Let $\bs$ be a general mutation class.
We note that there are natural identifications $\coker p^*_{(v)} \cong \coker p^*_{(v')}$ for $v, v' \in \bExch_\bs$, so we write $\coker p^*$ them simply.
Let $H_{\cA, \bs}$ denote the torus associated with the dual lattice of $\operatorname{coker}p^*$:
\begin{align*}
    H_{\cA, \bs} := T_{(\operatorname{coker}p^*)^*} := \Hom(\operatorname{coker}p^*, \bG_m).
\end{align*}
One can think that $(\operatorname{coker}p^*)^* = N^{(v)}/\im p^*_{(v)} \subset N^{(v)}$ for some $v \in \bExch_\bs$.
The additive action $(\coker p^*) \otimes N^{(v)} \to N^{(v)}$ induces an action on $\cA_\bs$ by the torus $H_{\cA, \bs}$ (\cite[Section 2.3]{FG09}).

\begin{prop}[{\cite[Lemma 2.10]{FG09}}]
The ensemble map $p: \cA_\bs \to \cU_\bs$ is an $H_{\cA, \bs}$-torsor.
\end{prop}


The Goncharov--Shen potential gives a ``section'' of this torsor in some sense.

In what follows, we fix $\bs = \bs_\Sigma$ for a punctured surface $\Sigma$.
For each ideal triangulation $\tri$ and each puncture $p \in P$, we define
\begin{align}\label{eq:pot_gen}
W^{\tri}_p = \sum_{t \in T_{\tri, p}} \frac{A^0_{p,t}}{A^1_{p,t}A^2_{p,t}}
\end{align}
where $A^i_{p,t}$ is $A_{\alpha^i_{p,t}}^{\tri}$ and $\alpha^i_{p,t}$ are edges of triangles $t \in T_{\tri, p}$ (see \cref{fig:t_p}).
The maps $W^{\tri}_p: \cA_{\tri} \to \bA^1$ glue and determine the map $W_p: \cA_\Sigma \to \bA^1$.
We write $W = (W_p)_{p \in P}: \cA_\Sigma \to (\bA^1)^P$ and call it \emph{Goncharov--Shen potential function} (\emph{GS potential} for short).
Let $\cV_\Sigma$ denote $W^{-1}(1) \subset \cA_\Sigma$ for the unit $1 \in (\bA^1)^P$ and we call it \emph{$\cV$-variety}.

Since the GS potential $W$ is a positive rational map, namely it has an expression without subtractions, we can take the tropicalization at the semifield $\bP$.
So we can take the tropicalization of the $\cV$-varieties $\cV_\Sigma(\bP)$ by definition $(W^\bP)^{-1}(e)$, where $W^\bP$ is the tropicalization of $W$ at $\bP$ and $e \in \bP^\times$ denotes the multiplicative unit.

\begin{prop}\label{prop:V=U}
Let $\bP = \bR^\trop$ or $\bR_{>0}$.
Then, the ensemble map induces an isomorphism $\cV_\Sigma(\bP) \xrightarrow{\sim} \cU_\Sigma(\bP)$.
\end{prop}
We can prove it by using the proof of \cite[Theorem 6.2]{GS15} in the more general setting, the mutation class $\bs$ is obtained from the pair $(G, \Sigma)$ of a semisimple group $G$ and a marked surface $\Sigma$.
Here, we give a geometric proof for $\bP = \bR^\trop$.
The case $\bP = \bR_{>0}$ is proven by the identification $\cA_\Sigma(\bR_{>0})$ and the decorated \Teich\ space (\cite{Pen}).

First, we see the geometric description of the tropicalized GS potential $w:= W^{\bR^\trop}: \cA_\Sigma(\bR^\trop) \to \bR^P$.


\begin{lem}
Under the identification $\dMF(\Sigma) \cong \cA_\Sigma(\bR^\trop)$, we have $w = \pot$.
\end{lem}



\begin{proof}
For $\tri \in \Tri(\Sigma)$ and $p \in P$,
\begin{align*}
\pot_p
&= \min\big\{\sfa^1_{p, t} + \sfa^2_{p,t} - \sfa^0_{p,t} \ \big|\ t \in T_{\tri, p} \big\}\\
&= \min \big\{ a^0_{p,t} - a^1_{p, t} - a^2_{p,t} \ \big|\ t \in T_{\tri, p}\big\} = w_{\tri, p} \circ a_\tri = w_p.
\end{align*}
\end{proof}

Therefore, we get
\begin{align*}
    \cU_\Sigma(\bR^\trop) \cong \MF(\Sigma) \cong \pot^{-1}(0) \cong w^{-1}(0) \cong \cV_\Sigma(\bR^\trop).
\end{align*}


\begin{rmk}
The map $\pot$ is the minus of the max-plus tropicalization of the GS potential $W$.
\end{rmk}

By the identification of the map $\pot$ of $\dMF(\Sigma)$ and $w$ of $\cA_\Sigma(\bR^\trop)$, we can translate the results \cref{thm:TT_fan,thm:tt_atlas_MF}:
\begin{cor}\label{cor:TT_D}
There is a bijection $\TT^{\max}_\tri \xrightarrow{\sim} D_\tri(w)$.
Here, $D_\tri(w)$ denotes the set of maximal domains of linearity of $w$ in the cluster $\cA$-coordinate $\mathbf{a}^\tri$.
\end{cor}

\begin{ex}
Here we give the example of the GS potential of the sphere with 4 punctures and its tropicalization.
Let us consider the following labeled triangulation $(\tri, \ell)$
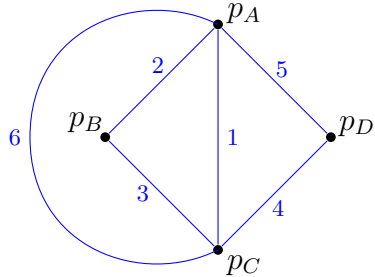
\begin{center}
\begin{tikzpicture}
    \node [fill, circle, inner sep=1.3] (v1) at (-0.35,2.85) {};
    \node [fill, circle, inner sep=1.3] (v3) at (-0.35,-0.15) {};
    \node [fill, circle, inner sep=1.3] (v2) at (-1.85,1.35) {};
    \node [fill, circle, inner sep=1.3] (v4) at (1.15,1.35) {};
    \draw [blue](v1) -- (v2) -- (v3) -- (v4) -- (v1) -- (v3);
    \draw [blue](-0.35,-0.15) .. controls (-1.35,-0.65) and (-2.85,-0.15) .. (-2.85,1.35) .. controls (-2.85,2.85) and (-1.35,3.35) .. (-0.35,2.85);
    \node [blue] at (-0.15,1.35) {\scriptsize 1};
    \node [blue] at (-1.15,2.3) {\scriptsize 2};
    \node [blue] at (-1.35,0.6) {\scriptsize 3};
    \node [blue] at (0.45,0.4) {\scriptsize 4};
    \node [blue] at (0.5,2.25) {\scriptsize 5};
    \node [blue] at (-3.05,1.35) {\scriptsize 6};
    \node at (0,3) {$p_A$};
    \node at (-2.1,1.55) {$p_B$};
    \node at (0,-0.35) {$p_C$};
    \node at (1.5,1.5) {$p_D$};
\end{tikzpicture}
\end{center}
and take a labeling for the punctures as above.
Then, 
\begin{align*}
    W^\tri_{p_A} &= \frac{A_3}{A_1 A_2} + \frac{A_3}{A_2 A_6} + \frac{A_4}{A_1 A_5} + \frac{A_4}{A_5 A_6},\\
    W^\tri_{p_B} &= \frac{A_1}{A_2 A_3} + \frac{A_6}{A_2 A_3},\\
    W^\tri_{p_C} &= \frac{A_2}{A_1 A_3} + \frac{A_2}{A_3 A_6} + \frac{A_5}{A_1 A_4} + \frac{A_5}{A_4 A_6},\\
    W^\tri_{p_D} &= \frac{A_1}{A_4 A_5} + \frac{A_6}{A_4 A_5}.
\end{align*}
Its tropicalizations are
\begin{align*}
    w^\tri_{p_A} &= \min\{ a_3 - a_1 - a_2,\, a_3 - a_2 - a_6,\, a_4 - a_1 - a_5,\, a_4 - a_5 - a_6\}, \\
    w^\tri_{p_B} &= \min\{ a_1 - a_2 - a_3,\, a_6 - a_2 - a_3 \},\\
    w^\tri_{p_C} &= \min\{ a_2 - a_1 - a_3,\, a_2 - a_3 - a_6,\, a_5 - a_1 - a_4,\, a_5 - a_4 - a_6 \},\\
    w^\tri_{p_D} &= \min\{ a_1 - a_4 - a_5,\, a_6 - a_4 - a_5 \}.
\end{align*}
The subspace $\cV_\Sigma(\bR^\trop) \subset \cA_\Sigma(\bR^\trop)$ is defined by the PL equations $w^\tri_{p_A} = w^\tri_{p_B} = w^\tri_{p_C} = w^\tri_{p_D} = 0$.
If $w^\tri_{p_A} = a_3 - a_1 - a_2 = 0$, then we have
\begin{align*}
a_1 - a_6 \geq 0,\ a_4 - a_1 - a_5 \geq 0,\ a_4 - a_5 - a_6 \geq 0
\end{align*}
from $w^\tri_{p_A} \geq 0$.
Using this (in)equalities, we get
\begin{align*}
    w^\tri_{p_B} = a_6 - a_2 - a_3 = 0,\ 
    w^\tri_{p_C} = a_5 - a_1 - a_4 = 0,\ 
    w^\tri_{p_D} = a_6 - a_4 - a_5 = 0.
\end{align*}
From them, the points of $\cV_{\Sigma}(\bR^\trop)$ satisfying $w^\tri_{p_A} = a_3 - a_1 - a_2 = 0$ is parametrized by $a_2 + a_4 - a_3 - a_5$ and $a_1 - a_6$ and they satisfies the inequalites
\begin{align*}
    a_2 + a_4 - a_3 - a_5 \geq 0,\ a_1 - a_6 \geq 0.
\end{align*}
The other cases $w^\tri_{p_A} =a_3 - a_2 - a_6 = 0$, $w^\tri_{p_A} = a_4 - a_1 - a_5 = 0$ and $w^\tri_{p_A} = a_4 - a_5 - a_6 = 0$ are also parametrized by $a_2 + a_4 - a_3 - a_5$ and $a_1 - a_6$ and they satisfies some inequalities.
We can summarize them as \cref{fig:PL_str_V_4_sph}.

\begin{figure}[h]
    \centering
    \hspace{2.8cm}
    \begin{tikzpicture}
    \draw (0,-2) -- (6,-2);
    \draw (3,1)  -- (3,-5);
    
    \node at (1,-0.5) {$\begin{array}{c}
        a_2 + a_4 \geq a_3 + a_5\\
        a_1 \leq a_6  
        \end{array}$};
    \node at (5,-0.5) {$\begin{array}{c}
        a_2 + a_4 \geq a_3 + a_5\\
        a_1 \geq a_6  
        \end{array}$};
    \node at (1,-3.5) {$\begin{array}{c}
        a_2 + a_4 \leq a_3 + a_5\\
        a_1 \leq a_6  
        \end{array}$};
    \node at (5,-3.5) {$\begin{array}{c}
        a_2 + a_4 \leq a_3 + a_5\\
        a_1 \geq a_6  
        \end{array}$};
    
    \node at (3,1.5) {$a_1=a_6$};
    \node at (8,-2) {$a_2 + a_4 = a_3 + a_5$};
\end{tikzpicture}
    \caption{The domains of linearity of the tropicalized GS potential of the sphere with 4 punctures.}
    \label{fig:PL_str_V_4_sph}
\end{figure}
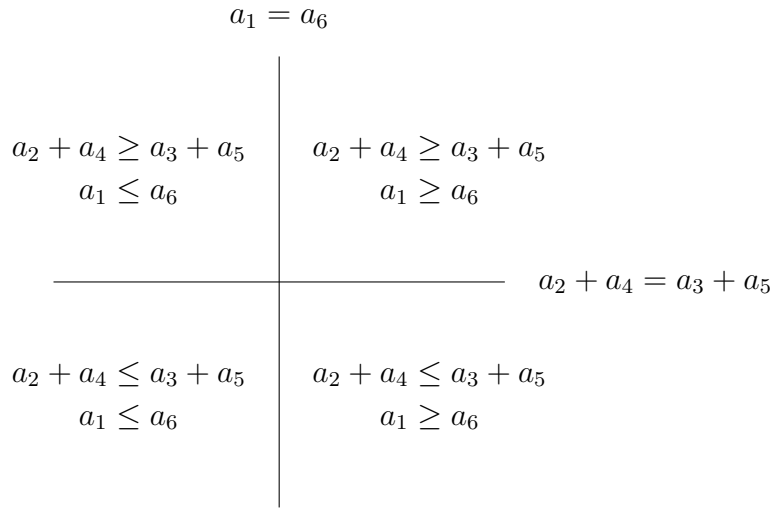

One can verify that each domain in this figure corresponds to the domain in the same position of \cref{fig:tt_fan_4shp}.
We note that $a_i = - \sfa_i$.
\end{ex}

Let us identify $\cV_\Sigma(\bR^\trop) \subset \cA_\Sigma(\bR^\trop)$ and $\MF(\Sigma) \subset \dMF(\Sigma)$, respectively.
By \cref{thm:tt_atlas_MF}, we have the PL atlas on $\cV_\Sigma(\bR^\trop) \setminus \{0\}$ and $\cA_\Sigma(\bR^\trop) \setminus p_\Sigma^{-1}(0)$ given by the complete train tracks which suited to ideal triangulations.

\begin{thm}\label{thm:atlases}
Let $\Sigma$ be a punctured surface.
Then, the PL structures on $\cV_\Sigma(\bR^\trop) \setminus \{0\}$ and $\cA_\Sigma(\bR^\trop) \setminus p_\Sigma^{-1}(0)$ given by complete train tracks and tropicalized cluster structures are equivalent.
\end{thm}
\begin{proof}
It is enough to show that the piecewise linearity of the coordinate change from the chart given by the tropicalized cluster structure to the chart given by a complete train track.
Let $\tau$ be a complete train track that is suited to an ideal triangulation $\tri$ of $\Sigma$.
Then, our aim is to prove the map
\begin{align}\label{eq:coord_change}
    (\phi_\tau \times \mathrm{id}_{\bR^P}) \circ (a_\tri|_{K(\tau)})^{-1}: \bR^\tri \to \bR^{B_\tau} \times \bR^P
\end{align}
is piecewise linear.
Here, $K(\tau) := {a}_\tri(\interior \widetilde{\cV}(\tau))$ and $B_\tau$ is the subset of $B(\tau)$ as in \eqref{eq:tt_chart}.
Let $B(\tau_\tri, \tau) \subset B(\tau_\tri)$ consist of the short branches of the freeway $\tau_\tri$ such that the graph obtained from $\tau_\tri$ by removing the branches in $B(\tau_\tri, \tau)$ coincides with the train track $\tau$.
We note that the extra map $\mathrm{id}_{\bR^P}$ corresponds to the set $B(\tau_\tri, \tau)$ when we think that $\widetilde{V}(\tau)$ is the set of weights of the freeway $\tau_\tri$.
Also, the map $\mathbf{a}^\tri$ is nothing but the weight at the long branches of $\tau_\tri$.
Thus, the map \eqref{eq:coord_change} is the restriction of the base change map from the basis given by the long branches of $\tau_\tri$ to the basis given by the branches $B_\tau \cup B(\tau_\tri, \tau)$, so especially this map is linear.
\end{proof}

\begin{proof}[Proof of \cref{thm:tt_atlas_MF}]
We see that the coordinate change from the chart given by a complete train track $\tau_1$ to the chart given by the other complete train track $\tau_2$ is PL.
We note that
\begin{align*}
    (\phi_{\tau_2} \circ \phi_{\tau_1}^{-1}) |_{\phi_{\tau_1}(\interior \widetilde{\cV}(\tau_1) \cap  \interior\widetilde{\cV}(\tau_2))} =
    (\phi_{\tau_2} \circ a_{\tri_2}^{-1} \circ a_{\tri_2} \circ a_{\tri_1}^{-1} \circ a_{\tri_1} \circ \phi_{\tau_1}^{-1}) |_{\phi_{\tau_1}(\interior \widetilde{\cV}(\tau_1) \cap  \interior\widetilde{\cV}(\tau_2))}\ .
\end{align*}
Here, RHS of the above equation is the composition of PL maps $\phi_{\tau_2} \circ a_{\tri_2}^{-1}$, $a_{\tri_2} \circ a_{\tri_1}^{-1}$ and $a_{\tri_1} \circ \phi_{\tau_1}^{-1}$, thus itself is a PL map.
By combining \cref{lem:open_cov_TT}, the desired statement is proven.
\end{proof}

\part{Train track splittings and tropical cluster transformations}\label{part:SS}
In this second half part, we prove the sign stability of the general pseudo-Anosov mapping classes by using a specific sequence of splittings and shiftings of the ``invariant track''.
The observation about the binary relation $\lambda_k$ between $\TT^{\max}_{(\tri, \ell)}$ and $\TT^{\max}_{(\tri', \ell')}$ for a horizontal edge $(\tri, \ell) \overbar{k} (\tri', \ell')$ in $\bTri(\Sigma)$, summarized in \cref{tab:deform_train_track}, plays a central role here.
\section{Sign-stable mutation loops}\label{sec:sign-stab}
In this section, we give a brief review of the notion of sign-stability of a representation path of a mutation loop from \cite{IK21}.

Let $\bs$ be a mutation class.
First, we recall the tropical cluster $\X$-transformation $\mu_\indk: \X_{(v)}(\bR^\trop) \to \X_{(v')}(\bR^\trop)$ associated with an edge $v \overbar{k} v'$ in $\bExch_\bs$.
For a real number $a \in \bR$, let $\sgn(a)$ denote its sign:
\[
\sgn(a):=
\begin{cases}
    + & \mbox{ if } a>0,\\
    0 & \mbox{ if } a=0,\\
    - & \mbox{ if } a<0.
\end{cases}
\]
The following expression is useful in the sequel:

\begin{lem}[{\cite[Lemma 3.1]{IK21}}]\label{l:x-cluster signed}
The tropical cluster $\X$-transformation $\mu_\indk: \X_{(v)}(\bR^\trop) \to \X_{(v')}(\bR^\trop)$ is given by
\begin{align}\label{eq:sign x-cluster}
    \mu_\indk^* x^{(v')}_\indi =
\begin{cases}
    -x^{(v)}_\indk & \mbox{if $\indi=\indk$}, \\
    x^{(v)}_\indi(w)+[\sgn(x^{(v)}_\indk)b^{(v)}_{\indi\indk}]_+x^{(v)}_\indk & \mbox{if $\indi \neq \indk$}.
\end{cases}
\end{align}
\end{lem}


\begin{dfn}[sign of a path]\label{d:sign}
Let $\gamma: v_0 \overbar{k_0} v_1 \overbar{k_1} \cdots \overbar{k_{m-1}} v_m$ be a path in $\bExch_\bs$ and we fix a point $w \in \X_{\bs}(\bR^\trop)$.
\item The \emph{sign} of $\gamma$ at $w$ is the sequence $\boldsymbol{\epsilon}_\gamma(w)=(\epsilon_0,\dots,\epsilon_{h-1}) \in \{+,0,-\}^h$ of signs defined by
\[\epsilon_{\nu} := \sgn(x^{(v_{i(\nu)})}_{k_{i(\nu)}}(w))\]
for $\nu=0,\dots,h-1$.
Here, $(k_{i(0)}, k_{i(1)}, \dots, k_{i(h-1)})$ is the subsequence of $(k_0, k_1, \dots, k_m)$ consists of horizontal edges.
\end{dfn}

\begin{ex}\label{ex:sign_geom}
For $[F,\mu] \in \MF(\Sigma) \cong \cV_\Sigma(\bR^\trop) \xrightarrow{\sim} \cU_\Sigma(\bR^\trop)$, we have $x_\alpha^{\tri}([F,\mu]) = \sum_{\beta} b^{\tri}_{\alpha \beta} a_\beta([F,\mu])$.
Thus, $\sgn(x_\alpha^\tri[F, \mu]) = +$ (resp. $-$) if and only if its support likes in the configuration in the left (resp. right) of \cref{f:shear}.
Namely, $x^\tri_\alpha([F, \mu])$ is given by $\int_\gamma \mu$ (resp. $-\int_\gamma \mu$) if $\sgn(x_\alpha^\tri([F, \mu])) = +$ (resp. $=-$).
Here, $\gamma$ is the transversal arc which connects the singular leaves emanating from the singular points in the ideal triangles whose side contains $\alpha$.
\begin{figure}[h]
    \begin{tikzpicture}[scale=0.87]
    \node [fill, circle, inner sep=1.3] (v1) at (0,2.5) {};
    \node [fill, circle, inner sep=1.3] (v3) at (0,-2.5) {};
    \node [fill, circle, inner sep=1.3] (v2) at (-4,0) {};
    \node [fill, circle, inner sep=1.3] (v4) at (4,0) {};
    \draw [blue](v1) -- (v2) -- (v3) -- (v4) -- (v1) -- (v3);
    \draw [thick](-1.75,0.35) .. controls (-2.2,-0.35) and (-2.3,-0.65) .. (-2.45,-0.95);
    \draw [thick](-1.75,0.35) .. controls (-2.2,0.8) and (-2.2,0.8) .. (-2.4,1);
    \draw [thick](-1.75,0.35) .. controls (-1.15,0.45) and (-0.1,1.05) .. (1.2,1.75);
    \draw [thick](1.2,-0.35) .. controls (0,-0.75) and (-0.65,-1.4) .. (-1.05,-1.85);
    \draw [thick](1.2,-0.35) .. controls (1.9,0.4) and (2.3,0.7) .. (2.55,0.9);
    \draw [thick](1.2,-0.35) .. controls (1.65,-0.75) and (1.9,-0.95) .. (2.1,-1.2);
    \draw (-1.95,1.3) .. controls (-1.45,0.9) and (-1.5,0.55) .. (1,1.9);
    \draw (-1.5,1.55) .. controls (-1.15,1.3) and (-1,1.15) .. (0.65,2.05);
    \draw (-1.05,1.85) .. controls (-0.8,1.65) and (-0.15,1.9) .. (0.4,2.25);
    \draw (-2.75,0.75) .. controls (-2.35,0.45) and (-2.35,0.05) .. (-2.85,-0.7);
    \draw (-3.15,0.5) .. controls (-2.95,0.3) and (-2.9,0) .. (-3.2,-0.5);
    \draw (-3.6,0.25) .. controls (-3.45,0.1) and (-3.45,-0.05) .. (-3.55,-0.25);
    \draw (-2.15,-1.15) .. controls (-1.5,0) and (-1.3,0) .. (1.45,1.6);
    \draw (-1.8,-1.35) .. controls (-1.15,-0.3) and (-0.25,0.3) .. (1.7,1.4);
    \draw (-1.55,-1.55) .. controls (-0.7,-0.45) and (0.25,0.1) .. (2,1.25);
    \draw (-1.3,-1.7) .. controls (-0.25,-0.3) and (1.15,-0.05) .. (2.25,1.1);
    \draw (-0.75,-2) .. controls (0.6,-0.75) and (1,-0.85) .. (1.7,-1.45);
    \draw (-0.5,-2.2) .. controls (0.25,-1.4) and (0.95,-1.45) .. (1.2,-1.75);
    \draw (-0.25,-2.3) .. controls (0,-2) and (0.55,-1.9) .. (0.75,-2.05);
    \draw (2.8,0.75) .. controls (1.75,-0.35) and (1.75,-0.35) .. (2.45,-0.95);
    \draw (3.1,0.6) .. controls (2.35,-0.2) and (2.35,-0.35) .. (2.8,-0.75);
    \draw (3.4,0.4) .. controls (2.95,-0.1) and (2.95,-0.25) .. (3.15,-0.55);
    \draw (3.7,0.2) .. controls (3.5,0) and (3.4,-0.15) .. (3.5,-0.3);
    \draw[red, thick] (-0.5,0.85) .. controls (-0.4,0.3) and (0.4,0) .. (0.55,-0.6);
    \node[red] at (0.25,0.3) {$\gamma$};
    \begin{scope}[xshift=280, xscale=-1]
    \node [fill, circle, inner sep=1.3] (v1) at (0,2.5) {};
    \node [fill, circle, inner sep=1.3] (v3) at (0,-2.5) {};
    \node [fill, circle, inner sep=1.3] (v2) at (-4,0) {};
    \node [fill, circle, inner sep=1.3] (v4) at (4,0) {};
    \draw [blue](v1) -- (v2) -- (v3) -- (v4) -- (v1) -- (v3);
    \draw [thick](-1.75,0.35) .. controls (-2.2,-0.35) and (-2.3,-0.65) .. (-2.45,-0.95);
    \draw [thick](-1.75,0.35) .. controls (-2.2,0.8) and (-2.2,0.8) .. (-2.4,1);
    \draw [thick](-1.75,0.35) .. controls (-1.15,0.45) and (-0.1,1.05) .. (1.2,1.75);
    \draw [thick](1.2,-0.35) .. controls (0,-0.75) and (-0.65,-1.4) .. (-1.05,-1.85);
    \draw [thick](1.2,-0.35) .. controls (1.9,0.4) and (2.3,0.7) .. (2.55,0.9);
    \draw [thick](1.2,-0.35) .. controls (1.65,-0.75) and (1.9,-0.95) .. (2.1,-1.2);
    \draw (-1.95,1.3) .. controls (-1.45,0.9) and (-1.5,0.55) .. (1,1.9);
    \draw (-1.5,1.55) .. controls (-1.15,1.3) and (-1,1.15) .. (0.65,2.05);
    \draw (-1.05,1.85) .. controls (-0.8,1.65) and (-0.15,1.9) .. (0.4,2.25);
    \draw (-2.75,0.75) .. controls (-2.35,0.45) and (-2.35,0.05) .. (-2.85,-0.7);
    \draw (-3.15,0.5) .. controls (-2.95,0.3) and (-2.9,0) .. (-3.2,-0.5);
    \draw (-3.6,0.25) .. controls (-3.45,0.1) and (-3.45,-0.05) .. (-3.55,-0.25);
    \draw (-2.15,-1.15) .. controls (-1.5,0) and (-1.3,0) .. (1.45,1.6);
    \draw (-1.8,-1.35) .. controls (-1.15,-0.3) and (-0.25,0.3) .. (1.7,1.4);
    \draw (-1.55,-1.55) .. controls (-0.7,-0.45) and (0.25,0.1) .. (2,1.25);
    \draw (-1.3,-1.7) .. controls (-0.25,-0.3) and (1.15,-0.05) .. (2.25,1.1);
    \draw (-0.75,-2) .. controls (0.6,-0.75) and (1,-0.85) .. (1.7,-1.45);
    \draw (-0.5,-2.2) .. controls (0.25,-1.4) and (0.95,-1.45) .. (1.2,-1.75);
    \draw (-0.25,-2.3) .. controls (0,-2) and (0.55,-1.9) .. (0.75,-2.05);
    \draw (2.8,0.75) .. controls (1.75,-0.35) and (1.75,-0.35) .. (2.45,-0.95);
    \draw (3.1,0.6) .. controls (2.35,-0.2) and (2.35,-0.35) .. (2.8,-0.75);
    \draw (3.4,0.4) .. controls (2.95,-0.1) and (2.95,-0.25) .. (3.15,-0.55);
    \draw (3.7,0.2) .. controls (3.5,0) and (3.4,-0.15) .. (3.5,-0.3);
    \draw[red, thick] (-0.5,0.85) .. controls (-0.4,0.3) and (0.4,0) .. (0.55,-0.6);
    \node[red] at (0.25,0.3) {$\gamma$};
    \node [blue] at (-0.35,-1.5) {$\alpha$};
    \node [blue] at (9.65,-1.5) {$\alpha$};
    \node (F) at (5,-2.5) {$(F, \mu)$};
    \draw [->] (F) -- (3.5,-1.5) node [blue] {};
    \draw [->] (F) -- (6.5,-1.5) node [blue] {};
    \end{scope}
    \end{tikzpicture}
    \caption{Left (resp. right) figure shows that the local model of the measured foliation $(F, \mu)$ whose shear coordinate at $\alpha$ is negative (resp. positive); left: $x^\tri_\alpha([F, \mu]) = + \int_\gamma \mu$, right: $x^\tri_\alpha([F, \mu]) = - \int_\gamma \mu$.}
    \label{f:shear}
\end{figure}
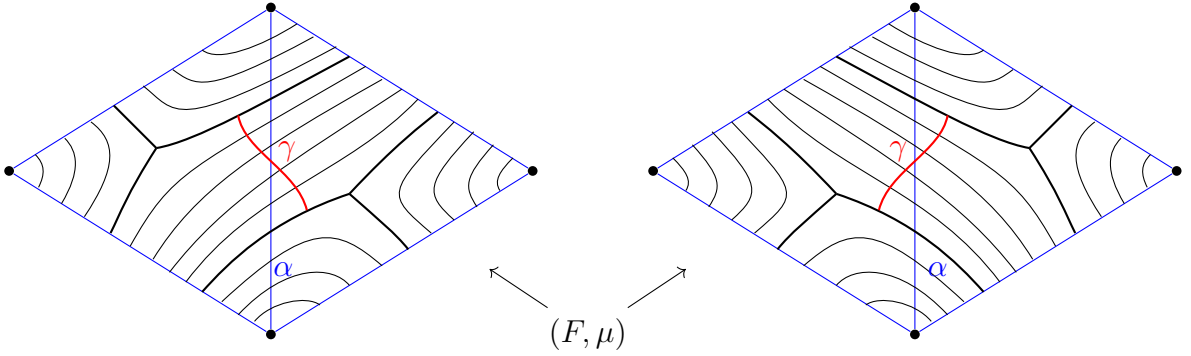
\end{ex}

\begin{ex}
As we saw in \cref{tab:deform_train_track}, the elementary moves of train tracks are corresponding to the binary relations arising from flips.
Thus, each elementary move corresponds to a domain of linearity of a tropicalized cluster transformation.
In \cref{fig:split,fig:shift}, the green letters mean the corresponding signs.
\end{ex}

The notion of sign-stability is the asymptotic stability property of the sign of a representation path of a mutation loop as iterating the mutation loop:

\begin{dfn}[sign stability]\label{d:sign stability}
Let $\gamma: v_0 \to \phi^{-1}(v_0)$ be a representation path of a mutation loop $\phi$.
Let $\Omega \subset \X_\bs(\bR^\trop)$ be a subset which is invariant under the rescaling action of $\bR_{> 0}$. 
Then we say that $\gamma$ is \emph{sign-stable} on $\Omega$ if there exists a sequence $\boldsymbol{\epsilon}^\stab_{\gamma,\Omega} \in \{+,-\}^{h(\gamma)}$ of strict signs such that for each $w \in \Omega \setminus \{0\}$, there exists an integer $n_0 \in \mathbb{N}$ such that   \[\boldsymbol{\epsilon}_\gamma(\phi^n(w)) = \boldsymbol{\epsilon}^\stab_{\gamma, \Omega} \]
for all $n \geq n_0$. 
We call $\boldsymbol{\epsilon}_{\gamma,\Omega}^\stab$ the \emph{stable sign} of $\gamma$ on $\Omega$.
\end{dfn}

\section{Train track splittings}\label{sec:TTsplit}

In this section, we give a concrete relationship between the sequence of train track splittings which represents a pseudo-Anosov mapping class and the stable sign of some representation path of the pseudo-Anosov mapping class. 

\subsection{Sign stability of generic pseudo-Anosov mapping classes}
First, we recall the concept of the pseudo-Anosov mapping classes.
The orientation preserving homeomorphism $f: \Sigma \to \Sigma$ on a surface $\Sigma$ is \emph{pseudo-Anosov} if there exists a pair $((F^+_f, \mu^+_f), (F^-_f, \mu^-_f))$ of measured foliations on $\Sigma$ satisfying the following:
\begin{itemize}
    \item for any pair consisting of the respective leaves of $F^+_f$ and $F^-_f$, the leaves are transversal each other if they intersect and
    \item there is a number $\lambda_f > 1$ such that the measured foliation $f_*(F^\pm_f, \mu^\pm_f)$ obtained by pushforward\footnote{
    For a measured foliation $(F, \mu)$ and a homeomorphism $f$, the pushforward $f_*(F, \mu)$ of $(F, \mu)$ by $f$ is the measured foliation $(f(F), f_* \mu)$, where $f_* \mu(\alpha) := \mu(f^{-1}(\alpha))$.} by $f$ is equal to $(F^\pm_f, \lambda_f^{\pm 1} \mu^\pm_f)$.
\end{itemize}
In particular, a pseudo-Anosov homeomorphism $f$ is \emph{generic} if the measured foliations $\fol^\pm_f$ have only 1- or 3-pronged singularities.
A mapping class is \emph{(generic) pseudo-Anosov} if it is represented by a (generic) pseudo-Anosov homeomorphism.

\begin{rmk}
By the first condition of the pair of the measured foliations $(F^+_f, F^-_f)$ of a pseudo-Anosov homoeomorphism $f$, the set of singularities of $F^+_f$ and $F^-_f$ are the same.
Also, the set of singularities of $F^+_f$ and $F^-_f$ does not necessarily contain the set of punctures.
In the space $\MF(\Sigma)$ of measured foliations, the degree of singularities is not well-defined since the Whitehead move does not preserve it but this data will be important for the sign stability.
\end{rmk}

First, we recall the main result of \cite{IK20a}:

\begin{thm}[{\cite[Theorem 7.1 and Remark 7.8]{IK20a}}]\label{thm:unifSS_genericpA}
Let $\phi$ be a mapping class of a punctured surface $\Sigma$.
Then $\phi$ is generic pseudo-Anosov if and only if any representation path $\gamma:(\tri,\ell) \to \phi^{-1}(\tri,\ell)$ in $\bTri(\Sigma)$ is sign-stable on $\bR_{>0} \cdot \cX_\Sigma(\bZ^\trop)$.
\end{thm}



For a representation path $\gamma$ of a pseudo-Anosov mapping class $\phi$, we write simply $\bep_\gamma^\stab$ for its stable sign on $\bR_{>0} \cdot \cX_\Sigma(\bZ^\trop)$.

\subsection{Sign stability of general pseudo-Anosov mapping classes via train track splittings}

Here, we discuss the sign stability of a ``general'' pseudo-Anosov mapping class.
The following theorem plays a key role in this subsection:

\begin{thm}[{\cite[Theorem 2.4.1]{PH}}]\label{thm:split_seq}
If birecurrent train tracks $\tau$ and $\tau'$ are satisfying $\tau \succ \tau'$, then for each transverse measures $\nu \in \relint V(\tau)$ and $\nu' \in \relint V(\tau')$ such that they correspond to the same arational point $\fol \in \MF(\Sigma)$, there is a sequence of train tracks $\tau = \tau_0 \succ \tau_1 \succ \cdots \succ \tau_h = \tau'$ such that $\tau_{t} \succ \tau_{t+1}$ is a splitting or a shifting with a transition matrix $M_i$ for each $t=0, \dots, h-1$, and satisfy $M_{h-1} \cdots M_1 M_0 \nu = \nu'$.
\end{thm}

This sequence of train tracks is called \emph{RLS-word} of $\fol \in \MF(\Sigma)$ in \cite{PP87}.
For later discussions, we recall the construction of the sequence of splittings and shiftings likes as \cref{thm:split_seq}.

By the assumption $\tau \succ \tau'$, we can deform $\tau'$ by a homotopy such that $\tau' \subset \interior N_\tau$ and $\tau$ is transverse to the ties of $N_\tau$.
Furthermore, we can take a fibered neighborhood $N_{\tau'}$ of $\tau'$ such that $N_{\tau'} \subset \interior N_\tau$ and each tie of $N_{\tau'}$ is a restriction of some tie of $N_\tau$.
Since $\fol \in \cV(\tau)$, we can take a partial measured foliation $(F, \mu)$ such that $|F| \subset N_\tau$ and it is transverse to the ties.
In particular, we can take the partial measured foliation $(F, \mu)$ satisfying $|F| = N_\tau$ since $\fol \in \relint \cV(\tau)$ (\emph{cf.}, \cite[Construction 1.7.7]{PH}).
Let us consider a singular leaf $\varsigma$ of $F$ which is starting from a cusp $s$ of $N_\tau$ with a parametrization $\alpha: [0, \infty) \to \varsigma$ and the set $T$ of the ties of $N_\tau$ which contain some cusp of $N_{\tau'}$.
By some argument, the cusp $s'$ of $N_{\tau'}$ such that it is contained in the tie $t \in T$ that the leaf $\varsigma$ hits first is opened in the same direction as the cusp $s$, see \cref{fig:rel_fib_nbd}.
For $u \in [0, \infty)$ such that $\alpha(u) \in t$, we put $N_1 := N_\tau \setminus \alpha([0, u])$ and let $\tau_1$ be a train track obtained by collapsing to a point each tie of $N_1$.
This procedure is called an \emph{unzipping} along a singular leaf $\varsigma$.
It is known that there is a sequence of unzipping from $\tau$ to $\tau'$, and one can verify easily that an unzipping is splitting or shifting.
So by unzipping, we obtain the sequence consisting of splittings and shiftings, which is the desired sequence.

\begin{figure}[h]
    \centering
    \begin{tikzpicture}[scale=1.5]
    \begin{scope}[xscale=-1]
    \draw [blue, thick](-1.5,-0.4) .. controls (0,-0.4) and (0.5,-0.4) .. (1.5,-0.9);
    \draw [blue, thick](0.5,0.15) .. controls (0.9,0.1) and (1.15,0.15) .. (1.5,0.25);
    \draw [blue, thick](0.5,0.15) .. controls (0.9,0.05) and (1.15,-0.05) .. (1.5,-0.25);
    \draw [blue](-1.5,0.95) -- (-1.5,-0.4);
    \draw [blue](-1.3,0.95) -- (-1.3,-0.4);
    \draw [blue](-1.1,1) -- (-1.1,-0.4);
    \draw [blue](-0.9,1) -- (-0.9,-0.4);
    \draw [blue](0.5,1.15) -- (0.5,-0.5);
    \draw [blue] (0.7,1.15) -- (0.7,0.1);
    \draw [blue] (0.7,0.1) -- (0.7,-0.55);
    \draw [blue](1.5,0.25) -- (1.5,1.45);
    \draw [blue](1.5,-0.25) -- (1.5,-0.9);
    \draw [blue](-0.7,1) -- (-0.7,-0.4);
    \draw [blue](0.1,1.05) -- (0.1,-0.45);
    \draw [blue](0.3,1.08) -- (0.3,-0.48);
    \draw [blue](0.9,1.2) -- (0.9,0.11);
    \draw [blue](0.9,0.02) -- (0.9,-0.65);
    \draw [blue](1.1,1.25) -- (1.1,0.15);
    \draw [blue](1.1,-0.05) -- (1.1,-0.7) -- cycle;
    \draw [blue](1.3,1.35) -- (1.3,0.2);
    \draw [blue](1.3,-0.15) -- (1.3,-0.8);
    \node [blue] at (-0.35,1.4) {$N_\tau$};
    \end{scope}
    \draw [red, thick](-0.5,0.15) .. controls (0,0.2) and (0.7,0.2) .. (1.5,0.25);
    \draw [violet, thick](-1.5,1) .. controls (-0.5,0.5) and (0,0.45) .. (1.5,0.45);
    \draw [violet, thick](-1.5,-0.65) .. controls (-0.5,-0.15) and (0,-0.05) .. (1.5,0.1);
    \draw [violet, thick](-1.5,0.65) .. controls (-0.5,0.25) and (0.05,0.2) .. (0.5,0.2);
    \draw [violet, thick](-1.5,-0.45) .. controls (-0.85,-0.1) and (0.05,0.2) .. (0.5,0.2);
    \draw [violet](-1.5,1);
    \draw [violet](-1.4,-0.4) -- (-1.4,-0.6);
    \draw [violet](-1.2,0.85) -- (-1.2,0.55);
    \draw [violet](-1,0.75) -- (-1,0.45);
    \draw [violet](-0.8,0.7) -- (-0.8,0.4);
    \draw [violet](-0.6,0.65) -- (-0.6,0.35);
    \draw [violet](-0.4,0.6) -- (-0.4,0.3);
    \draw [violet](-0.2,0.55) -- (-0.2,0.25);
    \draw [violet](0,0.5) -- (0,0.25);
    \draw [violet](0.2,0.5) -- (0.2,0.2);
    \draw [violet](-1.2,-0.3) -- (-1.2,-0.5);
    \draw [violet](-1,-0.2) -- (-1,-0.4);
    \draw [violet](-0.8,-0.1) -- (-0.8,-0.35);
    \draw [violet](-0.6,-0.05) -- (-0.6,-0.25);
    \draw [violet](-0.4,0) -- (-0.4,-0.2);
    \draw [violet](-0.2,0.1) -- (-0.2,-0.15);
    \draw [violet](0,0.15) -- (0,-0.1);
    \draw [violet](0.2,0.15) -- (0.2,-0.05);
    \draw [violet](0.4,0.45) -- (0.4,-0.05);
    \draw [violet](0.6,0.45) -- (0.6,0);
    \draw [violet](0.8,0.45) -- (0.8,0.05);
    \draw [violet](1,0.45) -- (1,0.05);
    \draw [violet](1.2,0.45) -- (1.2,0.05);
    \draw [violet](1.4,0.45) -- (1.4,0.1);
    \draw [violet](-1.4,0.95) -- (-1.4,0.6);
    \draw [->](-1.8,0.5) .. controls (-1.5,0.5) and (-0.75,0.25) .. (-0.55,0.15);
    \draw [->](1,1.15) .. controls (1,0.75) and (0.9,0.3) .. (0.55,0.2);
    \node at (-1.95,0.5) {$s$};
    \node at (1.05,1.4) {$s'$};
    \node[violet] at (0,-0.7) {$N_{\tau'}$};
    \draw [blue, thick] (-1.5,1.45) .. controls (-1,1.15) and (0.05,1) .. (1.5,0.95);
    \draw [blue](0.5,1) .. controls (0.5,0.7) and (0.5,0.2) .. (0.5,-0.4);
    \draw [blue](0.1,1.05) .. controls (0.1,0.95) and (0.1,0.2) .. (0.1,-0.45);
    \draw [blue](0.3,1) .. controls (0.3,0.85) and (0.3,0.1) .. (0.3,-0.4);    
    \node[red] at (1.65,0.25) {$\varsigma$};
    \end{tikzpicture}
    \caption{The relative positions of the fibered neighbourhoods $N_\tau$ and $N_{\tau'}$.}
    \label{fig:rel_fib_nbd}
\end{figure}
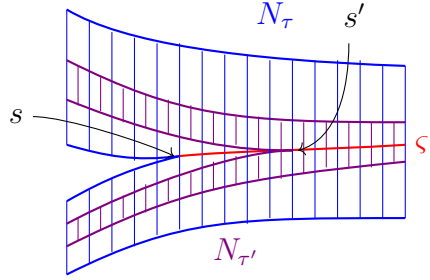


\begin{thm}[{\cite{BH95}}]\label{thm:BH}
For a pseudo-Anosov mapping class $\phi \in MC(\Sigma)$ with an unstable measured foliation $(F^+_\phi, \mu^+_\phi)$, there is a birrecurent train track $\tau_\phi^+$ such that $[F^+_\phi, \mu^+_\phi] \in \relint \cV(\tau^+_\phi)$ and $\phi(\tau^+_\phi) \prec \tau^+_\phi$.
\end{thm}

The train track which satisfies this property is called \emph{invariant track} of $\phi$.

\begin{proof}
The existence of the train track $\tau^+_\phi$ satisfying $\phi(\tau^+_\phi) \prec \tau^+_\phi$ for a pseudo-Anosov mapping class $\phi$ is proven in \cite{BH95} but the birrecurency is not clear.
If $\tau^+_\phi$ is not recurrent, then some branches $b \in B(\tau^+_\phi)$ satisfy $\nu^+_\phi(b) = 0$, where $\nu^+_\phi$ is a measure of $\tau^+_\phi$ corresponding to $[F^+_\phi, \mu^+_\phi]$.
Thus in this case, we replace $\tau^+_\phi$ with the subtrack of it obtained by cutting off such branches.
The transverse recurrency of $\tau^+_\phi$ follows from the recurrency of the dual train track of $\tau^+_\phi$ (see \cite[Section 3.4]{PH}).
\end{proof}

Combining \cref{thm:split_seq,thm:BH}, we know that there is a sequence of splittings and shiftings
\begin{align}\label{eq:seq_split_inv_track}
    \tau^+_\phi = \tau_0 \succ \tau_1 \succ \cdots \succ \tau_h = \phi(\tau^+_\phi)
\end{align}
for a pseudo-Anosov mapping class $\phi$ but the invariant track $\tau^+_\phi$ is not suited to an ideal triangulation in general.
Therefore, we have to deform the invariant track to make it so.

First, we recall the construction of the invariant tracks briefly.
The invariant track $\tau^+_\phi$ of a pseudo-Anosov mapping class $\phi$ is obtained from a graph $G$ which is homotopic to $\Sigma \setminus P$ and a graph map $f: G \to G$ induced by $\phi$ called \emph{train track map} as follows:
For each vertex $v$ of $G$, we define an equivalence class on the set $E(v)$ of edges originating at $v$ as $e_1 \sim e_2$ for $e_1, e_2 \in E(v)$ if $f^r(e_1)$ and $f^r(e_2)$ have a nontrivial common initial segment for some $r>0$.
The equivalence class is called the \emph{gate} at $v$.
To each vertex $v$ of $G$, we assign a small disk with marked points on the boundary which are labeled by the gates at $v$.
If two edges $e_1, e_2 \in E(v)$ satisfy $e_1 \not \sim e_2$ and there is an edge $e$ and $r>0$ such that $f^r(e)$ contains $e_2 e_1$ or $e_1 e_2$ as a subpath, then we connect the marked points of the disk labeled by the gates containing $e_1$ and $e_2$ with an edge called \emph{infinitesimal branch}.
Eventually, an infinitesimal $k$-gon or an infinitesimal $k$-gon missing one side is drawn in each small disk (see \cref{fig:inf_k-gon}).
Finally, for each edge $e$, we put the edge connecting the gates which contain $e$, and smoothing the resulting graph as each gate is a switch as the unique tangent vector is transversal to the boundary of the small disk.

\begin{figure}[h]
    \centering
    \begin{tikzpicture}[scale=.8]
    \draw[dashed, fill=gray!30]  (-1.5,-0.5) ellipse (1.2 and 1.2);
    \draw[red, very thick] (-1.5,0.7) .. controls (-1.5,-0.2) and (-1.8,-0.45) .. (-2.65,-0.1) .. controls (-1.8,-0.45) and (-1.65,-0.75) .. (-2.2,-1.45) .. controls (-1.65,-0.75) and (-1.35,-0.75) .. (-0.75,-1.45) .. controls (-1.35,-0.75) and (-1.2,-0.45) .. (-0.4,-0.1) .. controls (-1.2,-0.45) and (-1.5,-0.2) .. (-1.5,0.7);
    \draw[dashed, fill=gray!30]  (3,-0.5) ellipse (1.2 and 1.2);
    \draw[red, very thick] (3,0.7) .. controls (3,-0.1) and (2.65,-0.3) .. (2,0.1) .. controls (2.65,-0.3) and (2.65,-0.65) .. (2,-1.25) .. controls (2.65,-0.65) and (3,-0.85) .. (3,-1.7) .. controls (3,-0.85) and (3.35,-0.65) .. (3.95,-1.3) .. controls (3.35,-0.65) and (3.35,-0.3) .. (4.1,0);
    \end{tikzpicture}
    \caption{Left: infinitesimal pentagon, right: infinitesimal hexagon missing one-side.}
    \label{fig:inf_k-gon}
\end{figure}
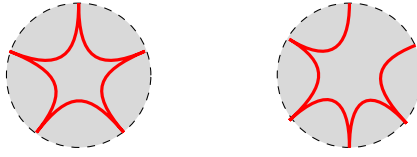

Next, we deform the invariant track to be suited to some ideal triangulation.
Obviously, the train tracks which are suited to ideal triangulations do not have any infinitesimal $(k>3)$-gons, so we have to delete them by sticking some infinitesimal branches in a $(k>3)$-gon.
By this deformation, each infinitesimal $(k>3)$-gon splits into connected $(k-2)$ trigons.
Concretely, to do this deformation, we take the following process:
For the invariant track $\tau = \tau_\phi^+$, take an ideal polygon decomposition $\tri'(\tau)$ which is dual to the graph $G$ which we used to make $\tau$.
Around infinitesimal $k$-gon of $\tau$, $\tri'(\tau)$ forms ideal $k$-gon.
Taking an ideal triangulation $\tri = \tri(\tau)$ containing $\tri'(\tau)$ and take a train track $s_\tri(\tau)$ such that it is
\begin{itemize}
    \item suited to $\tri$,
    \item the same with $\tau$ out of the ideal $(k>3)$-gons in $\tri$ and
    \item of type III in each ideal triangle of $\tri$ which is contained in the ideal $(k>3)$-gon of $\tri'(\tau)$.
\end{itemize}
Then, the cone $\cV(\tau)$ is a subset of the cone $\cV(s_\tri(\tau))$.

\begin{figure}[h]
    \centering
    \begin{tikzpicture}
    \draw[red, very thick] (-1.5,0.9) .. controls (-1.5,-0.2) and (-1.8,-0.45) .. (-2.85,-0.05) .. controls (-1.8,-0.45) and (-1.65,-0.75) .. (-2.3,-1.6) .. controls (-1.65,-0.75) and (-1.35,-0.75) .. (-0.65,-1.6) .. controls (-1.35,-0.75) and (-1.2,-0.45) .. (-0.2,-0.05) .. controls (-1.2,-0.45) and (-1.5,-0.2) .. (-1.5,0.9);
    \draw[blue] (-0.6,0.75) coordinate (v1) -- (-2.4,0.75) -- (-2.95,-0.95) -- (-1.5,-2.05) -- (0,-0.95) -- (v1);
	\node[red] at (-3.1,0.35) {$\tau$};
	\node[blue] at (-3.2,-1.65) {$\triangle'(\tau)$};
    
    \draw [ultra thick,-{Classical TikZ Rightarrow[length=4pt]},decorate,decoration={snake,amplitude=1.5pt,pre length=2pt,post length=3pt}](0.55,-0.5) -- (1.95,-0.5);
    
    \draw[red, very thick] (4.2,0.9) .. controls (4.2,-1.2) and (3.4,-0.7) .. (2.85,-0.05) .. controls (3.4,-0.7) and (3.4,-0.75) .. (3.3,-1.6) .. controls (3.35,-0.1) and (5,-0.1) .. (5.15,-1.6) .. controls (5,-0.75) and (5,-0.7) .. (5.5,-0.05) .. controls (5,-0.7) and (4.2,-1.2) .. (4.2,0.9);
    \draw[blue] (5.1,0.75) coordinate (v1) -- (3.3,0.75) coordinate (v2) -- (2.75,-0.95) -- (4.2,-2.05) coordinate (v3) -- (5.7,-0.95) -- (v1);
    \draw [blue] (v2) edge (v3);
    \draw [blue] (v1) edge (v3);
    \node[red] at (6,0.35) {$s_\triangle(\tau$)};
	\node[blue] at (5.85,-1.45) {$\triangle$};
    \end{tikzpicture}
    \caption{Deformation of the infinitesimal pentagon to 3 infinitesimal triangles}
    \label{fig:my_label}
\end{figure}
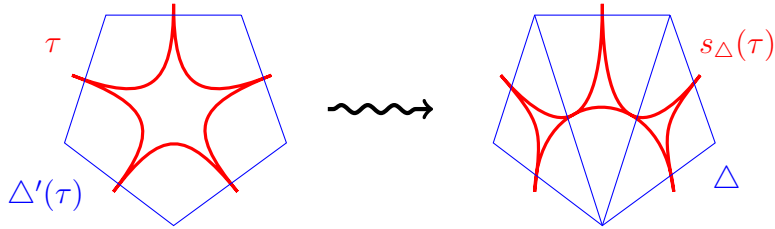



Let $\phi$ be a pseudo-Anosov mapping class of a punctured surface $\Sigma$.

\begin{lem}\label{lem:unizip_path}
There is a horizontal path $\gamma: (\tri, \ell) \to (\tri', \ell')$ in $\Tri(\Sigma)$ such that 
\begin{itemize}
    \item $\tri$ and $\tri'$ contain the ideal polygon decompositions $\tri'(\tau^+_\phi)$ and $\tri'(\phi(\tau^+_\phi))$ respectively, and
    \item the sign $\bep^+_\gamma = \bep_\gamma([F^+_\phi, \mu^+_\phi])$ is strict, namely, $\bep^+_\gamma \in \{+, -\}^{h(\gamma)}$.
\end{itemize}
\end{lem}

\begin{proof}
We recall that the sequence \eqref{eq:seq_split_inv_track} is obtained from unzipping along some singular leaves of $F^+_\phi$.
Taking the subarcs $\varsigma_1, \varsigma_2, \dots, \varsigma_m$ of that singular leaves such that the train track $\phi(\tau^+_\phi)$ is given by unzipping along them to the invariant track $\tau^+_\phi$.
For $i = 1, 2, \dots, m$, let $\mathbf{k}(\varsigma_i) = (k_1, k_2, \dots, k_s)$ be the tuple of indices such that the arc $\varsigma_i$ hits to the ideal arcs $\ell(k_1), \ell(k_2), \dots, \ell(k_s) \in \tri$ in this order for a parametrization of $\varsigma_i$ which starts from the corresponding cusp of $F^+_\phi$.
Then, we will see that the horizontal path $\gamma$ along the tuple of indices $(\mathbf{k}(\varsigma_1), \mathbf{k}(\varsigma_2), \dots, \mathbf{k}(\varsigma_m))$ is the desired one.
Since singular leaves are not intersect each other, unzippings along different singular leaves are commutative.
Thus, it is enough to show that the case $m=1$.
We write $\varsigma$ for the subarc of a singular leaf.

For $\tau^+_\phi$, take a train track $s_\tri(\tau^+_\phi) = \tau$ likes as above.
By applying the retraction $r: |F^+_\phi| \searrow \tau$ of a fibered neighborhood of $\tau$ for $\varsigma$, we obtain the path $b_1 b_2 \cdots b_u$ of branches on the train track $\tau$.
It is clear that the subsequence $(i(1), i(2), \dots, i(s))$ of $(1, 2, \dots, u)$ corresponding to long branches of $\tau$ satisfies $b_{i(j)} \pitchfork \ell(k_j)$ for $j = 1, \dots, s$ and $\mathbf{k}(\varsigma) = (k_1, \dots, k_s)$.

In the case that the path $b_1 b_2 \cdots b_u$ does not contain infinitesimal branches of any infinitesimal $(k>3)$-gons, 
the unzipping along $\varsigma$ for $\tau^+_\phi$ is (not central) splitting or shifting thus so is for $\tau$.
Thus, the conditions for $\gamma$ are satisfied trivially.

In the other cases, the arc $\varsigma$ contains an infinitesimal branch of an infinitesimal $(k>3)$-gon.
We assume that $\varsigma$ contains only one such infinitesimal branch and the subsequence $(i(1), \dots, i(s))$ has length 2 for simplicity.
(Namely, $\varsigma$ consists of 2 large branches and one infinitesimal branch between them.)
After unzipping, the infinitesimal $(k>3)$-gon which has an infinitesimal branch contained in $\varsigma$ is still infinitesimal $(k>3)$-gon, and the adjacency of the switches of the infinitesimal $(k>3)$-gon which is not contained in $\varsigma$ to the long branches are preserved.
Thus, the ideal triangulation $\tri$ which contains $\tri'(\tau^+_\phi)$ induces the ideal triangulation $\tri'$ which contains $\tri'(\phi(\tau^+_\phi))$ naturally as \cref{fig:unizip_flip}.
Also, the sequence $\mathbf{k}(\varsigma)$ gives the path from $(\tri, \ell)$ to $(\tri', \ell')$ and the sign is strict by \cref{tab:deform_train_track}. See right side of \cref{fig:unizip_flip}.

Since the other cases are given by combining the two cases above, the statement was proven.
\end{proof}

\begin{figure}[h]
    \centering
    \begin{tikzpicture}
    \draw [very thick, red](-2.5,2.5) .. controls (-2,2) and (-2,2) .. (-1.5,2) .. controls (0.25,2) and (0.25,2) .. (-0.75,3.5) .. controls (0.25,2) and (0.25,2) .. (1.25,3.5) .. controls (0.25,2) and (0.25,2) .. (2,2) .. controls (0.25,2) and (0.25,2) .. (1.25,0.5) .. controls (0.25,2) and (0.25,2) .. (-0.75,0.5) .. controls (0.25,2) and (0.25,2) .. (-1.5,2) .. controls (-2,2) and (-2,2) .. (-2.5,1.5);
    \draw [blue](-2.5,2) coordinate (v1) -- (-1,3) coordinate (v2) -- (0.25,3.5) -- (1.5,3) -- (1.5,1) -- (0.25,0.5) -- (-1,1) coordinate (v3) -- (v1);
    \draw [blue] (v2) edge (v3);
    
    \draw [very thick, red] (-1.5,-4) .. controls (0.25,-4) and (0.25,-4) .. (-0.75,-2.5) .. controls (0.25,-4) and (0.25,-4) .. (1.25,-2.5) .. controls (0.25,-4) and (0.25,-4) .. (2,-4) .. controls (0.25,-4) and (0.25,-4) .. (1.25,-5.5) .. controls (0.25,-4) and (0.25,-4) .. (-0.75,-5.5) .. controls (0.25,-4) and (0.25,-4) .. (-1.5,-4) .. controls (-2,-4) and (-2,-4) .. (-2.5,-4.5);
    \draw [very thick, red](-0.15,-3.45) .. controls (0.2,-4) and (-2.3,-3.7) .. (-2.5,-3);
    
    \draw [very thick, red] (-1.5,-10) .. controls (0.25,-10) and (0.25,-10) .. (-0.75,-8.5) .. controls (0.25,-10) and (0.25,-10) .. (1.25,-8.5) .. controls (0.25,-10) and (0.25,-10) .. (2,-10) .. controls (0.25,-10) and (0.25,-10) .. (1.25,-11.5) .. controls (0.25,-10) and (0.25,-10) .. (-0.75,-11.5) .. controls (0.25,-10) and (0.25,-10) .. (-1.5,-10) .. controls (-2,-10) and (-2,-10) .. (-2.5,-10.5);
    \draw [very thick, red](-0.6,-8.7) .. controls (-0.25,-9.25) and (-2.3,-9.7) .. (-2.5,-9);
    \draw [blue](-2.5,-10) coordinate (v1) -- (-1.2,-8.85) coordinate (v2) -- (0.25,-8.5) -- (1.5,-9) -- (1.5,-11) -- (0.25,-11.5) -- (-1,-11) coordinate (v3) -- (v1);
    \draw [blue](-2.5,-10) .. controls (-1.1,-9.65) and (-0.6,-9.35) .. (0.25,-8.5);
    
    \draw [very thick, red](4.3,2.7) .. controls (5.45,1.6) and (8.15,0.55) .. (6,3.5) .. controls (7.2,1.95) and (7.6,1.95) .. (7.5,3.5) .. controls (7.5,2.35) and (7.5,2.35) .. (8.5,2) .. controls (7.1,2.55) and (6.05,2.55) .. (7.5,0.5) .. controls (6.4,1.9) and (5.7,2) .. (6,0.5) .. controls (5.8,2.1) and (5.35,2.25) .. (4.35,1.35);
    \draw [blue](4,2) coordinate (v1) -- (5.5,3) coordinate (v2) -- (6.75,3.5) coordinate (v6) {} -- (8,3) -- (8,1) coordinate (v5) {} -- (6.75,0.5) coordinate (v4) {} -- (5.5,1) coordinate (v3) -- (v1);
    \draw [blue] (v2) edge (v3);
    \draw [blue](v4) -- (v2) -- (v5) -- (v6);
    
    \draw [very thick, red](4.35,-1.25) .. controls (5.5,-2.35) and (8.15,-3.45) .. (6,-0.5) .. controls (7.2,-2.05) and (7.6,-2.05) .. (7.5,-0.5) .. controls (7.5,-1.65) and (7.5,-1.65) .. (8.5,-2) .. controls (7.1,-1.45) and (6.05,-1.45) .. (7.5,-3.5) .. controls (6.4,-2.1) and (5.55,-1.85) .. (6,-3.5) .. controls (5.75,-2.35) and (5.05,-2.6) .. (4.5,-2.65);
    \draw [very thick, red](5.15,-2.6) .. controls (5.6,-2.65) and (5.8,-2.65) .. (5.9,-2.55);
    \draw [blue](4,-2) coordinate (v1) -- (5.5,-1) coordinate (v2) -- (6.75,-0.5) coordinate (v6) {} -- (8,-1) -- (8,-3) coordinate (v5) {} -- (6.75,-3.5) coordinate (v4) {} -- (5.5,-3) coordinate (v3) -- (v1);
    \draw [blue](4,-2) .. controls (5,-2) and (6,-2.5) .. (6.75,-3.5);
    \draw [blue](v4) -- (v2) -- (v5) -- (v6);
    
    \draw [very thick, red](4.55,-5.15) .. controls (5.4,-5.75) and (7.95,-7) .. (6,-4.5) .. controls (7.2,-6.05) and (7.6,-6.05) .. (7.5,-4.5) .. controls (7.5,-5.65) and (7.5,-5.65) .. (8.5,-6) .. controls (7.1,-5.45) and (5.7,-5.5) .. (7.55,-7.45) .. controls (6.4,-6.4) and (5.55,-6.45) .. (6,-7.5) .. controls (5.75,-6.4) and (5.05,-6.6) .. (4.5,-6.65);
    \draw [blue](4,-6) coordinate (v1) -- (5.5,-5) coordinate (v2) -- (6.75,-4.5) coordinate (v6) {} -- (8,-5) -- (8,-7) coordinate (v5) {} -- (6.75,-7.5) coordinate (v4) {} -- (5.5,-7) coordinate (v3) -- (v1);
    \draw [blue](4,-6) .. controls (5.15,-6.1) and (6.3,-6.55) .. (6.75,-7.5);
    \draw [blue] (v2) -- (v5) -- (v6);
    \draw [blue](4,-6) .. controls (5,-5.95) and (6.5,-6.15) .. (8,-7);
    
    \draw [very thick, red](4,-9) .. controls (5,-9.5) and (5.85,-9.3) .. (5.85,-8.4);
    \draw [very thick, red](7.5,-8.5) .. controls (7.65,-9.4) and (7.5,-9.6) .. (8.5,-10);
    \draw [very thick, red](8.5,-10) .. controls (7.2,-9.3) and (6.95,-9.75) .. (6.6,-10.15) .. controls (6.1,-10.6) and (6.6,-10.9) .. (7.6,-11.4);
    \draw [very thick, red](5.85,-8.4) .. controls (5.75,-9.25) and (6,-9.9) .. (6.7,-9.75) .. controls (7.85,-9.55) and (7.6,-9.5) .. (7.5,-8.5);
    \draw [very thick, red](7.6,-11.4) .. controls (5.45,-10.2) and (5.5,-10.55) .. (6,-11.5);
    \draw [very thick, red](6,-11.5) .. controls (5.6,-10.5) and (5.15,-10.55) .. (4.55,-10.65);
    \draw [very thick, red](4.55,-10.65) .. controls (5.35,-10.45) and (6.15,-10.9) .. (6.45,-10.35) .. controls (6.65,-10) and (6.35,-9.9) .. (6.1,-9.6);
    \draw [blue](4,-10) coordinate (v1) -- (5.3,-8.85) coordinate (v2) -- (6.75,-8.5) -- (8,-9) -- (8,-11) coordinate (v8) {} -- (6.75,-11.5) -- (5.5,-11) coordinate (v3) -- (v1);
    \draw [blue](4,-10) .. controls (5.4,-9.65) and (5.9,-9.35) .. (6.75,-8.5) coordinate (v9) {};
    \draw [blue](4,-10) coordinate (v7) {} .. controls (5.5,-10.45) and (6,-10.65) .. (6.75,-11.5);
    \draw [blue] (v8) -- (v9);
    \draw [blue](4,-10) .. controls (5.5,-10) and (6.5,-10) .. (8,-11);
    
    \draw [very thick, red](5.2,-6.6) .. controls (5.55,-6.65) and (6.9,-6.85) .. (6.7,-6.3);
    \draw [ultra thick,-{Classical TikZ Rightarrow[length=4pt]},decorate,decoration={snake,amplitude=1.5pt,pre length=2pt,post length=3pt}](2.5,2) -- (3.5,2);
    \draw [->, thick](0,-0.25) -- node[midway, right]{split} (0,-1.75);
    \draw [->, thick](0,-6.25) -- node[midway, right]{shift} (0,-7.75);
    \draw [ultra thick,-{Classical TikZ Rightarrow[length=4pt]},decorate,decoration={snake,amplitude=1.5pt,pre length=2pt,post length=3pt}](2.5,-10) -- (3.5,-10);
    \draw [-implies, double distance=2pt](5.5,0.5) -- node[midway, left]{$\mu_{k_1}$} (5.5,-0.5);
    \draw [-implies, double distance=2pt](5.5,-3.5) -- node[midway, left]{$\mu_{k_2}$} (5.5,-4.5);
    \draw [-implies, double distance=2pt](5.5,-7.5) -- node[midway, left]{$\mu_{k_3}$} (5.5,-8.5);
    \node[blue] at (5.3,2.4) {$k_1$};
    \node[blue] at (5.65,-1.65) {$k_2$};
    \node[blue] at (6,-5.55) {$k_3$};
    \node[blue] at (-1.8,3) {$\tri'(\tau)$};
    \node[blue] at (4.75,3) {$\tri$};
    \node[blue] at (-1.85,-8.75) {$\tri'(\tau')$};
    \node[blue] at (4.75,-8.75) {$\tri'$};
    \node[red] at (1.9,2.3) {$\tau$};
    \node[red] at (8.55,2.35) {$s_\tri(\tau)$};
    \node[red] at (1.9,-9.65) {$\tau'$};
    \node[red] at (8.7,-9.55) {$s_{\tri'}(\tau')$};
    \end{tikzpicture}
    \caption{The compatibility of the RLS-word and the sequence of flips around the infinitesimal hexagon of $\tau$.}
    \label{fig:unizip_flip}
\end{figure}
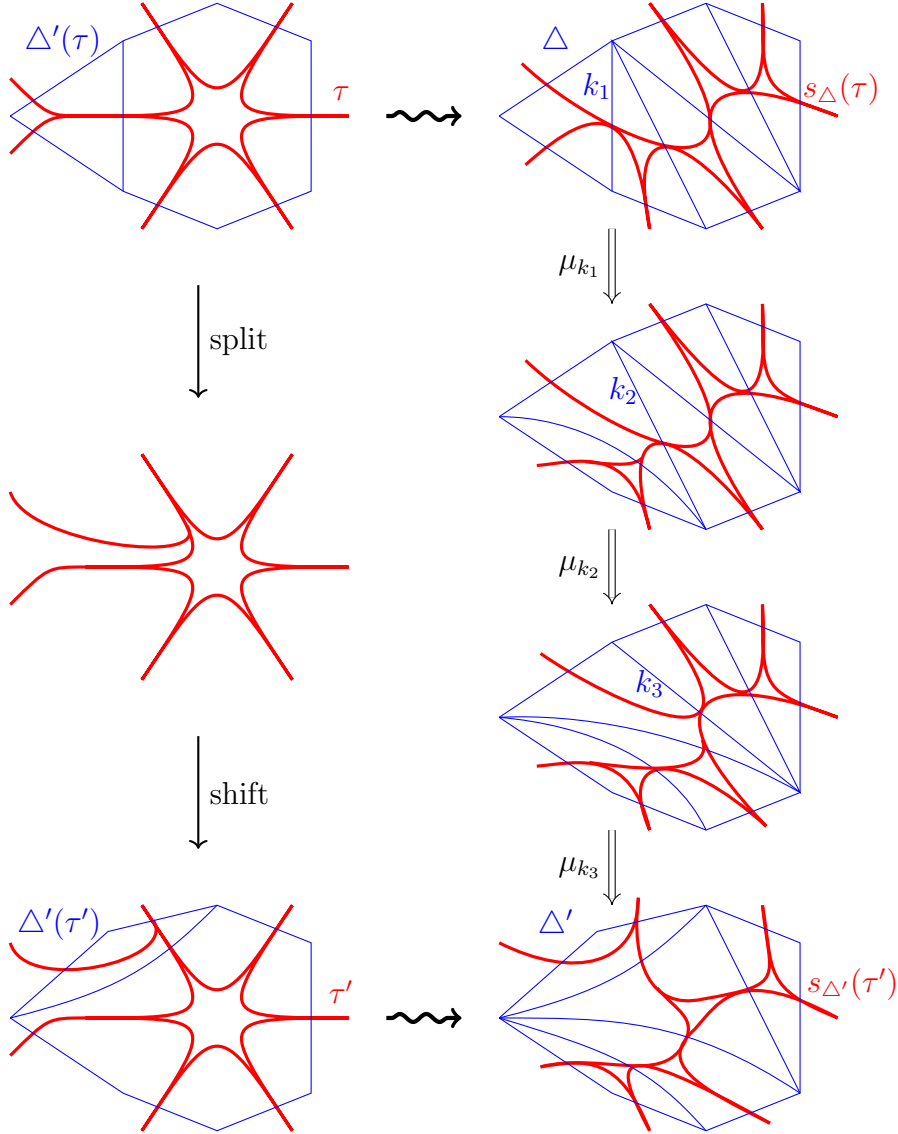

We note that the ideal triangulation $\tri'$ appeared in \cref{lem:unizip_path} is not $\phi(\tri)$ but they contain the ideal polygon decomposition $\tau'(\phi(\tau^+_\phi))$.
Thus, the difference between them is only the ideal triangulation of the ideal $(k>3)$-gons.
Since the coordinates $x_\alpha^{\tri'}([F^+_\phi, \mu^+_\phi])$ are zero for any $\alpha \in \tri' \setminus \tri'(\phi(\tau^+_\phi))$, so the paths from $(\tau', \ell')$ to $\phi(\tri, \ell)$ which deform only inside the ideal polygons are not desired ones.
Although, by taking the same ideal triangulations of the ideal $k$-gons of $\tri'(\tau^+_\phi)$ up to rotation for each $k>3$ at the beginning, the difference of the ideal triangulations of them are only some rotations.
Therefore, iterating this procedure several times, we obtain the ideal triangulation $\phi^r(\tri)$:

\begin{lem}\label{lem:phi^r}
There is an integer $r \geq 1$, such that $s_{\phi^r(\tri)}(\phi^r(\tau^+_\phi)) = \phi^r(s_\tri(\tau^+_\phi))$.
\end{lem}

Therefore, the inverse of the path $\gamma$ of \cref{lem:unizip_path} for $\phi^r$ is a sign-stable path of $\phi^r$.
Now ready to prove our main statement:
\begin{thm}\label{thm:SS_pA}
Let $\Sigma$ be a punctured surface and $\phi$ be a (general) pseudo-Anosov mapping class on $\Sigma$.
Then, $\phi^r$ has a representation path which is sign-stable on $\bR_{>0} \cdot \cX_\Sigma^\uf(\bZ^\trop)$ for some $r \geq 1$.
\end{thm}

\begin{proof}
Let $\phi$ be a pseudo-Anosov mapping class on $\Sigma$.
Then, by \cref{lem:unizip_path}, there is a representation path 
\begin{align*}
    \gamma: (\tri, \ell) \to \phi^{-r}(\tri, \ell)
\end{align*}
of $\phi^r$ for some $r \geq 1$ such that the sign $\bep^+_\gamma$ of $\gamma$ at $[F^+_\phi, \mu^+_\phi]$ is strict.
Then, we can use this path $\gamma$ for the discussion in \cite[Section 7.1]{IK20a}.
Namely, for each point $\cF \in \cX_\Sigma(\bZ^\trop)$, there is $n_0 >0$ such that $\phi^{nr}(\cF) \in \interior \cC^{\bep^+_\gamma}_\gamma$ for all $n \geq n_0$.
Here, $\cC^{\bep^+_\gamma}_\gamma$ is a full dimensional cone, which is defined by
\begin{align*}
    \cC^{\bep^+_\gamma}_\gamma := \overline{\{ \cF \in \cX_{\Sigma}(\bR^\trop) \mid \bep_\gamma(\cF) = \bep^+_\gamma\}}.
\end{align*}
It is nothing but the desired statement.
\end{proof}

\subsection{Comments on the bounded sign stability of pseudo-Anosov mapping classes}\label{subsec:bdd_SS}
The Perron--Frobenius theory is one of the strong methods to study the linear dynamical system.
This theory tells us that the spectral radius of an irreducible nonnegative matrix is larger than 1 and simple, and all entries of its eigenvector are positive, in brief.
Nowadays, the Perron--Frobenius theory is generalized for the matrices  $M$ with the \emph{invariant cones} $\cC$.
Namely, the linear map $f$ corresponding to $M$ satisfies $f(\cC) \subset \cC$.
The original theory is the Perron--Frobenius theory for the nonnegative matrices with the nonnegative cone as the invariant cone.
For most of the theorems of the Perron--Frobenius theory, the invariant cones are polyhedral.

In the theory of sign stability of mutation loops, the cone
\begin{align*}
    \cC^\stab_\gamma := \overline{\bigcap_{n \geq 0} \phi^{-n}(\cC^{\bep^\stab_\gamma}_\gamma)}
\end{align*}
corresponds to the invariant cone but it is not polyhedral in general.
Here, $\gamma$ is a sign-stable representation path of a mutation loop $\phi$ and $\bep^\stab_\gamma$ is the stable sign.
The easy sufficient condition of that this cone is polyhedral is as follows:
\begin{defi}
A representation path $\gamma$ of a mutation loop $\phi$ is \emph{bounded sign-stable} on $\Omega$ if it is sign-stable on $\Omega$ with the stable sign $\bep^\stab_\gamma$ and satisfies
\begin{align*}
    \cC^\stab_\gamma = \bigcap_{0 \leq n \leq n_0} \phi^{-n} (\cC^{\bep^\stab_\gamma}_\gamma).
\end{align*}
\end{defi}

Since $\cC_\gamma^{\bep}$ is (rational) polyhedral, so is $\cC_\gamma^{\stab}$ if $\gamma$ is bounded sign-stable.
Then, we can use some techniques of Perron--Frobenius theory.
For instance, the bounded sign stability is a sufficient condition of the \emph{North dynamics} \cite[Proposition 3.14]{Kan21}.

For a pseudo-Anosov mapping class $\phi$, it is clear that the cone $\cV(\tau^+_\phi)$ is (rational) polyhedral and it is an invariant cone of $\phi$: $\phi(\cV(\tau^+_\phi)) \subset \cV(\tau^+_\phi)$.
On the other hand, the bounded sign stability holds for some pseudo-Anosov mapping classes which the author checked.
Moreover, these examples satisfy the following:
\begin{align}
    \cC^\stab_\gamma \cap \cU_\Sigma(\bR^\trop) = p_\Sigma(\cV(\tau^+_\phi)).
\end{align}
So we conjecture that this property holds for every pseudo-Anosov mapping class.

\subsection{Algebraic and categorical entropies of pseudo-Anosov mapping classes}

Finally, as the direct consequence of \cref{thm:SS_pA}, we compute the algebraic entropies of some birational automorphisms and categorical entropies of some exact autoequivalences induced by pseudo-Anosov mapping classes.

Let $\phi$ be a pseudo-Anosov mapping class on a punctured surface.
Then, it induces birational maps
\begin{align*}
    \phi_a: \cA_\Sigma \to \cA_\Sigma, \quad \phi_x: \cX_\Sigma \to \cX_\Sigma.
\end{align*}
Then, we can associate them with the positive real numbers $h^\mathrm{alg}(\phi_a)$ and $h^\mathrm{alg}(\phi_x)$, called \emph{algbraic entropies} (\cite{BV99}).

On the other hand, $\phi^r$ has a sign-stable path $\gamma: (\tri, \ell) \to \phi^{-r}(\tri, \ell)$ on $\cC^+_{(\tri, \ell)}$.
Therefore, it induces an autoequivalence $F_{\phi^r} = F_{\phi^r}^{(\tri, \ell)}: \sD(\Gamma_{Q^{(\tri, \ell)}, W^{(\tri, \ell)}}) \to \sD(\Gamma_{Q^{(\tri, \ell)}, W^{(\tri, \ell)}})$ of the derived category $\sD(\Gamma_{Q^{(\tri, \ell)}, W^{(\tri, \ell)}})$ of the Ginzburg dg algebra $\Gamma_{Q^{(\tri, \ell)}, W^{(\tri, \ell)}}$ of the quiver with potential of $(\tri, \ell)$.
We note that the existence of the induced functor $F_{\phi^r}$ depends on the sign-stability, in contrast to the birational maps $\phi_a$ and $\phi_x$.
It restricts to the subcategories
\begin{align*}
    F_{\phi^r}|_{\Dfd}: \Dfd^{(\tri, \ell)} \to \Dfd^{(\tri, \ell)}, \quad F_{\phi^r}|_\per: \per^{(\tri, \ell)} \to \per^{(\tri, \ell)}.
\end{align*}
Here, $\Dfd^{(\tri, \ell)}$ and $\per^{(\tri, \ell)}$ are called finite-dimensional derived category and perfect derived category, respectively.
Then, we can associate them with the positive real numbers $h^\mathrm{cat}_T(F_{\phi^r}|_{\Dfd})$ and $h^\mathrm{cat}_T(F_{\phi^r}|_{\per})$ for $T \in \bR$, called \emph{categorical entropies} (\cite{DHKK14}).
See \cite{Kan21} for more detail.

\begin{cor}\label{cor:entropy}
We have
\begin{align*}
    h^\mathrm{alg}(\phi_a) = h^\mathrm{alg}(\phi_x) =  h^\mathrm{cat}_T(F_{\phi^r}|_{\Dfd})/r = h^\mathrm{cat}_0(F_{\phi^r}|_{\per})/r = h^\mathrm{top}(\phi).
\end{align*}
Here, $h^\mathrm{top}(\phi)$ is the topological entropy of $\phi$, which is given by $\log \lambda_\phi$.
\end{cor}

\begin{proof}
It is clear that the sign stability of $\gamma$ on $\bR \cdot \cX_\Sigma(\bZ^\trop)$ induces that on $\Omega^\mathrm{can}_{(\tri, \ell)}$.
Moreover,
\begin{itemize}
    \item the palindromicity conjecture \cite[Conjecture 3.13]{IK21} holds for any representation paths of any mapping classes of a punctured surface \cite[Proposition 8.7]{IK20a} and
    \item the cluster stretch factor coincides with the stretch factor by applying the same argument in \cite[Section 7.2]{IK20a} for the cone $V(\tau^+_\phi)$.
\end{itemize}
Also, we recall the basic property of algebraic entropy:
\begin{itemize}
\item For a rational map $\varphi:(\bG_m)^n \to (\bG_m)^n$ and an integer $N \geq 0$, we have $\cE_{\varphi^N} = N \cE_\varphi$.
\end{itemize}
Therefore, we get the desired statement by \cite[Corollary 1.2]{IK21} and \cite[Corollary 1.2]{Kan21}.
\end{proof}

\appendix
\section{Cluster ensembles}\label{sec:appCA}
In this section, we briefly recall some basic notions around cluster varieties.
Here, we use simple settings so that we can apply them in the case of punctured surfaces.

\subsection{Mutation classes and their labeled exchange graphs}
Let us fix a finite index set $I$ and the field $\cF_A$ and $\cF_X$ of rational functions on $|I|$ indeterminants.

\begin{defi}
A \emph{labeled seed} is a triple $(B, \mathbf{A}, \mathbf{X})$ such that
\begin{itemize}
    \item $B = (b_{ij})_{i,j \in I}$ is a skew-symmetric integral matrix;
    \item $\mathbf{A} = (A_i)_{i \in I}$ (resp. $\mathbf{X} = (X_i)_{i \in I}$) is a transcendence basis of $\cF_A$ (resp. $\cF_X$).
\end{itemize}
The matrix $B$ is called the \emph{exchange matrix} and the variables $A_i$ (resp. $X_i$) are called $\cA$-variables (resp. $\cX$-variables).
\end{defi}

\begin{defi}[seed mutation]
Let $(B, \mathbf{A}, \mathbf{X})$ be a seed.
We define the new seed $\mu_k(B, \mathbf{A}, \mathbf{X}) = (B', \mathbf{A}', \mathbf{X}')$ for $k \in I$ by
\begin{align}
    b'_{ij} &:=
    \begin{cases}
    -b_{ij} & \mbox{if $i=k$ or $j=k$},\\
    b_{ij} + \dfrac{1}{2}(|b_{ik}|b_{kj} + b_{ik}|b_{kj}|) & \mbox{otherwise},
    \end{cases}\nonumber\\
    A'_i &:=
    \begin{cases}
    \dfrac{1}{A_k}\big( \prod_{j}A_j^{[b_{kj}]_+} + \prod_{j}A_j^{[-b_{kj}]_+} \big) & \mbox{if $i=k$},\\
    A_i & \mbox{if $i \neq k$},
    \end{cases}\label{eq:A_mut}\\
    X'_i &:=
    \begin{cases}
    X_{k}^{-1} & \mbox{if $i=k$},\\
    X_{i}\big(1+X_{k}^{-\sgn (b_{ik})}\big)^{-b_{ik}} & \mbox{if $i\neq k$}.
    \end{cases}\label{eq:X_mut}
\end{align}
Here, $B = (b_{ij})_{i,j \in I}$, $B'=(b'_{ij})_{i, j\in I}$, $\mathbf{A} = (A_i)_{i \in I}$, $\mathbf{A}' = (A'_i)_{i \in I}$, $\mathbf{X} = (X_i)_{i \in I}$ and $\mathbf{X}' = (X'_i)_{i \in I}$.
\end{defi}

It is easy to verify that the seed mutation $\mu_k$ is involutive.
For a permutation $\sigma \in \fS_I$, we similarly define $\sigma.(B, \mathbf{A}, \mathbf{X}) = (B', \mathbf{A}',\mathbf{X}')$ by 
\begin{align*}
    b'_{ij}=b_{\sigma^{-1}(i),\sigma^{-1}(j)},
    \quad A'_i = A_{\sigma^{-1}(i)}
    \quad \mbox{and} \quad X'_i= X_{\sigma^{-1}(i)}. 
\end{align*}

\begin{defi}
We say that two labeled seeds $(B, \mathbf{A}, \mathbf{X})$, $(B', \mathbf{A}', \mathbf{X}')$ are \emph{mutation-equivalent} if there is a finite composition of seed mutations and permutations that maps $(B, \mathbf{A}, \mathbf{X})$ to $(B', \mathbf{A}', \mathbf{X}')$.
A mutation-equivalence class $\bs$ of labeled seeds is simply called a \emph{mutation class}.
\end{defi}

Mutation classes of labeled seeds are the basic subjects in the research field of cluster algebra. 

\begin{defi}
The relations among the labeled seeds in a given mutation class $\bs$ can be encoded in the \emph{(labeled) exchange graph} $\bExch_\bs
$. It is a graph with vertices $v$ corresponding to the labeled seeds $\bs^{(v)}$ in $\bs$, together with labeled edges of the following two types:
\begin{itemize}
    \item labeled edges of the form $v \overbar{k} v'$ whenever the seeds $\bs^{(v)}$ and $\bs^{(v')}$ are related by the mutation $\mu_k$ for some $k \in I$;
    \item labeled edges of the form $v \overbarnear{\sigma} v'$ whenever the seeds $\bs^{(v)}$ and $\bs^{(v')}$ are related by a transposition $\sigma=(j\ k) \in \fS_I$.
\end{itemize}
\end{defi}

When no confusion can occur, we simply denote a vertex of the labeled exchange graph by $v \in \bExch_\bs$.
For each vertex $v \in \bExch_\bs$, we denote the corresponding labeled seed by $\bs^{(v)}=(B^{(v)}, \mathbf{A}^{(v)}, \mathbf{X}^{(v)})$, $B^{(v)}=(b_{ij}^{(v)})_{i,j \in I}$, $\mathbf{A}^{(v)}=(A_i^{(v)})_{i \in I}$ and $\mathbf{X}^{(v)}=(X_i^{(v)})_{i \in I}$. 

\subsection{Cluster ensemble}
For a mutation class $\bs$, one can associate the pair $(\cA_\bs, \cX_\bs)$ of schemes called the cluster $\cA$/$\cX$-variety and a rational map $p_\bs: \cA_\bs \to \cX_\bs$ called the ensemble map.
The datum $p_\bs: \cA_\bs \to \cX_\bs$ is also called \emph{cluster ensemble}.

Let $\bs$ be a mutation class.
For $v \in \bExch_\bs$, we consider a lattice $N^{(v)}$ with a fixed basis $\{e^{(v)}_i \}_{i \in I}$ and its dual lattice $M^{(v)}$ with the dual basis $\{f^{(v)}_i\}_{i \in I}$.
Let us define two algebraic tori $\cA_{(v)} := \Hom(M^{(v)}, \bG_m)$ and $\cX_{(v)} := \Hom(N^{(v)}, \bG_m)$ of rank $|I|$, where $\bG_m = \Spec \bZ[u, u^{-1}]$.
The characters $A^{(v)}_i := \mathrm{ch}_{f_i^{(v)}}: \cA_{(v)} \to \bG_m$ and $X^{(v)}_i := \mathrm{ch}_{e_i^{(v)}}: \cX_{(v)} \to \bG_m$ are define the coordinate system of $\cA_{(v)}$ and $\cX_{(v)}$ respectively, they called the \emph{cluster coordinates}.
The mutation rules \eqref{eq:A_mut} and \eqref{eq:X_mut} turn into brational maps $\mu_{x,k}: \cX_{(v)} \to \cX_{(v')}$ and $\mu_{a,k}: \cA_{(v)} \to \cA_{(v')}$ respectively, called the \emph{cluster transformations}.
Then, the cluster $\cA$- and $\cX$-varieties are the schemes defined as
\begin{align*}
    \cA_\bs := \bigcup_{v \in \bExch_\bs} \cA_{(v)}, \quad
    \cX_\bs := \bigcup_{v \in \bExch_\bs} \cX_{(v)}.
\end{align*}
Here, the tori $\{\cA_{(v)}\}_{v \in \bExch_\bs}$ (resp. $\{\cX_{(v)}\}_{v \in \bExch_\bs}$) are identified via the cluster transformations $\mu_{a, k}$ (resp. $\mu_{x, k}$) and permutations according to the graph $\bExch_\bs$.

\subsection{Tropicalized cluster ensemble}
Let $\bs$ be a mutation class.
Since the coordinate transformations of the cluster varieties $\cA_\bs, \cX_\bs$ can be expressed as subtraction-free forms, we can take the set $\cA_\bs(\bP), \cX_\bs(\bP)$ of semifield $\bP$ valued points.
Here, we give the concrete definition of it for the tropical semifield $\bR^\trop = (\bR, \min, +)$ as $\bP$.
The \emph{tropicalized cluster $\cA$- (resp. $\cX$-) variety} associated with $\bs$ is a PL-manifold $\cA_\bs(\bR^\trop)$ (resp. $\cX_\bs(\bR^\trop)$) homeomorphic to $\bR^I$, equipped with an atlas consisting of global charts
\begin{align*}
    \mathbf{a}^{(v)}: \cA_\bs(\bR^\trop) \to \cA_{(v)}(\bR^\trop) = N^{(v)} \otimes \bR \quad
    (\mbox{resp. } \mathbf{x}^{(v)}: \cX_\bs(\bR^\trop) \to \cX_{(v)}(\bR^\trop) =  M^{(v)} \otimes \bR)
\end{align*}
for $v \in \bExch_\bs$ such that the coordinate transformations among them are given by tropicalized cluster transformations and permutations.
Here, the tropicalized cluster transformations are given by
\begin{align*}
    \mu_{a,k}^* a'_i &= \begin{cases}
    \min\big\{ \sum_j [b_{kj}]_+ a_j,\, \sum_j [-b_{kj}]_+ a_j \big\} - a_k & \mbox{if } i=k,\\
    a_i & \mbox{if } i \neq k,
    \end{cases}\\
    \mu_{x,k}^* x'_i &= \begin{cases}
    -x_k & \mbox{if } i=k,\\
    x_i - b_{ik} \min\{ 0, -\sgn(b_{ik}) x_k \} & \mbox{if } i \neq k
    \end{cases}
\end{align*}
where $\mathbf{a}^{(v)} = (a_i)_{i \in I}$, $\mathbf{x}^{(v)} = (x_i)_{i \in I}$, $\mathbf{a}^{(v')} = (a'_i)_{i \in I}$, $\mathbf{x}^{(v')} = (x'_i)_{i \in I}$, $B^{(v)} = (b_{ij})_{i,j \in I}$ and $v \overbar{k} v'$ in $\bExch_\bs$.
We put the linear map $p_{(v)}: \cA_{(v)}(\bR^\trop) \to \cX_{(v)}(\bR^\trop)$ as its presentation matrix is $(B^{(v)})^\tr$.
Note that we have the following compatibility:
\begin{align*}
    \mu_{x,k} \circ p_{(v)} = p_{(v')} \circ \mu_{a,k}.
\end{align*}
Thus, uniquely exists the PL map
\begin{align*}
    p_\bs: \cA_\bs(\bR^\trop) \to \cX_\bs(\bR^\trop),
\end{align*}
called the \emph{(tropicalized) ensemble map},
satisfying $\mathbf{x}^{(v)} \circ p_\bs \circ (\mathbf{a}^{(v)})^{-1} = p_{(v)}$ for all $v \in \bExch_\bs$.

\subsection{Cluster modular group}\label{subsec:cluster_mod}
Let $\mathrm{Mat}_\bs$ denote the mutation class of exchange matrices underlying the mutation class $\bs$. Then we have a map 
\begin{align*}
    B^\bullet: V(\bExch_\bs) \to \mathrm{Mat}_\bs, \quad v \mapsto B^{(v)}.
\end{align*}

\begin{dfn}
The \emph{cluster modular group} $\Gamma_\bs \subset \mathrm{Aut}(\bExch_\bs)$ consists of graph automorphism $\phi$ which preserves the fibers of the map $B^\bullet$ and the labels on the edges (in particular, the horizontal/vertical properties).
An element of the cluster modular group is called a \emph{mutation loop}. 
\end{dfn}

For a mutation loop $\phi \in \Gamma_\bs$, the edge path from a vertex $v$ to $\phi^{-1}(v)$ in $\bExch_\bs$ is called \emph{representation path} of $\phi$.

The cluster modular group $\Gamma_\bs$ acts on the cluster ensemble $p_\bs: \cA_\bs \to \cX_\bs$ so that
\begin{align*}
    \phi_a^* \mathbf{A}^{(v)} = \mathbf{A}^{(\phi^{-1}(v))}
    \quad \mbox{and} \quad
    \phi_x^* \mathbf{X}^{(v)} = \mathbf{X}^{(\phi^{-1}(v))}
\end{align*}
for $\phi \in \Gamma_\bs$.
Here, $\phi_a: \cA_\bs \to \cA_\bs$ and $\phi_x: \cX_\bs \to \cX_\bs$ are automorphisms such that $p_\bs \circ \phi_a = \phi_x \circ p_\bs$.
We can describe explicitly the action by a representation path of a mutation loop.
Let $\gamma: v \to \phi^{-1}(v)$ be a representation path of $\phi$.
We denote $\mu_{a, \gamma}$ and $\mu_{x, \gamma}$ by the composition of the tropicalized cluster transformations along $\gamma$.
Then, we have the following commutative diagrams:
\[
\begin{tikzcd}
\cA_\bs \ar[rr, "\phi_a"] && \cA_\bs\\
\cA_{(v)} \ar[u, hookrightarrow] \ar[r, "\mu_{a, \gamma}"] & \cA_{(\phi^{-1}(v))} \ar[r, "\sim"] & \cA_{(v)}\,, \ar[u, hookrightarrow]
\end{tikzcd}
\qquad
\begin{tikzcd}
\cX_\bs \ar[rr, "\phi_x"] && \cX_\bs\\
\cX_{(v)} \ar[u, hookrightarrow] \ar[r, "\mu_{x, \gamma}"] & \cX_{(\phi^{-1}(v))} \ar[r, "\sim"] & \cX_{(v)}\,. \ar[u, hookrightarrow]
\end{tikzcd}
\]
Here, $\cA_{(\phi^{-1}(v))} \xrightarrow{\sim} \cA_{(v)}$ (resp. $\cX_{(\phi^{-1}(v))} \xrightarrow{\sim} \cX_{(v)}$) is the natural identification given by $A^{(\phi^{-1}(v))}_i \mapsto A^{(v)}_i$ (resp. $X^{(\phi^{-1}(v))}_i \mapsto X^{(v)}_i$).

\end{document}